\documentclass[12pt]{article}
\usepackage[font=small]{caption}
\usepackage{tikz}
\usetikzlibrary{patterns}
\usepackage{comment}
\usepackage{amsmath}     
\usepackage{amsthm}     
\usepackage{amssymb}     
\usepackage{amsfonts}
\usepackage[showonlyrefs]{mathtools}
\mathtoolsset{showonlyrefs}
\def\re{\operatorname{Re}}
\def\im{\operatorname{Im}}
\def\cl{\operatorname{cl}}
\def\meas{\operatorname{meas}}
\def\N{\mathbb{N}}
\def\R{\mathbb{R}}
\def\C{\mathbb{C}}
\def\D{\mathbb{D}}
\def\bC{\overline{\C}}
\def\Z{\mathbb{Z}}
\def\SS{\mathcal{S}}
\def\F{\mathfrak{F}}

\newtheorem{theorem}{Theorem}[section]
\newtheorem{corollary}{Corollary}[section]
\newtheorem{proposition}{Proposition}[section]
\newtheorem*{theorema}{Theorem A}
\newtheorem*{theoremb}{Theorem B}
\newtheorem*{theoremc}{Theorem C}
\newtheorem{lemma}{Lemma}[section]
\theoremstyle{remark}
\newtheorem{remark-no}{Remark}[section]
\newtheorem*{remark}{Remark}
\newtheorem{example}{Example}[section]
\newtheorem*{ack}{Acknowledgments}
\numberwithin{equation}{section}
\begin{document}  
\title{Quasiconformal surgery and linear differential equations}
\author{Walter Bergweiler and Alexandre Eremenko\thanks{Supported by NSF Grant
DMS-1361836.}}
\date{}
\maketitle

\begin{abstract}
We describe a new method of constructing transcendental entire functions
$A$ such that the differential equation $w''+Aw=0$ has two linearly independent solutions
with relatively few zeros.  In particular, we solve a problem of Bank and Laine by showing that there
exist entire functions $A$ of any prescribed order greater than $1/2$  such that the
differential equation has two linearly independent 
solutions whose zeros have finite exponent of convergence.
We show that partial results by Bank, Laine, Langley, Rossi and Shen related to
this problem are in fact best possible.
We also improve a result of Toda and show that the estimate obtained is best possible.
Our method is based on gluing solutions of the Schwarzian differential equation
$S(F)=2A$ for infinitely many coefficients~$A$.

\medskip


\medskip

\end{abstract}

\section{Introduction and main result}
We consider ordinary differential equations of the form
\begin{equation}\label{w''+Aw}
w''+Aw=0,
\end{equation}
where $A$ is an entire function.  All solutions $w$ are entire functions.
The Wronskian determinant $W(w_1,w_2)=w_1 w_2^\prime-w_1^\prime w_2$
of any two linearly independent solutions $w_1$ and $w_2$ is  a non-zero constant,
and a pair $(w_1,w_2)$  of solutions will be called \emph{normalized} if $W(w_1,w_2)=1$.

We recall that the \emph{order} of an entire function
$f$ is defined by
\begin{equation}\label{order}
\rho(f)=\limsup_{z\to\infty}\frac{\log\log|f(z)|}{\log|z|},
\end{equation}
and the \emph{exponent of convergence} of the zeros of $f$ by
$$\lambda(f)
=\inf \left\{\lambda>0\colon \sum_{\{z\neq 0\colon f(z)=0\}} m(z)|z|^{-\lambda}<\infty\right\}
=\limsup_{r\to\infty}\frac{\log n(r,0,f)}{\log r},$$
where $m(z)$ is the multiplicity of the zero~$z$,
and 
\[
n(r,0,f)=\sum_{|z|\leq r}m(z)
\]
denotes the number of zeros of $f$ in the disk $\{z\in\C\colon |z|\leq r\}$.
It is well-known~\cite[Chapter~I, \S 5]{Lev} that 
\begin{equation}\label{lr}
\lambda(f)\leq\rho(f)
\end{equation}
for every entire function~$f$.

When $A$ is a polynomial of degree~$d$, then all non-trivial solutions $w$
have order $(d+2)/2$. Some solutions can be free of zeros, but when $d>0$,
the exponent of convergence of zeros of the product $$E=w_1w_2$$ of
any two linearly independent solutions is equal to $(d+2)/2$. 

We refer to~\cite[Theorem~1]{BL2} for a proof of these results, most of which are classical.
In fact, much more precise estimates on the location of the zeros and the 
asymptotic behavior of $n(r,0,w)$ can be obtained from 
the asymptotic integration method; cf.~\cite[Section 4.6]{Hille1926},
\cite[Section~5.6]{Hille} and \cite[Theorem~6.1]{Sibuya}.

Some special equations of the form~\eqref{w''+Aw} with transcendental coefficient~$A$,
such as the Mathieu equation, have been intensively studied since the 19th century, but
one of the first systematic studies of the  general case of transcendental entire $A$
is due to Bank and Laine \cite{BL1,BL2}. When $A$ is transcendental,
then every non-trivial solution has infinite order. It is possible that two
linearly independent solutions are free of zeros; 
for example,
if $p$ is a polynomial, then
\begin{equation}\label{w12a}
w''-\frac{1}{4}\!\left(e^{2p}+(p')^2-2p''\right)w=0
\end{equation}
has solutions
\begin{equation}\label{w12}
w_{1,2}(z)=\exp\!\left(-\frac12 \left( p(z)\pm \int_0^ze^{p(t)}dt\right)\right) .
\end{equation}

However, two linearly independent solutions without zeros
can occur only when $A$ is constant, or $\rho(A)$
is a positive integer or~$\infty$.
More generally, we have the following result.

\begin{theorema}
Let $E$ be the product of two linearly independent solutions of the 
differential equation~\eqref{w''+Aw}, where $A$ is a transcendental entire function of 
finite order.
Then:
\begin{itemize}
\item[$(i)$] If $\rho(A)$ is not an integer, then
\begin{equation}\label{1}
\lambda(E)\geq\rho(A).
\end{equation}
\item[$(ii)$] If $\rho(A)\leq 1/2$, then $\lambda(E)=\infty$.
\end{itemize}
\end{theorema}

These results were proved by Bank and Laine \cite{BL1,BL2}, except for the case
$\rho(A)=1/2$ in $(ii)$ which is due to Rossi \cite{R} and Shen~\cite{Shen1985}.

Based on these results, Bank and Laine conjectured that 
whenever $\rho(A)$ is not an integer, then $\lambda(E)=\infty$.
This conjecture raised considerable interest; see the
surveys \cite{Gund,LT} and references therein.
Counterexamples were recently constructed by the present authors~\cite{BE}
who proved the following result.

\begin{theoremb}
There is a dense set of $\rho\geq 1$, such that
there exist $A$ with $\rho(A)=\lambda(E)=\rho$. Moreover, one solution
for this $A$ is free of zeros.
\end{theoremb}

This shows that the inequality~\eqref{1} is best possible when $\rho(A)\geq 1$.
Rossi~\cite{R} showed that if $1/2\leq\rho(A)<1$, then~\eqref{1} can be improved
to the inequality
\begin{equation}\label{ineq1}
\frac{1}{\rho(A)}+\frac{1}{\lambda(E)}\leq 2.
\end{equation}
Here and in the following $1/\lambda(E)=0$ is understood to mean that $\lambda(E)=\infty$.
Solving the inequality~\eqref{ineq1} for $\lambda(E)$ we obtain $\lambda(E)\geq \rho(A)/(2\rho(A)-1)$ 
for $1/2<\rho(A)<1$.

The main result of this paper shows that~\eqref{ineq1} is also best possible.
\begin{theorem} \label{thm1}
For every $\rho\in (1/2,1)$ there exists an entire function $A$ of order  $\rho$
such that the differential equation~\eqref{w''+Aw} has two linearly independent solutions whose
product $E$ satisfies
\begin{equation}\label{eq1}
\frac{1}{\rho(A)}+\frac{1}{\lambda(E)}= 2.
\end{equation}
Moreover, $\lambda(E)=\rho(E)$ and one of the solutions of \eqref{w''+Aw} is free of zeros.
\end{theorem}
We note that our method also allows to construct coefficients $A$ of preassigned order
$\rho(A)\in[1,\infty)$ such that~\eqref{w''+Aw} has two linearly independent solutions whose
product has finite order, with one solution having no zeros; see Corollary~\ref{cor2} below.

The main idea in the proofs of this and subsequent results is to glue functions associated
to the differential equation~\eqref{w''+Aw} for different coefficients~$A$.
Unlike in~\cite{BE}, where two functions are glued, it is now required
to glue infinitely many functions, which creates substantial additional
difficulties;  see section~\ref{background} for a detailed description of the method.

An entire function $E$ is called a {\em Bank--Laine function} if
$E(z)=0$ implies that $E'(z)\in\{-1,1\}$. We call a Bank--Laine function
{\em special} if $E(z)=0$ implies that $E'(z)=1$. It is known~\cite[Proposition 6.4]{Laine1993}
that the product of two normalized solutions of~\eqref{w''+Aw}
is a Bank--Laine function and all Bank--Laine functions arise in this way. If $w_1$ has no
zeros, the corresponding Bank--Laine function is special.
It follows from~\eqref{ineq1} and equation~\eqref{BL} below
that $\rho(E)\geq 1$ for every transcendental Bank--Laine function~$E$; see~\cite[Theorem~1]{Shen1987}.

Many authors studied Bank--Laine functions (see the surveys \cite{Gund,LT}),
but it was open until~\cite{BE} whether there exist Bank--Laine functions 
of non-integer order -- except for those corresponding 
to a polynomial coefficient $A$ of odd degree, when the
order of $E$ is half of an integer.
Theorem~\ref{thm1} gives a complete answer to the question asked in~\cite{BL1},
what the possible orders of Bank--Laine functions are. In fact,
prescribing $\rho(A)\in (1/2,1)$ in~\eqref{eq1} is equivalent to 
prescribing $\rho(E)\in (1,\infty)$ in~\eqref{eq1}.
Since the exponential function is a special Bank--Laine function of order~$1$,
we obtain the following result.
\begin{corollary}\label{cor1}
For every $\rho\in [1,\infty)$ there exists a special Bank--Laine function of order $\rho$.
\end{corollary}
A major difference between the functions $A$ and $E$ constructed in the proof of Theorem~B 
in~\cite{BE} and the functions constructed in the proof of 
Theorem~\ref{thm1} is that the 
functions in~\cite{BE} have ``spiraling'' behavior;
that is, we have $E(\gamma_1(t))\to 0$ and $E(\gamma_2(t))\to\infty$
as $t\to\infty$ on certain logarithmic spirals $\gamma_1$ and~$\gamma_2$.

In contrast, 
one can show that the functions $A$ and $E$ constructed in the proof of Theorem~\ref{thm1} 
have completely regular growth in the sense of
Levin and Pfluger; see~\cite[Chapter~3]{Lev}.
More precisely, we have the following result.
\begin{theorem} \label{thm1a}
The functions $A$ and $E$ in Theorem~\ref{thm1} can be chosen so that
\begin{equation} \label{asyAE1}
\frac12 \log |A(re^{i\theta})|
\sim \log \frac{1}{|E(re^{i\theta})|}
\sim r^\rho \cos(\rho \theta) 
\quad \text{for } |\theta|\leq (1-\varepsilon) \frac{\pi}{2\rho}
\end{equation}
while, with $\sigma=\rho/(2\rho-1)$,
\begin{equation} \label{asyAE2}
\log |E(-re^{i\theta})|
\sim r^\sigma \cos(\sigma \theta) 
\quad \text{for } |\theta|\leq (1-\varepsilon) \frac{\pi}{2\sigma}
\end{equation}
and
\begin{equation} \label{asyAE3}
 |A(-re^{i\theta})|
\sim  \frac{\sigma^2}{4}r^{2\sigma-2}
\quad \text{for } |\theta|\leq (1-\varepsilon) \frac{\pi}{2\sigma}
\end{equation}
uniformly as $r\to\infty$, for any $\varepsilon\in (0,1)$.
\end{theorem}
Since the set of rays of completely regular growth is closed~\cite[Chapter~3, Theorem~1]{Lev},
it follows that functions $A$ and $E$ satisfying the
conclusions of Theorem~\ref{thm1a} have indeed completely regular growth.

In \cite{BankLaineLangley1986} and various subsequent papers 
the differential equation~\eqref{w''+Aw} is studied 
under suitable hypotheses on the asymptotic behavior of the coefficient~$A$.
In particular, the following result was obtained in~\cite[Theorem~1]{BankLaineLangley1986}.
\begin{theoremc}
Let $A$  be a transcendental entire function of finite order with the following
property: there exists a subset $H$ of $\R$ of measure zero such that for 
each $\theta\in\R\backslash H$, either
\begin{itemize}
\item[$(i)$] $r^{-N}|A(re^{i\theta})|\to \infty$ as $r\to\infty$, for each $N>0$, or
\item[$(ii)$] $\int_0^\infty r|A(re^{i\theta})|dr<\infty$, or
\item[$(iii)$] there exists positive real numbers $K$ and~$b$, and a non-negative real 
number~$n$, all possibly depending on $\theta$, such that $(n+2)/2<\rho(A)$ and 
$|A(re^{i\theta})|\leq Kr^n$ for all $r\geq b$. 
\end{itemize}
Let $E$ be the product of two linearly independent solutions
of~\eqref{w''+Aw}. Then $\lambda(E)=\infty$.
\end{theoremc}
We note that the condition $(n+2)/2<\rho(A)$ in $(iii)$ is sharp by the example~\eqref{w12a}
with the solutions~\eqref{w12}.

Since $n\geq 0$ we see that $(iii)$ can be satisfied only if $\rho(A)>1$.
However, the proof in~\cite{BankLaineLangley1986} shows that
in $(iii)$ one may replace $\rho(A)>(n+2)/2$ by $\rho(E)>(n+2)/2$.
Using~\eqref{ineq1} and noting that
$\rho(E)\geq\lambda(E)$ we thus also obtain a result for $1/2<\rho(A)<1$.
A short computation yields the following.
\begin{proposition}\label{prop1}
Let $A$ be a entire function satisfying $1/2<\rho(A)<1$.
Then the conclusion of Theorem~C holds when the condition $(n+2)/2<\rho(A)$
in $(iii)$ is replaced by $n<\rho(A)/(2\rho(A)-1)$.
\end{proposition}
Note that the latter condition is equivalent to 
\[
\frac{n+2}{2}<1+\frac{\rho(A)}{4\rho(A)-2},
\]
so for $\rho(A)<1$ this is indeed a weaker condition.
It follows from \eqref{asyAE3} that this modified condition is best possible.

We denote the lower order of an entire function~$f$, which is
defined by taking the lower limit in~\eqref{order}, by $\mu(f)$.
Huang~\cite{Huang1991} showed that Theorem~A, part (ii), and~\eqref{ineq1}
hold with $\rho(A)$ replaced by~$\mu(A)$.
We note that for the function $A$ constructed in the proof of Theorem~\ref{thm1} we have
$\mu(A)=\rho(A)$.

Toda~\cite{Toda1993} showed that Theorem~A and~\eqref{ineq1}, and in fact their
refinements obtained by Huang, can be strengthened
if the set $\{z\in\C\colon |A(z)|> K\}$ has more than one component for some $K>0$.
Let $N$ be the number of such components.
Then $N/2\leq \mu(A)$ by the Denjoy--Carleman--Ahlfors Theorem;
see \cite[Chapter~5, Theorem 1.2]{GO} or \cite[Chapter~XI, \S4]{NevanlinnaEAF}.
Toda~\cite[Theorem~3]{Toda1993} proved that 
if $\mu(A)= N/2$, then $\lambda(E)=\infty$.  Moreover, for $N/2<\mu(A)< N$
the inequality~\eqref{ineq1} can be improved to
\begin{equation}\label{ineq2}
\frac{N}{\mu(A)}+\frac{1}{\lambda(E)}\leq 2.
\end{equation}
Toda actually showed that these results hold for the number $N$ of unbounded components
of the set $\{z\in\C\colon |A(z)|> K|z|^p\}$ if $p>0$ and $K>0$.
Note that for 
$f(z)=z^{N}+\exp(z^{N})$ the set $\{z\in\C\colon |f(z)|> K\}$  is connected for all
$K>0$  whilst  $\{z\in\C\colon |f(z)|> 4|z|^{N+1}\}$ has $N$ unbounded components.

We sharpen Toda's inequality and show that the result obtained is best possible.
\begin{theorem} \label{thm2}
Let $A$ be an entire function, $N\geq 2$, $p>0,$ and $K>0$.
Suppose that  the set $\{z\in\C\colon |A(z)|> K|z|^p\}$ has $N$ unbounded components 
and let $E$ be the product of two linearly independent solutions of the 
differential equation~\eqref{w''+Aw}. Then $\mu(A)\geq N/2$ and if $\mu(A)< N$, then
\begin{equation}\label{ineq3}
\frac{N}{\mu(A)}+\frac{N}{\lambda(E)}\leq 2.
\end{equation}
\end{theorem}
We note that if $\mu(A)=N$, then we may have $\lambda(E)=0$ by the examples 
given in~\eqref{w12}.
\begin{theorem} \label{thm3}
Let $N\in\N$ and $\rho\in (N/2,N)$. Then there exists an entire function $A$ 
satisfying  $\mu(A)=\rho(A)=\rho$ such that
$\{z\in\C\colon |A(z)|> K|z|^p\}$ has $N$ unbounded components for $p=2N\rho/(2\rho-N)-2$
and large $K$ and
such that the differential equation~\eqref{w''+Aw} has two linearly independent solutions whose
product $E$ satisfies
\begin{equation}\label{eqN}
\frac{N}{\mu(A)}+\frac{N}{\lambda(E)}= 2.
\end{equation}
Moreover, $\lambda(E)=\rho(E)$ and one of the solutions of \eqref{w''+Aw} is free of zeros.
\end{theorem}
An immediate corollary is the following result.
\begin{corollary}\label{cor2}
For every $\rho\in (1/2,\infty)$ there exists an entire function $A$ satisfying $\mu(A)=\rho(A)=\rho$
for which the equation~\eqref{w''+Aw} has two linearly independent solutions $w_1$ and $w_2$
such that $w_1$ has no zeros and $\lambda(w_2)<\infty$.
\end{corollary}

\begin{ack}
We thank Jim Langley and Lasse Rempe--Gillen for helpful comments, and the referee for a large
number of corrections and suggestions, which improved the paper significantly.
\end{ack}

\section{Background and underlying ideas} \label{background}
\subsection{The Schwarzian derivative and linear differential equations}
To every differential equation~\eqref{w''+Aw} the associated Schwarzian differential equation
is given by 
\begin{equation}
\label{schwarz}
S(F):=
\left(\frac{F''}{F'}\right)'-\frac{1}{2}\left(\frac{F''}{F'}\right)^2
=
\frac{F'''}{F'}-\frac{3}{2}\left(\frac{F''}{F'}\right)^2=2A.
\end{equation}
The expression $S(F)$ is called the \emph{Schwarzian derivative} of~$F$.
The general solution of~\eqref{schwarz} is $F=w_2/w_1$ where $w_1$ and $w_2$
are linearly independent solutions of~\eqref{w''+Aw}. 
A normalized pair $(w_1,w_2)$ can be recovered from $F$ by the formulas
\[
w_1^2=\frac{1}{F'},\quad w_2^2=\frac{F^2}{F'},\quad\mbox{and}\quad E=w_1w_2=
\frac{F}{F'}.
\]
It follows that $F'=1/w_1^2$ is free of zeros,
and evidently all poles of $F$ are simple. So $F$ is a {\em locally univalent}
meromorphic function.  If $L$ is a linear-fractional transformation,
then $S(L\circ F)=S(F)$. On the set of all locally univalent meromorphic
functions we introduce an equivalence relation by saying that $F_1\sim F_2$ if
$F_1=L\circ F_2$ for some linear-fractional transformation~$L$.
Then the map $F\mapsto A=S(F)/2$ gives a bijection between the set
of equivalence classes of
locally univalent meromorphic functions and entire functions~$A$.

All these facts were known in the 19th century; see, e.g., Schwarz's collected
papers \cite[pp.\ 351--355]{Schwarz},
where he also discusses the work of Lagrange, Cayley, Riemann, Klein
and others in this context. For a modern exposition, see~\cite[Chapter~6]{Laine1993}
or~\cite[IV.2.2]{Ghys}.

We will need some facts about Bank--Laine functions.
First of all, we note that the Schwarzian $S(F)$
can be factored as $S(F)=B(F/F')/2$,
where 
\begin{equation}\label{Boperator}
B(E):=-2\frac{E''}{E}+\left(\frac{E'}{E}\right)^2-\frac{1}{E^2}.
\end{equation}
So the general solution of the differential equation $B(E)=4A$,
that is, of the equation
\begin{equation}\label{BL}
4A =-2\frac{E''}{E}+\left(\frac{E'}{E}\right)^2-\frac{1}{E^2},
\end{equation}
is a product of a normalized pair of solutions of~\eqref{w''+Aw}.
In particular, the general solution $E$ is entire when $A$ is entire,
which implies that the general solution
is free from movable singularities. 
The second order equations that are
linear with respect to the second derivative were classified by 
Painlev\'e and Gambier; see, e.g., Ince's book~\cite[Chapter XIV]{Ince}
for this classification.

The fact that the general solution of~\eqref{BL} is the product of a normalized pair
of solutions of~\eqref{w''+Aw}
seems to be due to Hermite (see~\cite{Hermite} or~\cite[p.~572]{WhittakerWatson}) and
can be found already in Julia's collection of exercises~\cite[Probl\`eme no.~33, pp.~193--201]{Julia}
and Kamke's reference book~\cite[Entry 6.139]{Kamke}.
However, the importance of this equation for the asymptotic study
of $A$ and $E$ was shown for the first time by Bank and Laine in~\cite{BL1,BL2}. 

We have seen that for every locally univalent function~$F$, the quotient
$F/F'$ is a Bank-Laine function, and all Bank--Laine
functions arise in this way. Zeros of a Bank--Laine function $E$ are
zeros and poles of~$F$, and $E$ is special if and only if $F$ is entire.

We keep the following permanent notation:
$(w_1,w_2)$ is a normalized pair of solutions
of~\eqref{w''+Aw}, the quotient $F=w_2/w_1$ is a solution of $S(F)=2A$, and
$E=w_1w_2=F/F'$ is a solution of~\eqref{BL}.

\subsection{Description of the method}
In this section we describe the underlying ideas
of our construction. The formal proofs of our results in sections \ref{proof1}--\ref{proof3}
are independent of this section.
On the other hand, an expert in conformal gluing
and the inverse problem of Nevanlinna theory may want to read only
this section.

Our approach consists of constructing $F$ by a geometric method (gluing),
and then recovering the needed asymptotic properties of $A$ and~$E$.
This idea was used for the first time by Nevanlinna \cite{Ne} and
Ahlfors \cite{A}. Nevanlinna solved a special case of the inverse
problem  of value distribution theory by using the asymptotic theory
of the differential equation~\eqref{w''+Aw} with polynomial coefficients.
Ahlfors showed that the same results
can be obtained by geometric methods, without appealing to the
differential equation~\eqref{w''+Aw}. These two papers initiated a long line
of research which culminated in the solution of the inverse problem
for functions of finite order with finitely many deficiencies
by Goldberg \cite{Gold,GO}. Further development of these ideas
led to the complete solution of the inverse problem
by Drasin~\cite{D}. The connection with differential equations
was not used in this research after \cite{Ne}.

Here we use this connection in the opposite direction to \cite{Ne}: we construct
a locally univalent map geometrically, use a form of the uniformization theorem
to obtain~$F$, and then derive asymptotic properties of $A$ and~$E$.

There is a subclass of locally univalent maps with especially
simple properties: there exists a finite set $X$ such that
\begin{equation}\label{S}
F\colon\C\backslash F^{-1}(X)\to\bC\backslash X
\end{equation}
is a covering.
The set of meromorphic functions with this property (not necessarily locally
univalent) is called the \emph{Speiser class} and denoted by $\SS$; it plays an important role in
holomorphic dynamics \cite{Bergweiler93,EL,Gol86} and in the general theory of
entire and meromorphic functions \cite{CER,Remp,GO}.
We only use the case when $X=\{0,1,\infty\}$.
To visualize functions of class $\SS$ one employs a classical tool, line complexes; 
see~\cite[Section 7.4]{GO} and~\cite[\S 11.2]{NevanlinnaEAF}.
Consider a graph $\Gamma_0$ embedded in the sphere $\bC$ with two vertices, $\times=i$
and $\circ=-i$, and three edges connecting the two vertices
and intersecting the real line exactly once, in the intervals
$(-\infty,0)$, $(0,1)$ and $(1,\infty)$, respectively. It is convenient
to make this graph symmetric with respect to the real line.

The preimage $\Gamma=F^{-1}(\Gamma_0)$ is called the \emph{line complex}.
Its vertices are labeled by $\times$ and $\circ$, and faces are labeled by $0,1,\infty$,
according to their images.
Two line complexes are equivalent if there is a homeomorphism of the plane
sending one to another, respecting the labels of vertices and faces.
In figures like
Figure~\ref{linecomplex} we draw one representative of the equivalence class,
usually not the true preimage of $\Gamma_0$ under~$F$.
The function $F$ is defined by the equivalence class of its line complex up to
an affine change of the independent variable.

As explained in the references given above, $\Gamma$ is a bi-partite connected embedded graph,
in which every vertex has degree~$3$. If $F$ is locally univalent, then
the faces (that is, the components of $\C\backslash\Gamma$) can be either
$2$-gons or $\infty$-gons. Here $2$-gons correspond to $a$-points of $F$ with
$a\in\{0,1,\infty\}$ and $\infty$-gons to logarithmic singularities of $F^{-1}$. 
Every bipartite connected graph with vertices of degree~$3$ embedded to the plane is
a line complex of some function $F$ of class $\SS$ with $X=\{0,1,\infty\}$, meromorphic
either in the unit disk or in the plane.

Our functions $F$ will correspond to the class of line complexes shown
in Figure~\ref{linecomplex}. They are parametrized by doubly infinite sequences of
non-negative integers $(\ell_k)_{k\in\Z}$,
showing the numbers of
$-\times\!\!=\!\!\circ-$ links on the vertical pieces between
the infinite horizontal branches. The faces are labeled by their images. 
Zeros of $F$ correspond to $2$-gons on the vertical part of the boundaries
of faces labeled $\infty$ and poles to the $2$-gons on the vertical parts
of the boundaries of faces labeled~$0$. So our function $F$ is entire
if and only if $\ell_k=0$ for all odd~$k$.
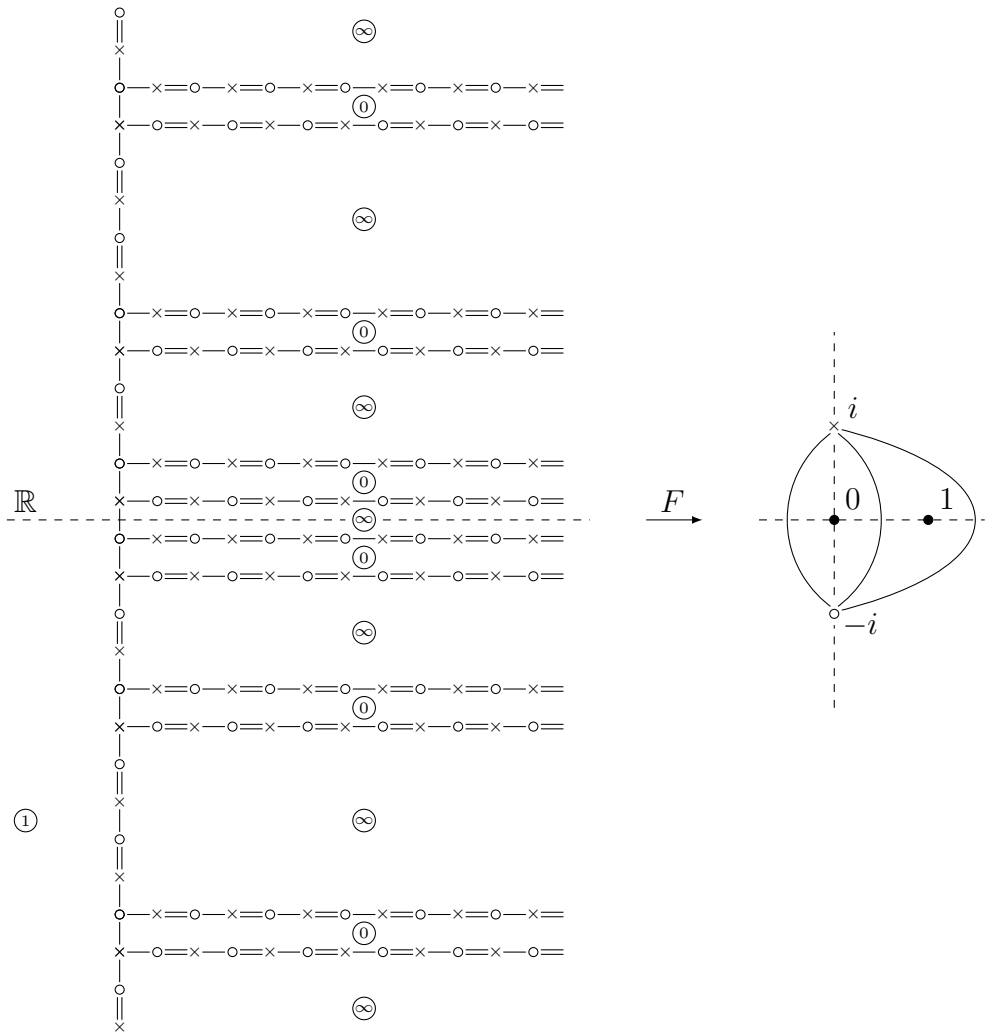
\begin{figure}[!htb]
\centering
\captionsetup{width=.85\textwidth}
\begin{tikzpicture}[scale=0.5,>=latex](-3,-14)(23,14)
% right half of picture
\draw[-,dashed](17,0)->(23,0);
\draw[-,dashed](19,-5)->(19,-2.8);
\draw[-,dashed](19,-2)->(19,2);
\draw[-,dashed](19,2.8)->(19,5);
\draw[->](14,0)->(15.5,0);
        \node at (14.7,0.5) {$F$};
\filldraw (21.5,0) circle (0.12) node[above right]{1};
\filldraw (19,0) circle (0.12) node[above right]{0};
\draw (19,-2.5) circle (0.12);
% node[below right]{$-i$};
\node at (19,2.5) {\tiny $\times$};
\node at (19.5,3) {$i$};
\node at (19.7,-2.8) {$-i$};
\draw [black] plot  [smooth,tension=1.2] coordinates {(19.1,-2.3) (20.25,0) (19.1,2.3) };
\draw [black] plot  [smooth,tension=1.2] coordinates {(19.2,-2.4) (22.75,0) (19.2,2.4) };
\draw [black] plot  [smooth,tension=1.2] coordinates {(18.9,-2.3) (17.75,0) (18.9,2.3) };
% vertical nodes
\foreach \y in {-12.5,-10.5,...,13.5}
        \draw (0,\y) circle (0.12);
\foreach \y in {-13.5,-11.5,...,12.5}
        \node at (0,\y) {\tiny $\times$};
\foreach \y in {-12.5,-11.5,-10.5,-8.5,-6.5,-5.5,-4.5,-2.5,-1.5,-0.5,0.5,1.5,3.5,4.5,5.5,7.5,9.5,10.5,11.5}
        \draw[-] (0,\y+0.2)->(0,\y+0.8);
\foreach \y in {-13.5,-9.5,-7.5,-3.5,2.5,6.5,8.5,12.5}
{
        \draw[-] (-0.06,\y+0.2)->(-0.06,\y+0.8);
        \draw[-] (0.06,\y+0.2)->(0.06,\y+0.8);
}
% horizontal nodes
\foreach \x in {0,2,...,10}
\foreach \y in {-10.5,-4.5,-0.5,1.5,5.5,11.5}
{
        \draw (\x,\y) circle (0.12);
        \node at (\x+1,\y) {\tiny $\times$};
        \draw[-] (\x+0.2,\y)->(\x+0.8,\y);
        \draw[-] (\x+1.2,\y+0.06)->(\x+1.8,\y+0.06);
        \draw[-] (\x+1.2,\y-0.06)->(\x+1.8,\y-0.06);
}
\foreach \x in {0,2,...,10}
\foreach \y in {-11.5,-5.5,-1.5,0.5,4.5,10.5}
{
        \draw (\x+1,\y) circle (0.12);
        \node at (\x,\y) {\tiny $\times$};
        \draw[-] (\x+0.2,\y)->(\x+0.8,\y);
        \draw[-] (\x+1.2,\y+0.06)->(\x+1.8,\y+0.06);
        \draw[-] (\x+1.2,\y-0.06)->(\x+1.8,\y-0.06);
}
% marking of faces
\foreach \y in {-11,-5,-1,1,5,11}
{
        \draw (6.5,\y) circle (0.30);
        \node at (6.5,\y) {\tiny $0$};
}
\foreach \y in {-13,-8,-3,0,3,8,13}
{
        \draw (6.5,\y) circle (0.30);
        \node at (6.5,\y) {\tiny $\infty$};
}
        \draw (-2.5,-8) circle (0.30);
        \node at (-2.5,-8) {\tiny $1$};
\draw[-,dashed](-3,0)->(6,0);
\draw[-,dashed](7,0)->(12.5,0);
\node at (-2.5,0.5)[]{$\R$};
\end{tikzpicture}
\caption{Sketch of the line complex corresponding to $\ell_0=\ell_{\pm1}=0$,
$\ell_{\pm 2}=1$, $\ell_{\pm 3}=0$, $\ell_{\pm 4}=2$, and $\ell_{\pm5}=0$.
The encircled labels indicate to which logarithmic
singularities the faces correspond. The graph $\Gamma_0$ is shown on the right.}
\label{linecomplex}
\end{figure}

\begin{example}\label{ex1}
$F(z)=e^{e^z}$ corresponds to $\ell_k=0$, $-\infty<k<\infty$.
\end{example}

\begin{example}\label{ex2}
$F(z)=P(e^z)e^{e^z}$, where $P$ is a polynomial of degree~$d$, such that $F$ is locally univalent,
corresponds to the line complex with $\ell_k=d$ when $k$ is even and $\ell_k=0$
when $k$ is odd. To see what this polynomial $P$ might be, we differentiate to obtain
$F'(z)=(P'(e^z)+P(e^z))e^ze^{e^z}$. This is free of zeros if and only if
$P'(w)+P(w)=cw^{d}$ for some $c\in\C\backslash\{0\}$. This easily implies that, up to a constant factor,
$$P(w)=\sum_{j=0}^{d} (-1)^j \frac{w^j}{j!},$$
a partial sum of $e^{-w}$.

We remark that a simple computation using~\eqref{schwarz} shows that the 
coefficient $A$ corresponding to $F$ is given by
\[
A(z)=-\frac14 e^{2z} -\frac{d}{2} e^z-\frac{(d+1)^2}{4} .
\]
The case that $A$ in \eqref{w''+Aw} has the form $A(z)=R(e^z)$ with a rational function $R$
has been thoroughly studied (see, e.g.,~\cite{BankLaine1983,ChiangIsmail}) but we shall 
not use these results.
\end{example}

\begin{example}\label{ex3}
$F(z)=R(e^z)e^{e^z}$, where $R$ is a rational function.  For $F$ to be locally univalent we need
$R'(w)+R(w)=c w^p/Q^2(w)$ where $Q$ is a polynomial with distinct roots and $c\in\C\backslash\{0\}$.
Then $R=P/Q$ where 
\[
P'(w)Q(w)-P(w)Q'(w)+P(w)Q(w)=cw^p.
\] Assuming $\deg P=m$, 
$\deg Q=n$ and $P(0)=Q(0)=1$ we conclude that $p=m+n$, and $P/Q$ is the
$(m,n)$-Pad\'e approximant to $e^{-z}$; see~\cite{BakerGravesMorris} for a thorough treatment of 
Pad\'e approximation, with a discussion of Pad\'e approximants to the exponential function
in~\cite[Section~1.2]{BakerGravesMorris}.
For this function $F$ we have $\ell_k=n$ when $k$ is even and $\ell_k=m$ when $k$ is odd.
\end{example}

In all these examples $F$ is periodic and the Bank--Laine
function $E=F/F'$ is of order~$1$. In Examples~\ref{ex1} and~\ref{ex2},
the Bank--Laine function is special.

To obtain different orders of~$E$,
we consider functions $F$ whose line complex has a non-periodic sequence $(\ell_k)$.
We restrict ourselves to entire functions $F$ and special
Bank--Laine functions~$E$, with $\ell_{2k+1}=0$,
\[
\ell_0=m_1+m_{-1},\quad \ell_{2k}=m_k+m_{k+1},\quad
\ell_{-2k}=m_{-k}+m_{-k-1},\quad k>0,
\]
where $(m_k)_{k\in\Z\backslash\{ 0\}}$ is a sequence of non-negative integers.

The construction of a function $F$ corresponding to such a line complex can be
visualized as follows.

Consider the strips 
$$\Pi_k=\{ x+iy\colon 2\pi(k-1)<y<2\pi k\}.$$
Let $\F_k$ be the bordered Riemann surface spread over the plane, which is the image
of this strip under the function
$$g_{m_k}(z)=P_{m_k}(e^z)e^{e^z},$$
where $P_{m_k}(w)$ is the partial sum of the Taylor series of $e^{-w}$
of degree $2m_k$, as in Example~\ref{ex2}.

All these Riemann surfaces have two boundary components which project
onto the ray $(1,+\infty)\subset\R$. We glue them together along these rays,
in the same order as the strips $\Pi_k$ are glued together in the plane.
The resulting Riemann surface $\F$ is open and simply connected.
Our function $F$ in Theorems~\ref{thm1} and~\ref{thm1a} is the conformal map from $\C$ to~$\F$.

Of course, the uniformization theorem by itself is not sufficient: one has
to know  that $\F$ is conformally equivalent to the plane.
Moreover, we need to know something about the asymptotic behavior
of $F$ to make conclusions about the order of $E=F/F'$ and $A=B(E)$.

So we do the following. Consider the piecewise analytic function $g$
defined by $g(z)=g_{m_k}(z)$ for $z\in\Pi_k$.
It follows from equation~\eqref{1a1} below that $g_{m_k}\colon \R\to (1,+\infty)$
is an increasing homeomorphism. Thus
each boundary component
of the strip $\Pi_k$ is mapped by $g$ homeomorphically onto the ray $(1,+\infty)$. 
We shall study the homeomorphisms $\phi_k\colon \R\to\R$
defined by $g_{m_{k+1}}(x) =g_{m_k}(\phi_k(x))$. We will
see that these homeomorphisms are close to the identity
on the positive real axis
and thus it is easy to glue the restrictions of these
functions to the half-strips $\{z\in\Pi_k\colon \re z>0\}$
to obtain a quasiregular map $U$ defined in the right
half-plane.  (For technical reasons we will actually use
the functions $u_{m_k}(z)=g_{m_{k}}(z+s_{m_k})$ instead
of $g_{m_{k}}(z)$, for certain constants $s_{m_k}$.)
The homeomorphisms $\phi_k$ are close to the linear
map $x\mapsto (2 m_{k+1}+1)x/(2m_k+1)$ on the negative
real axis. In the left half-plane we therefore consider
$v_{m_k}(z)=u_{m_k}(z/(2m_k+1))$ instead of $u_{m_k}$ and
find that it is easy to glue restrictions of these maps
to horizontal half-strips of width $2\pi(2m_k+1)$.
This way we obtain a quasiregular map $V$ defined in
the left half-plane. 

We then precompose the maps $U$ and
$V$ with appropriate powers and glue the resulting
functions to obtain a quasiregular map~$G$.
The dilatation $K_G$ of $G$ will satisfy the condition
\begin{equation}\label{TWB1}
\int_{|z|>1} \frac{K_G(z)-1}{x^2+y^2}dx\, dy<\infty.
\end{equation}
Then the existence theorem for quasiconformal mappings~\cite[\S V.1]{LV},
together with the theorem of  Teich\-m\"ul\-ler--Wit\-tich--Be\-linskii~\cite[\S V.6]{LV},
will guarantee the
existence of a homeomorphism $\tau\colon\C\to\C$ with 
\begin{equation}\label{TWB2}
\tau(z)\sim z\quad \text{as } z\to\infty,
\end{equation}
such that $G=F\circ\tau$ for some entire function~$F$.
The property~\eqref{TWB2} and explicit estimates for the $\phi_k$
will give sufficient control of the asymptotic
behavior of $F$ to prove the conclusions
of Theorems~\ref{thm1} and~\ref{thm1a}.

To achieve \eqref{TWB1} one needs a good control of the homeomorphisms $\phi_k$
as $x\to\infty$ and as $k\to\infty$. This is achieved by using
Szeg\H{o}-type asymptotics of the partial sums of the exponential which we prove in
section~3.

To prove Theorem~\ref{thm3}, we need a locally univalent function $F$ whose asymptotic
behavior is similar to $F_0(z^N)$, where $F_0$ is the locally univalent function
constructed in Theorem~\ref{thm1}. Of course, $F_0(z^N)$ has a critical point at $0$ 
and thus is not locally univalent.
So the idea is to proceed as follows. First we prepare a function $F_1$
which is similar to $F_0$, but maps the real line onto $(0,1)$ reversing the
orientation (so it is decreasing on the real line).
In fact, the construction would work with $F_1=1/F_0$, but then the resulting function $F$
would have poles so that the Bank--Laine function $E=F/F'$ would not be special.
Then we glue $F_0(z^N)$ and $F_1(z^N)$ in a suitable way.

To carry out the construction, we will actually work with the quasi\-regular maps 
$G_0$ and $G_1$ arising in the construction of $F_0$ and $F_1$, instead
of $F_0$ and $F_1$ themselves. 
Then we consider regions $C_j$, for $j=1,\dots,2N$, which are contained in
the sectors $\{ z\in\C\colon \pi (j-1)/N<\arg z<\pi j/N\}$ and which are
asymptotically close to these sectors; cf.\ Figure~\ref{C_j}.
\begin{figure}[!htb]
\captionsetup{width=.85\textwidth}
\centering
\begin{tikzpicture}[scale=0.55,>=latex](-10,-10)(-10,10)
\clip (-10,-10) rectangle (10,10);
\foreach \r in {120,240}
{
\filldraw[gray!10, rotate=\r, samples=150, domain=0:28] plot ({0.0+0.57735*\x+10/(\x+0.3)},{\x});
\filldraw[gray!10, rotate=\r, samples=150, domain=0:28] plot ({0.0+0.57735*\x+10/(\x+0.3)},{-\x});
}
\filldraw[gray!10, rotate=0, samples=150, domain=0:28] plot ({0.0+0.57735*\x+10/(\x+0.3)},{\x});
\filldraw[gray!10, rotate=0, samples=150, domain=0:28] plot ({0.0+0.57735*\x+10/(\x+0.3)},{-\x});
\foreach \r in {0,120,240}
{
\draw[black, rotate=\r, samples=150, domain=0:28] plot ({0.0+0.57735*\x+10/(\x+0.3)},{\x});
\draw[black, rotate=\r, samples=150, domain=0:28] plot ({0.0+0.57735*\x+10/(\x+0.3)},{-\x});
}
\draw[-](-10,0)->(10,0);
\draw[-](0,-10)->(0,10);
\draw[-,dashed](-5.7735,-10)->(5.7735,10);
\draw[-,dashed](5.7735,-10)->(-5.7735,10);
        \filldraw[white] (8,0) circle (0.80);
        \draw (8,0) circle (0.60);
        \node at (8,0) {\normalsize $\infty$};
        \filldraw[white] (-8,0) circle (0.80);
        \draw (-8,0) circle (0.60);
        \node at (-8,0) {\normalsize $1$};
        \filldraw[white] (0.57735*8,8) circle (0.80);
        \draw (0.57735*8,8) circle (0.60);
        \node at (0.57735*8,8) {\normalsize $1$};
        \filldraw[white] (0.57735*8,-8) circle (0.80);
        \draw (0.57735*8,-8) circle (0.60);
        \node at (0.57735*8,-8) {\normalsize $1$};
        \filldraw[white] (-0.57735*8,8) circle (0.80);
        \draw (-0.57735*8,8) circle (0.60);
        \node at (-0.57735*8,8) {\normalsize $0$};
        \filldraw[white] (-0.57735*8,-8) circle (0.80);
        \draw (-0.57735*8,-8) circle (0.60);
        \node at (-0.57735*8,-8) {\normalsize $0$};
        \node at (8,5) {\normalsize $C_1$};
        \node at (8,-5) {\normalsize $C_6$};
        \node at (-8,5) {\normalsize $C_3$};
        \node at (-8,-5) {\normalsize $C_4$};
        \node at (1,8) {\normalsize $C_2$};
        \node at (1,-8) {\normalsize $C_5$};
\end{tikzpicture}
\caption{Sketch of the domains $C_j$ for $N=3$.}
\label{C_j}
\end{figure}

For $z\in C_j$ we define $G(z)=G_k(\varphi_j(z^N))$ with some quasiconformal map~$\varphi$,
and $k\in\{0,1\}$ depending on~$j$.
In fact, we will have $G(z)=G_k(z^N)$ in a large subdomain $D_j$ of $C_j$.
Actually, we will take $k=1$ for $j=1$ and $j=2N$ and $k=0$ otherwise.

By construction, $G$ will map $\partial C_j$ 
homeomorphically onto one of the intervals $(0,1)$ and $(1,+\infty)$ of the real line.
It remains to define  a locally univalent quasi\-regular function in the complement of
$\bigcup_{j=1}^{2N} C_j$ which has these boundary values. In particular, this map
will tend to one of the values~$0$, $1$ and $\infty$ in the strips between the~$C_j$.
These values are encircled in Figure~\ref{C_j}.
In fact, the map will tend to $\infty$ in only one strip, and in the remaining strips
the limit will alternate between $1$ and~$0$.

The question whether a locally homeomorphic extension of $F$ to the complement 
of $\bigcup_{j=1}^{2N} C_j$ with these boundary values exists 
is a purely topological problem. This is solved in section~\ref{extension-bounded}.
However, we also need the extension to be quasi\-regular, with dilatation
satisfying~\eqref{TWB1}.
This is achieved by choosing an appropriate shape of the $C_j$ near infinity
in section~\ref{interpolation-infty}.  Composing the quasiregular map $G$ obtained 
with a quasiconformal map $\tau$ satisfying~\eqref{TWB2} will give our entire function~$F$.

We conclude this section with some general remarks and references.
Functions $f\in\SS$ are determined by their line complexes labeled
by the singular values up to an affine change of the independent
variable. It is an important problem to draw conclusions about asymptotic
properties of $f$ from the line complex. First of all one has to be
able to determine the conformal type of the Riemann surface defined
by the line complex \cite{Volk}. But once the type is determined, one
wants to know the asymptotic characteristics like the order of growth,
deficiencies, etc. Teich\-m\"ul\-ler \cite{Teich}
stated the general problem as follows:
\medskip

{\em Gegeben sei eine einfach zusammen\-h\"angende Riemannsche Fl\"a\-che
${\mathfrak{W}}$ \"uber der $w$-Kugel. Man kann sie bekannt\-lich
ein\-ein\-deutig und kon\-form auf den Ein\-heits\-kreis $|z|<1$, auf die punktierte
Ebene $z\neq\infty$ oder auf die volle $z$-Kugel abbilden,
so da\ss\ $w$ eine ein\-deutige Funktion von $z$ wird: $w=f(z)$. Die
Wert\-verteilung dieser ein\-deutigen Funktion ist zu unter\-suchen.

Dies ist ein Hauptproblem der modernen Funktionentheorie.}
\medskip

This problem has been intensively studied in connection
with the inverse problem of value distribution theory \cite{GO,Huk,Wit}.

In recent times there is a revival of interest in these questions, which is mainly stimulated
by questions of holomorphic dynamics.  In this connection,
we mention remarkable contributions by Bishop \cite{B1,B2,B3}.

\section{Proof of Theorem~\ref{thm1}}\label{proof1}
\subsection{Preliminary results}
\begin{lemma}\label{lemma1}
Let $\gamma>1$. Then there exists a sequence $(n_k)$ of odd positive integers, with $n_1=1$,
such that the function $h\colon [0,\infty)\to [0,\infty)$ which satisfies $h(0)=0$ and which
is linear on the intervals $[2\pi(k-1),2\pi k]$ and has slope $n_k$ there satisfies 
\begin{equation}\label{0a}
h(x)=x^{\gamma}+O(x^{\gamma-2})+O(1)
\end{equation}
as $x\to\infty$. 

Moreover, the function $g\colon [0,\infty)\to [0,\infty)$ defined by 
$h(g(x))=x^\gamma$ satisfies
\begin{equation}\label{0b}
g(x)=
x+O(x^{-1})+O(x^{1-\gamma})
\end{equation}
and
\begin{equation}\label{0c}
g'(x)=1+O(x^{-1})+O(x^{1-\gamma})
\end{equation}
as $x\to\infty$, where $g'$ denotes either the left or right derivative of~$g$.

Finally,
\begin{equation}\label{0d}
n_k=\gamma(2\pi k)^{\gamma-1}
+O(k^{\gamma-2})+O(1)
\end{equation}
as $k\to\infty$.
\end{lemma}
\begin{proof}
We set $h(0)=0$ and choose $k_0\in\N$ so that
\begin{equation}\label{zero}
(2\pi k)^\gamma-(2\pi(k-1))^\gamma\geq 4\pi\quad\text{for } k>k_0.
\end{equation}
Such a $k_0$ exists because $\gamma>1$.

For $k\leq k_0$ we set $n_k=1$. Suppose that $k>k_0$, and $h(2\pi(k-1))$ is already defined.
Then we define
$h(2\pi k):=2\pi p_k$, where $p_k$ is a positive integer of opposite parity to $h(2\pi(k-1))/(2\pi)$
minimizing $|(2\pi k)^\gamma-2\pi p_k|$. There are at most two such~$p_k$, and
when there are two, we choose the larger one. Then we interpolate $h$ linearly
between $2\pi(k-1)$ and $2\pi k$.  Evidently, with this definition, 
\begin{equation}
\label{one}
|h(2\pi k)-(2\pi k)^\gamma|\leq 2\pi.
\end{equation}
Next we show that $h$ is strictly increasing. Using~\eqref{zero} and~\eqref{one} we have 
$$h(2\pi k)\geq (2\pi k)^\gamma-2\pi>(2\pi(k-1))^\gamma+2\pi\geq h(2\pi(k-1)).$$
So $h$ is strictly increasing, and its slopes $n_k$ are positive odd integers.

To prove that the function $h$  
satisfies~\eqref{0a}, we note first that $h(2\pi k)=(2\pi k)^\gamma+O(1)$ by construction.
For $0\leq t\leq 1$ we have
\begin{align*}
& \;
(2\pi(k+t))^\gamma-t(2\pi(k+1))^\gamma-(1-t)(2\pi k)^\gamma
\\  = & \;
(2\pi k)^\gamma 
\left(1+\gamma t/k-t(1+\gamma/k)-1+t+O\!\left(k^{-2} 
\right)\right)=O(k^{\gamma-2})
\end{align*}
as $k\to\infty$, so the straight line connecting the points
$(2\pi k,(2\pi k)^\gamma)$ and $(2\pi(k+1),(2\pi(k+1))^\gamma)$
deviates from the graph 
of the function $x\mapsto x^\gamma$ between the points $k$ and $k+1$
by a term which is  $O(k^{\gamma-2})$.
This yields~\eqref{0a}.

To prove~\eqref{0b} we note that~\eqref{0a} implies that
$g(x)=x(1+\delta(x))$ where $\delta(x)\to 0$. Using~\eqref{0a} again we see that
\begin{align*}
x^\gamma&=h(g(x))
%\\ &
=x^\gamma(1+\delta(x))^\gamma +O(x^{\gamma-2})+O(1)
\\ &
=x^\gamma\left(1+ \gamma\delta(x)+O(\delta(x)^2)\right)  +O(x^{\gamma-2})+O(1)
\end{align*}
as $x\to\infty$.
This yields 
$$x^\gamma \delta(x)=O(x^{\gamma-2})+O(1),$$
from which~\eqref{0b} follows.

Similarly we see that 
$$h'(x)=\gamma x^{\gamma-1} +O(x^{\gamma-2})+O(1) $$
and thus
$$h'(g(x))=\gamma x^{\gamma-1} \left(1+\delta(x)\right)^{\gamma-1}+O(x^{\gamma-2})+O(1) 
=\gamma x^{\gamma-1} +O(x^{\gamma-2})+O(1).$$
Hence
\begin{equation}\label{g'x}
g'(x)=\frac{\gamma x^{\gamma-1}}{h'(g(x))}
=\frac{1}{1+O(x^{-1})+O(x^{1-\gamma})}
=1+O(x^{-1})+O(x^{1-\gamma}),
\end{equation}
which is~\eqref{0c}.

Finally, by construction we have
\[
n_k=\frac{h(2\pi k)-h(2\pi (k-1))}{2\pi}=
\frac{(2\pi k)^\gamma-(2\pi(k-1))^{\gamma}}{2\pi}+O(1),
\]
from which~\eqref{0d} easily follows.
\end{proof}

For $m\geq 0$ we will consider the Taylor polynomial
$$P_m(z)=\sum_{k=0}^{2m}(-1)^k \frac{z^k}{k!}$$
of $e^{-z}$ and the functions
$$h_m(z)=P_m(z)e^z$$
and
$$g_m(z)=h_m(e^z)=P_m(e^z)e^{e^z}.$$
Noting that 
\begin{equation}\label{1a1}
g_m'(z)=\frac{1}{(2m)!}\exp\!\left( e^z+(2m+1)z\right)\neq 0
\end{equation}
we see that $g_m$ is an increasing homeomorphism from $\R$ onto $(1,\infty)$.

We shall need some information about the asymptotic behavior of $h_m$ and~$g_m$.
\begin{lemma}\label{lemma1a}
Let $m\in\N$ and put $n=2m+1$.  Let $y\geq 0$.  Then 
\[
\log(h_m(y)-1)=-\log n! +y+n\log y-\log\!\left(1+\frac{y}{n}\right)+R(y,n)
\]
where
\[
|R(y,n)|\leq\frac{24y}{n(n+y)}
\]
for $n\geq 24$.
\end{lemma}
The slightly weaker result that $R(y,n)=O(1/n)$, uniformly in~$y$, can be obtained
from the work of 
Kriecherbauer, Kuijlaars,  McLaughlin and Miller~\cite{Kriecherbauer}.
The above approximation is better for small~$y$, which is 
advantageous for our purposes.

In terms of $g_m$ Lemma~\ref{lemma1a} takes the form
\begin{equation}\label{1b}
\log(g_m(x)-1)=-\log n! +e^x+nx-\log\!\left(1+\frac{e^x}{n}\right)+R(e^x,n).
\end{equation}
\begin{proof}[Proof of Lemma~\ref{lemma1a}]
By the formula for the error term of a Taylor series
we have
$$e^{-y}-P_m(y)=-\frac{1}{(2m)!}\int_0^y e^{-t}(y-t)^{2m}dt$$
and thus
\begin{align*}
h_m(y)-1
&=\frac{1}{(2m)!}\int_0^y e^{y-t}(y-t)^{2m}dt
%\\ &
=\frac{1}{(2m)!}\int_0^y e^{u}u^{2m}du\\
&=\frac{1}{(2m+1)!}\left(e^yy^{2m+1}-\int_0^y e^{u}u^{2m+1}du\right)\\
&=\frac{1}{(2m+1)!}e^yy^{2m+1}\left(1-\int_0^y e^{u-y}\left(\frac{u}{y}\right)^{2m+1}du\right)\\
&=\frac{1}{(2m+1)!}e^yy^{2m+1}\left(1-y\int_0^1 e^{y(s-1)}s^{2m+1}ds\right).
\end{align*}
Since $n=2m+1$ we thus have
\begin{equation}\label{1b3}
\log(h_m(y)-1)=-\log n! +y +n\log y+\log\!\left( 1-yI\right),
\end{equation}
where
\[
I=I(y,n)=\int_0^1 e^{y(s-1)}s^{n}ds.
\]
Since $\log s\leq s-1$ for $s>0$ we have
\begin{equation}\label{1b1}
I\geq \int_0^1 e^{y\log s}s^{n}ds
=\int_0^1 s^{y+n}ds=\frac{1}{y+n+1}.
\end{equation}
For an estimate in the opposite direction we use that 
\begin{equation}\label{log}
\log s\geq s-1-(s-1)^2\quad\text{for }  s\geq 1/2.
\end{equation}
With $\delta=\min\{1/2,1/\sqrt{y}\}$ 
and $\eta=1-\delta$ we write
\begin{align*}
I-\frac{1}{y+n+1}
&= \int_0^1 e^{y(s-1)}s^{n}ds
-\int_0^1 e^{y\log s}s^{n}ds
\\ &
= \int_0^1 \left(e^{y(s-1)}s^{n} - e^{y\log s}s^{n}\right)\!ds\\
&= \int_0^{\eta} \left(e^{y(s-1)}s^{n} - e^{y\log s}s^{n}\right)\!ds
+ \int_{\eta}^1 \left(e^{y(s-1)}s^{n} - e^{y\log s}s^{n}\right)\!ds\\
&= I_1 +I_2.
\end{align*}
Now  \eqref{log} yields
\[
I_2
=
\int_{\eta}^1 e^{y\log s} s^n\left(e^{y(s-1)-y\log s} - 1\right)\!ds
\leq
\int_{\eta}^1 s^{y+n}\left(e^{y(s-1)^2} - 1\right)\!ds
\]
since $\delta\leq 1/2$ and thus $s\geq 1/2$  if $s\geq\eta= 1-\delta$.
Since $\delta\leq 1/\sqrt{y}$ we have $y(s-1)^2 \leq 1$ for $s\geq 1-\delta$.
Noting that $e^t-1\leq 2t$ for $0\leq t\leq 1$ we obtain
\begin{align*}
I_2
&\leq 2 \int_{\eta}^1 s^{y+n}y(s-1)^2ds
%\\ &
\leq 2 \int_0^1 s^{y+n}y(s-1)^2ds\\
&=\frac{4y}{(y+n+1)(y+n+2)(y+n+3)}
%\\ &
\leq \frac{4}{(y+n)^2}.
\end{align*}
Moreover,
\begin{align*}
I_1
&\leq \int_0^{1-\delta} e^{y(s-1)}s^{n}ds
%\\ &
= \int_0^{1-\delta} e^{y(s-1)+n\log s}ds
\\ &
\leq \int_0^{1-\delta} e^{(y+n)(s-1)}ds
%\\ &
=\frac{1}{y+n} \left( e^{-\delta(y+n)}-e^{-(y+n)}\right)
%\\ &
\leq\frac{1}{y+n} e^{-\delta(y+n)}.
\end{align*}
Since $\delta=\min\{1/2,1/\sqrt{y}\}\geq 1/(2\sqrt{y+n})$ and since $e^t\geq t^2/2$ and
hence $e^{-t}\leq 2/t^2$ for $t\geq 0$ this yields
\[
I_1\leq \frac{1}{y+n} e^{-\sqrt{y+n}/2}\leq  \frac{8}{(y+n)^2}.
\]
Combining the bounds for $I_1$ and $I_2$ with~\eqref{1b1} we obtain
\[
 \frac{1}{y+n+1}  \leq I\leq \frac{1}{y+n+1} +\frac{12}{(y+n)^2}.
\]
Since 
\[
 \frac{1}{y+n} - \frac{1}{y+n+1}  = \frac{1}{(y+n)(y+n+1)}  \leq \frac{1}{(y+n)^2}
\]
we obtain
\[
 \left| I-\frac{1}{y+n}  \right| \leq \frac{12}{(y+n)^2}.
\]
Combining this with~\eqref{1b3} we find that 
\[
\log(h_m(y)-1)=-\log n! +y +n\log y+\log\!\left( \frac{n}{y+n} -r(y,n)\right),
\]
where
\[
|r(y,n)| \leq \frac{12y}{(y+n)^2}.
\]
Since 
\begin{align*}
\log\!\left( \frac{n}{y+n} -r(y,n)\right)
= -\log\!\left(1+\frac{y}{n}\right) +\log\!\left( 1-\frac{(y+n)r(y,n)}{n}\right)
\end{align*}
we see that $\log(h_m(x)-1)$ has the form given with 
\[
R(y,n)= \log\!\left( 1-\frac{(y+n)r(y,n)}{n}\right).
\]
To prove the estimate for $R(y,n)$ that was stated in the conclusion we note that
$|\log(1+t)|\leq 2|t|$ for $|t|\leq 1/2$ and
\[
\left|\frac{(y+n)r(y,n)}{n}\right|\leq 
 \frac{12y}{n(y+n)}\leq \frac{12}{n}\leq \frac12
\]
for $n\geq 24$. It follows that
\[
|R(y,n)|=\left|\log\!\left( 1-\frac{(y+n)r(y,n)}{n}\right)\right|\leq 
2\left|\frac{(y+n)r(y,n)}{n}\right|\leq  \frac{24y}{n(y+n)}
\]
for $n\geq 24$ as claimed.
\end{proof}
Recall that 
by~\eqref{1a1} the function
$g_m\colon \R\to (1,\infty)$ is an increasing homeomorphism.
Thus there exists 
$s_m\in\R$ such that 
\begin{equation}\label{s_m}
g_m(s_m)=2.
\end{equation}
\begin{lemma}\label{lemma2}
Let $r_0=-1.27846454\dots$ be the unique real solution of the equation
$e^{r_0}+r_0+1=0$.
Then
\[
s_m=\log n +r_0 -\frac{1}{2r_0}\frac{\log n}{n} +O\!\left(\frac{1}{n}\right)
\] 
as $m\to\infty$, with $n=2m+1$.
\end{lemma}
\begin{proof}
We write $s_m=\log n +r$. Then~\eqref{1b} and Stirling's formula yield
\begin{align*}
0&=-\log n! +ne^r+n\log n+nr-\log\!\left(1+e^r\right)
+O\!\left(\frac{1}{n}\right)
\\ &
= n(e^r+r+1)- \log\!\left(1+e^r\right)
-\frac12\log n -\frac12 \log 2\pi 
+O\!\left(\frac{1}{n}\right).
\end{align*}
This implies that $r=r_0+o(1)$  as $m\to\infty$. We write $r=r_0+t$
so that $t=o(1)$ as $m\to\infty$.
We obtain 
\begin{equation}\label{x1}
0= -nr_0 t -\frac12 \log n +O(1)+O(nt^2).
\end{equation}
This first yields that 
\[
r_0 t= -\frac12 \frac{\log n}{n} +O\!\left(\frac{1}{n}\right)+O(t^2)
\]
and hence 
\[
t\sim -\frac{1}{2r_0} \frac{\log n}{n}.
\]
Once this is known, \eqref{x1} actually gives
\[
t= -\frac{1}{2r_0} \frac{\log n}{n} +O\!\left(\frac{1}{n}\right)+O\!\left(\left(\frac{\log n}{n}\right)^2\right)
= -\frac{1}{2r_0} \frac{\log n}{n} +O\!\left(\frac{1}{n}\right),
\]
from which the conclusion immediately follows.
\end{proof}
Let now $m,M\in\N$ with $M>m$. 
We consider the function $\phi\colon\R\to\R$ defined by
\[
g_M(x)=g_m(\phi(x)).
\]
We will consider the functions $\phi$ for the case that $m=m_{k}$ and $M=m_{k+1}$
for the sequence $(m_k)$ such that $n_k=2m_k+1$, where
$(n_k)$ was constructed
in Lemma~\ref{lemma1}.

Thus we will consider the behavior of $\phi$ as $m\to\infty$, but in order
to simplify the formulas we suppress the dependence of $\phi$ from $m$ and $M$ 
from the notation. We shall assume that there exists a constant $C>1$ 
such that $M\leq Cm$.  We put $n=2m+1$ and $N=2M+1$.  Then clearly
\begin{equation}\label{NCn}
N\leq Cn.
\end{equation}
In the following, the constants appearing in the Landau notation $O(1/n)$ and $O(1)$ 
will depend on~$C$, but not on other variables.
\begin{lemma}\label{lemma1c}
The function $\phi$ has a unique fixed point $p$ which satisfies
\begin{equation}\label{1p}
p= \log n+ \frac{N\log \frac{N}{n}}{N-n} -1 
+O\!\left(\frac{1}{n}\right)
\end{equation}
as $m\to\infty$. Moreover, $\phi(x)<x$ for $x<p$ and $\phi(x)>x$ for $x>p$.  
\end{lemma}
\begin{proof}
Let $d\colon \R\to \R$, $d(x)=g_M(x)-g_m(x)$.  We have to show that $d$ has a unique zero.
Now
$$d'(x)=\left(\frac{1}{(2M)!}e^{2(M-m)x}-\frac{1}{(2m)!}\right)e^{(2m+1)x}e^{e^x},$$
from which we see that the derivative has one sign change, namely
at the point 
$$q=\frac{1}{2(M-m)}\log\frac{(2M)!}{(2m)!}.$$
Moreover, it follows from~\eqref{1b} that $d(x)<0$ if $x$ is negative and of 
sufficiently large modulus.  So $d$ decreases and stays negative on the left of~$q$,
and then increases to $+\infty$ on the right of~$q$.
We conclude that $d$ has exactly one zero $p>q$, which is the fixed point of $\phi$.

To determine the asymptotic behavior of $p$ as $m\to\infty$ we note that~\eqref{1b}
implies that
\[
-\log N! +Np-\log\!\left(1+\frac{e^p}{N}\right)
=-\log n! +np-\log\!\left(1+\frac{e^p}{n}\right)+O\!\left(\frac{1}{n}\right).
\]
We write $p=\log n+r$.
It follows that
\begin{align*}
& \; (N-n)r  \\
=& \; \log\frac{N!}{n!}
-(N-n)\log n
+\log\!\left(1+\frac{n}{N}e^r\right)
-\log\!\left(1+e^r\right)
+O\!\left(\frac{1}{n}\right)\\
=& \; \log\frac{N!}{n!}
-(N-n)\log n
+\log\!\left(1-\frac{N-n}{N}\frac{e^r}{1+e^r}\right)
+O\!\left(\frac{1}{n}\right).
\end{align*}
Now
\[
\begin{aligned}
\left|\log\!\left(1-\frac{N-n}{N}\frac{e^r}{1+e^r}\right)\right|
& \leq 
\left|\log\!\left(1-\frac{N-n}{N}\right)\right|
=\left|\log\frac{n}{N}\right|
\\ &
=\log\frac{N}{n}
=\log\!\left(1+\frac{N-n}{n}\right)
\leq 
\frac{N-n}{n}.
\end{aligned}
\]
Hence 
\[
r=
\frac{1}{N-n} \log\frac{N!}{n!} -\log n
+O\!\left(\frac{1}{n}\right).
\]
An application of Stirling's formula now yields 
\[
r=
\frac{N\log \frac{N}{n}}{N-n} -1 + 
\frac{\frac12 \log \frac{N}{n}}{N-n}  
+O\!\left(\frac{1}{n}\right).
\]
Since 
\[
\frac{\log \frac{N}{n}}{N-n}  
=
\frac{\log \frac{N}{n}}{n\left(\frac{N}{n}-1\right)}  \leq \frac{1}{n}
\]
we actually have 
\[
r=
\frac{N\log \frac{N}{n}}{N-n} -1 
+O\!\left(\frac{1}{n}\right),
\]
so that~\eqref{1p} follows.
\end{proof}
\begin{remark}
We can write the asymptotic formula for $r$ also in the form
\[
r=
\frac{\frac{N}{n}\log \frac{N}{n}}{\frac{N}{n}-1} -1 
+O\!\left(\frac{1}{n}\right).
\]
Since the function $x\mapsto (x\log x)/(x-1)$ is increasing on the 
interval $(1,C]$ and since $\lim_{x\to 1} (x\log x)/(x-1)=1$
we deduce that
\begin{equation}\label{x2}
p\geq  \log n -O\!\left(\frac{1}{n}\right)
\end{equation}
and
\begin{equation}\label{x3}
p\leq  \log n +  \frac{C\log C}{C-1} -1 
+O\!\left(\frac{1}{n}\right)
=\log n+O(1).
\end{equation}
If $N/n\to 1$, as it will be the case
in our application, we even have $p=\log n+o(1)$.
\end{remark}
\begin{lemma}\label{lemma1d}
For $m,M,n,N\in\N$ and $\phi\colon\R\to\R$ with the unique fixed point $p$ as 
before
there exist positive constants $c_1,\dots,c_8$, depending only on
the constant $C$ in~\eqref{NCn}, such that
\begin{equation}\label{1e}
0<\phi(x)-x\leq c_1 e^{-x/2} \leq \frac{c_1}{n}
\quad\text{for } 
x>8 \log n
\end{equation}
and
\begin{equation}\label{1f}
\left|\phi(x)-\frac{N}{n} x+\frac{1}{n}\log\frac{N!}{n!} \right| \leq c_2 e^{x}
\quad\text{for } x<p.
\end{equation}
Moreover, 
with $s_m$ defined by~\eqref{s_m} we have 
\begin{equation}\label{1g}
|\phi(x)-x|\leq c_3
\quad\text{for } 
x>s_m
\end{equation}
and
\begin{equation}\label{1h}
\left|\phi(x)-\frac{N}{n} x-\frac{1}{n}\log\frac{N!}{n!} \right| \leq c_4
\quad\text{for } 
x<p.
\end{equation}
Finally,
\begin{equation}\label{1h2}
\left|\phi'(x) -1\right| \leq c_5 \left(e^{-x/2}+\frac{1}{n}\right)
\leq \frac{2c_5}{n}
\quad\text{for } x>8\log n
\end{equation}
and
\begin{equation}\label{1h3}
\left|\phi'(x) -\frac{N}{n}\right| \leq 
\frac{c_6 e^{x}}{n}
\quad\text{for } x< p ,
\end{equation}
as well as
\begin{equation}\label{1h1}
c_7\leq \left|\phi'(x) \right| \leq c_8
\quad\text{for  all } x\in\R.
\end{equation}
\end{lemma}
\begin{remark}
Since $r_0<0$ it follows from Lemma~\ref{lemma2} and~\eqref{x2} that $s_m<p$ for large~$m$. 
Thus we may assume that the constants $c_2,c_4,c_6$ are 
chosen such that~\eqref{1f}, \eqref{1h} and~\eqref{1h3} also hold for $x<s_m$.
\end{remark}
\begin{proof}[Proof of Lemma~\ref{lemma1d}]
Let $y=\phi(x)$ so that $g_M(x)=g_m(y)$.
By~\eqref{1b} we have
\begin{equation} \label{1i}
\begin{aligned}
&\; -\log N! +e^x+Nx-\log\!\left(1+\frac{e^x}{N}\right)+R(e^x,N) \\
=& \;
-\log n! +e^y+ny-\log\!\left(1+\frac{e^y}{n}\right)+R(e^y,n).
\end{aligned}
\end{equation}
Suppose first that $x<p$, with $p$ as in Lemma~\ref{lemma1c}.
Then $y<x$ and thus
\[
\begin{aligned}
 & \;
\left| y-\frac{N}{n}x +\frac{1}{n} \log\frac{N!}{n!} \right|
\\ = & \;
\frac{1}{n}
\left|e^x-e^y
-\log\!\left(1+\frac{e^x}{N}\right)+R(e^x,N)
+\log\!\left(1+\frac{e^y}{n}\right)-R(e^y,n)
\right|
\\ \leq & \;
\frac{1}{n}
\left(e^x+e^y
+\frac{e^x}{N}
+\frac{24e^x}{N^2}
+\frac{e^y}{n}
+\frac{24e^y}{n^2}
\right)
%\\ \leq & \;
\leq 
\frac{52e^x}{n}.
\end{aligned}
\]
This yields~\eqref{1f} and, since $p=\log n+O(1)$ by Lemma~\ref{lemma1c}, also~\eqref{1h}.

Suppose now that $x>p$.
Using 
\[
\log\!\left(1+\frac{e^x}{N}\right)
=x-\log N +\log\!\left(1+\frac{N}{e^x}\right)
\]
and the  corresponding formula for $\log (1+e^y/n)$ we may write~\eqref{1i} in the 
form
\begin{equation} \label{1i2}
\begin{aligned}
&\; -\log N! +e^x+(N-1)x+\log N -\log\!\left(1+\frac{N}{e^x}\right)+R(e^x,N) \\
=& \;
-\log n! +e^y+(n-1)y+\log n - \log\!\left(1+\frac{n}{e^y}\right)+R(e^y,n).
\end{aligned}
\end{equation}
We write $x=\log n+s$ and $y=\log n +t$ and note that, since $x>p$, we have $t>s\geq -O(1/n)$ by~\eqref{x2}.
We obtain
\[
\begin{aligned}
 &
-\log N! +ne^s+(N-1)(\log n +s)+\log N -\log\!\left(1+\frac{N}{ne^s}\right) \\
=&
-\log n! +n e^t+(n-1)(\log n+t)+\log n - \log\!\left(1+\frac{1}{e^t}\right)
+O\!\left(\frac{1}{n}\right)
\end{aligned}
\]
and hence, using Stirling's formula,
\begin{align*}
 e^t-e^s
&= \frac{1}{n}\left( -\log\frac{N!}{n!}+(N-n)\log n+\log\frac{N}{n}
\right)
\\ & 
\qquad +\left(\frac{N}{n}-\frac{1}{n}\right)s-\left(1-\frac{1}{n}\right)t
+O\!\left(\frac{1}{n}\right)
\\  & 
\leq \frac{N}{n}s +O(1)
% \\ & 
\leq  Cs +O(1).
\end{align*}
It follows that
\begin{equation} \label{1j}
0<\phi(x)-x=t-s\leq e^{t-s}-1\leq Cse^{-s}+O(e^{-s})=O(1) \quad\text{for } x>p.
\end{equation}
For $x>8\log n$ we have $s=x-\log n>3x/4$ and thus
$2s/3>x/2$. 
It follows that
$se^{-s}=O(e^{-2s/3})=O(e^{-x/2})$ and thus~\eqref{1j} yields~\eqref{1e}.
Noting that $p-s_m=O(1)$ by Lemmas~\ref{lemma2} and~\ref{lemma1c} we can also 
deduce~\eqref{1g} 
from~\eqref{1j}.

Since $g_M'(x)=g_m'(y)\phi'(x)$ by the chain rule, \eqref{1a1} yields
\[
\phi'(x)=\frac{(n-1)!}{(N-1)!}\exp\!\left( e^x+Nx-e^y-ny\right)
\]
and thus
\[
\log\phi'(x)=
\log n!-\log N!-\log n+\log N +e^x+Nx-e^y-ny.
\]
Using~\eqref{1i2} and~\eqref{1j} we obtain
\[
\log\phi'(x)
=y-x+\log\!\left(1+\frac{N}{e^x}\right)- \log\!\left(1+\frac{n}{e^y}\right)
+O\!\left(\frac{1}{n}\right)
\]
for  $x>p$. 
Since $y>x>p\geq \log n-O(1/n)$ by~\eqref{x2} we have 
\[
\frac{n}{e^y}\leq \frac{N}{e^x}\leq (1+o(1))\frac{N}{n} \leq C+o(1).
\]
Together with~\eqref{1e} and~\eqref{1g},
and since for all $A>0$ there exists $B>0$ such that $|t-1|\leq B|\log t|$ whenever $|\log t|\leq A$, 
the same arguments as the ones used before now yield~\eqref{1h2},
as well as 
\begin{equation} \label{1k}
|\log \phi'(x)|=O(1) \quad \text{for } x>p.
\end{equation}

Similarly~\eqref{1i} yields
\[
\log\phi'(x)
=\log \frac{N}{n} +\log\!\left(1+\frac{e^x}{N}\right)
-\log\!\left(1+\frac{e^y}{n}\right)
-R(e^x,N)+R(e^y,n)
\]
and hence
\[
\log\!\left(\frac{n}{N}\phi'(x)\right)
=O\!\left(\frac{e^{x}}{n}\right)
 \quad \text{for } x<p.
\]
This yields~\eqref{1h3} as well as $|\log \phi'(x)|=O(1)$ for $x<p$ which
together with~\eqref{1k} yields~\eqref{1h1}.
\end{proof}

\subsection{Definition of a quasiregular map}\label{defG}
The idea is to construct an entire function by gluing functions $g_m$ with different values
of $m$ appropriately.
Actually, we will first modify the functions $g_m$ slightly to obtain closely related
functions $u_m$ and $v_m$.
We then glue restrictions of these maps to half-strips along horizontal lines to
obtain quasiregular maps $U$ and $V$ which are defined in the right
and left half-plane. Then we will glue these functions along the imaginary axis to
obtain a quasiregular map $G$ in the whole plane.

In the next section we will show that
the map constructed satisfies the hypothesis of the  Teich\-m\"ul\-ler--Wit\-tich--Be\-linskii theorem.

The maps $U$, $V$ and $G$ will commute with complex conjugation,
so it will be enough to define them in the upper half-plane.
We begin by constructing the map~$U$.

Instead of $g_m$ we consider the map
\[
u_m\colon \{z\in\C\colon \re z \geq 0\}\to\C, \quad
u_m(z)=g_m(z+s_m).
\]
Note that $u_m$ is increasing on the real line, and
maps $[0,\infty)$ onto $[2,\infty)$.

Let $(n_k)$ be the sequence from Lemma~\ref{lemma1} and write $n_k=2m_k+1$.
Basically, we would like to put $U(z)=u_{m_k}(z)$ in the half-strip
$$\Pi_k^+=\{ x+iy: x>0, 2\pi(k-1)<y<2\pi k\}.$$
However, this function $U$ will be discontinuous.
In order to obtain a continuous function  we consider the function
$\psi_k\colon [0,\infty)\to [0,\infty)$ defined by $u_{m_{k+1}}(x)=u_{m_k}(\psi_k(x))$.
This function $\psi_k$ is closely related to the functions $\phi$ considered in
Lemmas~\ref{lemma1c} and Lemma~\ref{lemma1d}.
In fact, denote by $\phi_k$ the function $\phi$ corresponding to $m=m_k$ and $M=m_{k+1}$.
Then 
\begin{equation} \label{2a}
\psi_k(x)
= \phi_k(x+s_{m_{k+1}})-s_{m_k} .
\end{equation}

We then define $U\colon \{z\in\C\colon \re z \geq 0\}\to\C$ by interpolating
between $u_{m_{k+1}}$ and $u_{m_k}$ as follows.
If $2\pi (k-1)\leq y < 2\pi k$, say $y=2\pi (k-1)+2\pi t$ where $0\leq t< 1$,
then we put
$$U(x+iy)=u_{m_k}((1-t)x+t\psi_k(x)+iy)=u_{m_k}(x+iy+t(\psi_k(x)-x)).$$
Actually, by $2\pi i$-periodicity we have
$$U(x+iy)=u_{m_k}((1-t)x+t\psi_k(x)+2\pi i t).$$
The function $U$ defined this way is continuous in the right half-plane.

We now define a function $V$ in the left half-plane.
In order to do so, 
we define
\[
v_m\colon \{z\in\C\colon \re z \leq 0\}\to\C, \quad
v_m(z)=g_m\!\left(\frac{z}{2m+1}+s_m\right).
\]
Note that $v_m$ maps $(-\infty,0]$ monotonically onto $(1,2]$.

Let $(n_k)$ and $(m_k)$ be as before and put $N_k=\sum_{j=1}^{k}n_j$,
with $N_0=0$.
This time we would like to define
$V(z)=v_{m_k}(z)$ in the half-strip 
$$\{ x+iy: x\leq 0,\; 2 \pi N_{k-1} \leq y < 2 \pi N_k\},$$
but again this function would be
discontinuous, so in order
to obtain a continuous function we again interpolate between $v_{m_{k+1}}$ and $v_{m_k}$.
Similarly as before we consider the map $\psi_k\colon (-\infty,0]\to (-\infty,0]$ defined by
$v_{m_{k+1}}(x)=v_{m_k}(\psi_k(x))$.
Then we define $V\colon \{z\in\C\colon \re z \leq 0\}\to\C$ by interpolating
between $v_{m_{k+1}}$ and $v_{m_k}$ as follows:
if $2\pi N_{k-1}\leq y < 2\pi N_k$, say $y=2\pi N_{k-1}+2\pi n_k t$ where $0\leq t< 1$,
then we put
$$V(x+iy)=v_{m_k}((1-t)x+t\psi_k(x)+iy)=v_{m_k}(x+iy+t(\psi_k(x)-x)).$$
This map $V$ is continuous in the left half-plane.

Now we define our map $G$ by gluing $U$ and $V$ along the imaginary axis.
In order to do this we note that by construction we have $U(iy)=V(ih(y))$
and thus $U(ig(y))=V(iy^\gamma)$ for $y\geq 0$,
with the maps $h$ and $g$ from Lemma~\ref{lemma1}.
Therefore we will consider a
homeomorphism $Q$ of the right half-plane 
\[
H^+=\{z\in\C\colon\re z\geq 0\}
\]
onto itself,
satisfying $Q(\overline{z})=\overline{Q(z)}$,
such that
$Q(\pm iy)=\pm i g(y)$ for $y\geq 0$ while $Q(z)=z$ for $\re z\geq 1$.
We thus have to define $Q(z)$ in the strip $\{z\in\C\colon 0<\re z<1\}$.
For $|\im z|\geq 1$ we define $Q$ in this strip by interpolation; that is, we put
\begin{equation}\label{Q}
Q(x\pm iy) =
x\pm i((1-x) g(y)+xy)       \quad\text{if } 0< x< 1 \text { and } y\geq 1.
\end{equation}
In the remaining part of the strip we define $Q$ by
\begin{equation}\label{Q1}
Q(z) =
\begin{cases}
z |z|^{\gamma-1}         & \text{if } 0< |z|\leq 1,\\[1mm]
z                   & \text{if } |z|> 1, \text{ but } 0< \re z< 1 \text{ and } 0\leq |\im z|\leq  1.
\end{cases}
\end{equation}
Note that $h(y)=y$ for $0\leq y\leq 2\pi$ since $n_1=1$, and thus $g(y)=y^\gamma$ for $0\leq y\leq 1$.
Thus $Q(\pm iy)=\pm i g(y)$ also for $|y|\leq 1$. Moreover, we have $g(1)=1$, meaning that
the above expressions for $Q(z)$ do indeed coincide for $|\im z|=1$.

We conclude that the map $W=U\circ Q$ satisfies
\[
W(\pm iy)=U(\pm ig(y))=V(\pm iy^\gamma)
\quad\text{for }y\geq 0.
\]

Let  now $\rho\in (1/2,1)$ 
as in the hypothesis of Theorem~\ref{thm1}.
We choose 
$\gamma=1/(2\rho-1)$
in the above construction and put
$\sigma=\rho\gamma=\rho/(2\rho-1)$.
The hypothesis that $1/2<\rho<1$
corresponds to 
$\gamma>1$
as well as $\sigma>1$.
The map
$$
G(z) =
\begin{cases}
W(z^{\rho})              & \text{if } |\arg z|\leq \displaystyle \frac{\pi}{2\rho} ,\\[3mm]
V(-(-z)^{\sigma})      & \text{if } |\arg(-z)|\leq \displaystyle \frac{\pi}{2\sigma},
\end{cases}
$$
is continuous in $\C$.
Here, for $\eta>0$, we denote by $z^\eta$ the principal branch of the power which is defined in
$\C\backslash (-\infty,0]$.

\subsection{Estimation of the dilatation} \label{est-dil}
We will use the Teich\-m\"ul\-ler--Wit\-tich--Be\-linskii theorem stated in section 2.2
to show
that the map $G$ defined in the previous section has the form
$G(z)=F(\tau(z))$ with an entire function $F$
and a homeomorphism $\tau$ satisfying
$\tau(z)\sim z$ as $z\to\infty$.

For a quasiregular map~$f$, let
\[
\mu_f(z)=\frac{f_{\overline{z}}(z)}{f_z(z)}
\quad\text{and}\quad
K_f(z)=\frac{1+|\mu_f(z)|}{1-|\mu_f(z)|} .
\]
In order to apply the Teich\-m\"ul\-ler--Wit\-tich--Be\-linskii theorem, 
we have to estimate $K_G(z)-1$.
We note that
\[
K_G(z)-1
= \frac{2|\mu_G(z)|}{1-|\mu_G(z)|}
= \frac{2|\mu_G(z)|(1+|\mu_G(z)|)}{1-|\mu_G(z)|^2}
\leq \frac{4|\mu_G(z)|)}{1-|\mu_G(z)|^2}.
\]
We begin by estimating $K_U(z)-1$.
Let 
$2\pi (k-1)\leq y < 2\pi k$ so that $y=2\pi (k-1)+2\pi t$ where 
\[
0\leq t = \frac{y}{2\pi}-(k-1)< 1.
\]
Then $U(z)=u_{m_k}(q(z))$ where 
\begin{equation} \label{defq}
\begin{aligned}
q(x+iy)
&=x+iy+ t(\psi_k(x)-x) 
\\ &
=x+iy+ \left(\frac{y}{2\pi}-(k-1)\right)(\psi_k(x)-x).
\end{aligned}
\end{equation}
Thus
\begin{equation} \label{x4}
q_z(z)=1+a(z)-i b(z)
\quad\text{and}\quad
q_{\overline{z}}(z)=a(z)+i b(z)
\end{equation}
with
\begin{equation} \label{x4a}
a(x+iy)=\frac{t}{2}(\psi_k'(x)-1) 
\quad\text{and}\quad
b(x+iy)=\frac{1}{4\pi}(\psi_k(x)-x).
\end{equation}
Note that if $a(z)<0$, then
\[
a(z)=\frac{t}{2}(\psi_k'(x)-1) \geq \frac12 (\psi_k'(x)-1)
\]
and thus 
\[
1+2a(z)\geq \psi_k'(x) >0.
\]
Thus we have 
\[
1+2a(z)\geq \min\{1,\psi_k'(x)\} >0
\quad\text{as well as}\quad
1+a(z) >0
\]
in any case.

We deduce from~\eqref{x4} that
\[
|\mu_U(z)|^2=
|\mu_q(z)|^2=\frac{a(z)^2+b(z)^2}{(1+a(z))^2+b(z)^2}
\]
and thus
\begin{align*}
K_U(z)-1
&\leq \frac{4|\mu_q(z)|}{1-|\mu_q(z)|^2} \\
&=\frac{4\sqrt{(1+a(z))^2 +b(z)^2}\sqrt{a(z)^2+b(z)^2}}{1+2a(z)} \\
&\leq\frac{4( 1+a(z) +|b(z)|)(|a(z)|+|b(z)|)}{1+2a(z)}.
\end{align*}
Altogether we find that
\[
K_U(z)-1
\leq 
\frac{4( 1+|\psi_k'(x)-1| +|\psi_k(x)-x|)(|\psi_k'(x)-1|+|\psi_k(x)-x|)}{ \min\{1,\psi_k'(x)\}}.
\]
With 
\begin{equation}\label{3a1}
r(x)=|\psi_k'(x)-1| +|\psi_k(x)-x|
\end{equation}
 we thus have
\begin{equation}\label{3a}
K_U(z)-1
\leq 
\frac{4( 1+r(x))r(x)}{ \min\{1,\psi_k'(x)\}}.
\end{equation}
We shall use Lemma~\ref{lemma1d} to estimate the terms occurring here.

Let now 
\[
S_k=\{x+iy\colon 2\pi (k-1)<y<2\pi k,\; x>\max\{1, 8\log n_k\}\}.
\]
In order to estimate $K_U(z)-1$ for $z\in S_k$ we note that
Lemma~\ref{lemma2} yields that if $z=x+iy\in S_k$, then $x+s_{m_{k+1}}>8\log n_k$
for large~$k$.
For such $k$ we deduce from~\eqref{1e}, \eqref{2a} and Lemma~\ref{lemma2} that
\begin{equation}\label{x9}
\begin{aligned}
|\psi_k(x)-x| 
&=
|\phi_k(x+s_{m_{k+1}})-(x+s_{m_{k+1}})+s_{m_{k+1}}-s_{m_{k}}| 
\\ &
\leq \frac{c_1}{n_k} +\left| \log n_{k+1} -\log n_{k}\right|
+\frac{1}{2|r_0|}\left| \frac{\log n_{k+1}}{n_{k+1}} 
- \frac{\log n_k}{n_k} \right|.
\end{aligned}
\end{equation}
With $\delta=\min\{1,\gamma-1\}$ we deduce from~\eqref{0d} that
\begin{equation}\label{x5}
n_k= \left(1+O\!\left(\frac{1}{k^\delta}\right)\right)\gamma(2\pi k)^{\gamma-1}
\end{equation}
and hence
\begin{equation}\label{x6}
\frac{n_{k}}{n_{k+1}}
=\left(1+O\!\left(\frac{1}{k^\delta}\right)\right) \left(\frac{k}{k+1}\right)^{\gamma-1} 
=1+O\!\left(\frac{1}{k^\delta}\right).
\end{equation}
It follows that
\begin{equation}\label{x7}
\log n_{k+1} -\log n_{k}=\log\frac{n_{k+1}}{n_k}=O\!\left(\frac{1}{k^\delta}\right)
\end{equation}
as well as
\begin{equation}\label{x8}
\frac{\log n_{k+1}}{n_{k+1}} - \frac{\log n_k}{n_k}
=\left( \frac{n_{k}}{n_{k+1}}-1\right) \frac{\log n_{k+1}}{n_{k}}
+\frac{\log \frac{n_{k+1}}{n_k}}{n_k}
=O\!\left(\frac{1}{k^\delta}\right).
\end{equation}
Thus~\eqref{x9} yields that
\begin{equation}\label{x10}
|\psi_k(x)-x| 
=O\!\left(\frac{1}{k^\delta}\right).
\end{equation}
By~\eqref{0d} and~\eqref{1h2} we also have
\[
|\psi_k'(x)-1| = |\phi_k'(x+s_{m_{k+1}})-1|
= O\!\left(\frac{1}{n_k}\right)
= O\!\left(\frac{1}{k^\delta}\right).
\]
The last two equations imply that the function $r(x)$ defined by~\eqref{3a1}
satisfies $|r(x)|=O(1/k^{\delta})$. Hence
\[
K_U(z)-1
= O\!\left(\frac{1}{k^\delta}\right)
\quad\text{for } z\in S_k
\]
by~\eqref{1h1} and~\eqref{3a}.

Now 
\begin{align*}
\int_{S_k} \frac{dx\, dy}{x^2+y^2}
&
\leq \int_{S_k} \frac{dx\, dy}{x^2+4\pi^2(k-1)^2}
\\ &
\leq 2\pi\int_0^\infty \frac{dx}{x^2+4\pi^2(k-1)^2}
%\\ &
=\frac{\pi}{2(k-1)}
%\\ &
\leq \frac{\pi}{k}
\end{align*}
for $k\geq 2$ and 
\[
\int_{S_1} \frac{dx\, dy}{x^2+y^2}
\leq 2\pi \int_1^\infty \frac{dx}{x^2}=2\pi .
\]
Combining the last three inequalities and noting that $W(z)=U(z)$  for $z\in S_k$ we deduce that
\begin{equation}\label{3c}
\int_{S_k} \frac{K_W(z)-1}{x^2+y^2}dx\, dy
=
\int_{S_k} \frac{K_U(z)-1}{x^2+y^2}dx\, dy
\leq 
\frac{A_1}{k^{1+\delta}}
\end{equation}
for some constant~$A_1$.

Combining~\eqref{1g} and~\eqref{1h1} with~\eqref{3a} we also see that $U$ is 
quasiregular in the right half-plane $H^+$.

Next we show that the function $Q$ defined in section~\ref{defG} is quasiconformal.
The quasiconformality is trivial in those parts of $H^+$  where $Q(z)$ is the identity, and
it is also easily shown in $\{z\in\C\colon |z|<1,\; \re z>0\}$, where $Q$ is given
by $Q(z)=z|z|^{\gamma-1}$.
So we only have to consider the half-strip $\{z\in\C\colon 0<\re z<1,\;\im z>1\}$, where $Q(z)$
is given by~\eqref{Q}.
Similarly as in the computation of $\mu_U$ we find, for $z$ in this half-strip, that
\[
Q_z(z)=1+a(z)-i b(z)
\quad\text{and}\quad
Q_{\overline{z}}(z)=-a(z)-i b(z),
\]
where 
\[
a(x+iy)=\frac{1-x}{2}(g'(y)-1) 
\quad\text{and}\quad
b(x+iy)=\frac{1}{2}(g(y) -y)
\]
and thus 
\[
|\mu_Q(z)|^2=\frac{a(z)^2+b(z)^2}{(1+a(z))^2+b(z)^2}
=\frac{a(z)^2+b(z)^2}{1+2a(z)+a(z)^2+b(z)^2}.
\]
Since $g(y)-y=o(1)$ and $g'(y)-1=o(1)$ as $y\to\infty$ by Lemma~\ref{lemma1},
we conclude $a$ and $b$ are bounded. 
Since $g'(y)=\gamma y^{\gamma-1}/h'(g(y))=\gamma y^{\gamma-1}/n$ for some positive integer $n$ depending on $y$
by the definition of $g$ and $h$, we deduce from~\eqref{0c} that $\inf_{y\geq 1} g'(y)>0$.
This implies that 
\[
\inf_{0<x<1,\; y>1} |1+2a(x+iy)|>0.
\]
We deduce that $Q$ is indeed quasiconformal in~$H^+$.
Hence $W=U\circ Q$ is quasiregular in~$H^+$.

We put  
\[
S_k'=\{x+iy\colon 2\pi (k-1)<y<2\pi k,\; 0<x<\max\{1,8\log n_k\}\}.
\]
For $k\geq 2$ we have
\[
\int_{S_k'} \frac{dx\, dy}{x^2+y^2}
\leq \frac{1}{4\pi^2(k-1)^2}\int_{S_k}  dx\, dy
=\frac{\max\{1,8\log n_k\}}{2\pi(k-1)^2}
\leq A_2 \frac{\log k}{k^2}
\]
for some constant $A_2$ in view of~\eqref{0d}.
Let $K$ be the dilatation of $W$ in~$H^+$; that is, 
$K=\sup_{z\in H^+} K_W(z)$.
We conclude that 
\begin{equation}\label{3d}
\int_{S_k'} \frac{K_W(z)-1}{x^2+y^2}dx\, dy
\leq (K-1)A_2 \frac{\log k}{k^2}
\end{equation}
for $k\geq 2$.
Let now $R>2\pi +1$. Then $\{z\in\C\colon |z|>R\}\cap S_1'=\emptyset$ since $n_1=1$.
We deduce from~\eqref{3c} and~\eqref{3d} that
\begin{align*}
& \; \int_{\substack{z\in H^+\\|z|>R}} \frac{K_W(z)-1}{x^2+y^2}dx\, dy
\\ \leq & \; 
2\sum_{k=2}^\infty \int_{S_k'} \frac{K_W(z)-1}{x^2+y^2}dx\, dy
+
2\sum_{k=1}^\infty \int_{S_k} \frac{K_W(z)-1}{x^2+y^2}dx\, dy
\\ 
\leq & \;
2(K-1)A_2\sum_{k=2}^\infty  \frac{\log k}{k^2}
+
2A_1\sum_{k=1}^\infty \frac{1}{k^{1+\delta}}dx\, dy
<  \infty.
\end{align*}
For $|\arg z|<\pi/(2\rho)$ we have $G(z)=W(z^{\rho})$. 
It follows that
\begin{equation}\label{3e}
\begin{aligned}
\int_{\substack{|\arg z|<\frac{\pi}{2\rho}\\|z|>R^{1/\rho}}} \frac{K_G(z)-1}{x^2+y^2}dx\, dy
 & =
\int_{\substack{|\arg z|<\frac{\pi}{2\rho}\\|z|>R^{1/\rho}}}
\frac{K_W(z^{\rho})-1}{x^2+y^2}dx\, dy
\\ & = 
\frac{1}{\rho^2}\int_{\substack{z\in H^+\\|z|>R}} \frac{K_W(z)-1}{x^2+y^2}dx\, dy
< \infty.
\end{aligned}
\end{equation}

The estimate of $K_G(z)$ for $\arg(-z)<\pi/(2\sigma)$ is similar. 
We have $V(z)=v_{m_k}(q(z))$ where, instead of~\eqref{defq}, we have
\begin{equation} \label{defqx}
q(x+iy)
=x+iy+ \frac{1}{n_k}\left(\frac{y}{2\pi}-N_{k-1}\right)(\psi_k(x)-x)
\end{equation}
for $2\pi N_{k-1}\leq y<2\pi N_k$,
with
\[
\psi_k(x)=n_k\phi_k\!\left(\frac{x}{n_{k+1}}+s_{m_{k+1}}\right)-n_ks_{m_k}.
\]
Now~\eqref{x4} holds with $a(x+iy)$ as in~\eqref{x4a}, but
\begin{equation} \label{x4b}
b(x+iy)=\frac{1}{4\pi n_k}(\psi_k(x)-x).
\end{equation}
Instead of \eqref{3a1} and~\eqref{3a} we obtain
\begin{equation}\label{3f}
K_V(z)-1
\leq 
\frac{4( 1+r(x))r(x)}{ \min\{1,\psi_k'(x)\}}
\end{equation}
with
\begin{equation}\label{3f1}
r(x)=|\psi_k'(x)-1| +\frac{1}{n_k}|\psi_k(x)-x|.
\end{equation}

Now
\begin{equation}\label{3f2}
\begin{aligned}
\frac{\psi_k(x)-x}{n_k}
&= 
\phi_k\!\left(\frac{x}{n_{k+1}}+s_{m_{k+1}}\right)
-\frac{n_{k+1}}{n_k} \left(\frac{x}{n_{k+1}} + s_{m_{k+1}}\right) 
\\ & \qquad
+\frac{1}{n_k}\log \frac{n_{k+1}!}{n_k!}
+ \frac{n_{k+1}}{n_k}s_{m_{k+1}}
-s_{m_k}
-\frac{1}{n_k}\log \frac{n_{k+1}!}{n_k!}.
\end{aligned}
\end{equation}
It follows from Lemma~\ref{lemma2}, \eqref{0d} and~\eqref{x6} that
\begin{equation}\label{3f3}
\left| \frac{n_{k+1}}{n_k}s_{m_{k+1}}
-s_{m_k}
-\frac{1}{n_k}\log \frac{n_{k+1}!}{n_k!}\right| =O\!\left(\frac{1}{k^\delta}\right).
\end{equation}
For  $x\leq -2n_k\log n_k$ we find, using~Lemma~\ref{lemma2}, Lemma~\ref{lemma1c} and~\eqref{x6}, that
\[
\frac{x}{n_{k+1}}+s_{m_{k+1}}\leq-2\frac{n_k}{n_{k+1}}\log n_k+\log n_k+r_0+o(1)
\leq -\log n_k +r_0+o(1)
\]
as $k\to\infty$.  Hence $x/n_{k+1}+s_{m_{k+1}}\leq -\log n_k$
for sufficiently large~$k$.  We thus can deduce from~\eqref{1f} that 
\begin{equation}\label{3f4}
\begin{aligned}
& \;
\left|
\phi_k\!\left(\frac{x}{n_{k+1}}+s_{m_{k+1}}\right)
-\frac{n_{k+1}}{n_k} \left(\frac{x}{n_{k+1}} + s_{m_{k+1}}\right) 
+\frac{1}{n_k}\log \frac{n_{k+1}!}{n_k!}
\right|
\\ \leq & \;
c_2 \exp\!\left( \frac{x}{n_{k+1}}+s_{m_{k+1}}\right)
%\\ \leq & \;
\leq
\frac{c_2}{n_k}
\quad\text{for } x\leq -2n_k\log n_k
\end{aligned}
\end{equation}
for large~$k$.
It follows from~\eqref{3f2}, \eqref{3f3} and~\eqref{3f4} that
\begin{equation}\label{3g}
\frac{1}{n_k} |\psi_k(x)-x|
= O\!\left(\frac{1}{k^\delta}\right)
\quad \text{for } x\leq -2n_k\log n_k.
\end{equation}
Moreover, the above arguments in conjunction with~\eqref{1h} show that
\begin{equation}\label{3h}
\frac{1}{n_k} |\psi_k(x)-x|
=O(1)
\quad \text{for }  -2n_k\log n_k<x<0.
\end{equation}

Next we note that 
\[
\psi_k'(x)=\frac{n_k}{n_{k+1}}\phi_k'\!\left(\frac{x}{n_{k+1}}+s_{m_{k+1}}\right).
\]
From~\eqref{1h3} and~\eqref{x6} we can now deduce that
\[
|\psi_k'(x)-1|
= O\!\left(\frac{1}{n_k}\right)
= O\!\left(\frac{1}{k^\delta}\right)
\quad \text{for } x\leq -2n_k\log n_k
\]
and~\eqref{1h1} yields that
\[
|\psi_k'(x)-1|=O(1)
\quad \text{for }  -2n_k\log n_k<x<0.
\]
Combining the last two inequalities with~\eqref{3f}, \eqref{3f1},
\eqref{3g} and~\eqref{3h} we conclude that 
$V$ is quasiregular in the left half-plane 
$$H^-=\{z\in\C\colon \re z< 0\}$$
and that
\begin{equation}\label{3i}
K_V(z)-1 
= O\!\left(\frac{1}{k^\delta}\right)
\quad \text{for } x\leq -2n_k\log n_k.
\end{equation}

In analogy with $S_k$ and $S_k'$ we put
\[
T_k=\{x+iy\colon 2\pi N_{k-1}<y<2\pi N_k,\; x<-2n_k\log n_k\}
\]
and
\[
T_k'=\{x+iy\colon 2\pi N_{k-1}<y<2\pi N_k,\; -2n_k\log n_k<x<0\}.
\]
For $k\geq 2$ we have
\begin{align*}
\int_{T_k} \frac{dx\, dy}{x^2+y^2}
&
\leq \int_{T_k} \frac{dx\, dy}{x^2+4\pi^2 N_{k-1}^2}
\leq 2\pi n_k \int_0^\infty \frac{dx}{x^2+4\pi^2 N_{k-1}^2}
=\frac{\pi n_k}{2N_{k-1}}.
\end{align*}
Since
\begin{equation}\label{0e}
N_k\sim (2\pi)^{\gamma-1}k^{\gamma}
\end{equation}
as $k\to\infty$ by~\eqref{0d} 
this yields
\[
\int_{T_k} \frac{dx\, dy}{x^2+y^2}
= O\!\left(\frac{1}{k}\right)
\]
for $k\geq 2$.
In analogy with~\eqref{3c} we can use these estimates and~\eqref{3i} to deduce that
\[
\int_{T_k} \frac{K_V(z)-1}{x^2+y^2}dx\, dy
\leq 
\frac{A_3}{k^{1+\delta}}
\]
for $k\geq 2$ and some constant~$A_3$.
Similarly, if $k\geq 2$, then
\[
\begin{aligned}
\int_{T_k'} \frac{dx\, dy}{x^2+y^2}
& \leq \frac{1}{4\pi^2N_{k-1}^2}\int_{T_k'}  dx\, dy
\\ &
=\frac{2\pi (N_k-N_{k-1})2n_k\log n_k}{4\pi^2 N_{k-1}^2}
=\frac{n_k^2\log n_k}{\pi N_{k-1}^2}
\leq A_4 \frac{\log k}{k^2}
\end{aligned}
\]
for some constant $A_4$ by~\eqref{0d} and~\eqref{0e}.
Since also
\[
\int_{T_1\cap \{z\colon |z|>R\}} \frac{dx\, dy}{x^2+y^2}<\infty 
\]
for $R>0$,
the above inequalities now imply that
\[
\int_{\substack{z\in H^-\\|z|>R}} \frac{K_V(z)-1}{x^2+y^2}dx\, dy
< \infty .
\]
For $|\arg(-z)|<\pi/(2\sigma)$ we have $G(z)=V(-(-z)^{\sigma})$. 
Similarly as in~\eqref{3e} it follows that
\[
\int_{\substack{|\arg (-z)|<\frac{\pi}{2\sigma}\\|z|>R^{1/\sigma}}} \frac{K_G(z)-1}{x^2+y^2}dx\, dy
 =
\frac{1}{\sigma^2} \int_{\substack{z\in H^-\\|z|>R}} \frac{K_V(z)-1}{x^2+y^2}dx\, dy
< \infty.
\]
Combining this with~\eqref{3e} we see that if $r>0$, then
\[
\int_{|z|>r } \frac{K_G(z)-1}{x^2+y^2}dx\, dy
< \infty.
\]
Thus $G$ satisfies the hypothesis of the Teich\-m\"ul\-ler--Wit\-tich--Be\-linskii
theorem~\cite[\S V.6]{LV}.
This theorem,
together with the existence theorem for quasiconformal mappings~\cite[\S V.1]{LV},
yields that there exists a quasiconformal homeomorphism $\tau\colon\C\to\C$
and an entire function $F$ such that 
\begin{equation} \label{FG}
G(z)=F(\tau(z)) \quad\text{and}\quad 
\tau(z)\sim z \quad\text{as } z\to\infty.
\end{equation} 

\subsection{Asymptotic behavior of $F,E$ and $A$: proof of Theorems~\ref{thm1} and~\ref{thm1a}}
\label{AsyBeh}
\begin{proof}[Proof of Theorem~\ref{thm1}]
We begin by estimating the counting function of the zeros of~$F$.
Let $r>0$ and choose $k\in\N$ such that $2\pi(k-1)<r\leq 2\pi k$.
It follows from the construction and~\eqref{0e} that
\[
n(r,0,U)\leq 4\sum_{j=1}^k m_j\sim 2 N_k 
\sim 2(2\pi)^{\gamma-1}  k^\gamma 
=O(r^\gamma).
\]
Similarly, if $2\pi N_{k-1}<r\leq 2\pi N_k$, then
\[
n(r,0,V)
\leq 4\sum_{j=1}^k m_j\sim 2 N_k 
\sim 2(2\pi)^{\gamma-1}  k^\gamma 
=O(r).
\]
This implies that
\begin{equation} \label{Gn}
n(r,0,G)=O(r^{\rho\gamma})+O(r^{\sigma})
=O(r^{\sigma})
\end{equation} 
since $\rho\gamma=\sigma$.
Now~\eqref{FG} yields
$n(r,0,F)=O(r^{\sigma})$ and hence
\begin{equation} \label{Fn}
N(r,0,F)=O(r^{\sigma})
\end{equation} 
as $r\to\infty$.

Next we note that the coefficients in the Taylor series expansion of $h_m$ 
are all non-negative. This implies that $|h_m(z)|\leq h_m(|z|)$ 
and hence 
\[
|g_m(z)|\leq g_m(\re z)
\]
for all $z\in\C$.
Clearly this implies that 
\[
|v_m(z)|\leq v_m(\re z)
\quad\text{for } z\in H^-
\quad\text{and} \quad
|u_m(z)|\leq u_m(\re z)
\quad\text{for } z\in H^+.
\]

Let $z=x+iy\in H^-$ with $\im z\geq 0$  and 
$k\in\N$ with  $2\pi N_{k-1}\leq \im z < 2\pi N_k$.
With $t=(y-2\pi N_{k-1})/(2\pi n_k)$ we have $0\leq t< 1$ and
\begin{align*}
|V(x+iy)|
&
=|v_{m_k}((1-t)x+t \psi_k(x) +iy)|
\\ &
\leq v_{m_k}((1-t)x+t \psi_k(x))\leq v_{m_k}(0)
=g(s_{m_k})=2.
\end{align*}
Thus
\begin{equation} \label{maxV}
|V(z)|\leq 2
\end{equation} 
for $z\in H^-$.
For $z\in H^+$
we use the estimate
\begin{equation}\label{gmH+}
\begin{aligned}
|g_m(z)|
&
\leq g_m(\re z)
%\\ &
=\sum_{j=0}^{2m}(-1)^j \frac{1}{j!}e^{j\re z}\exp \!\left( e^{\re z}\right)
\\ &
\leq e^{1+2m \re z}\exp \!\left( e^{\re z}\right).
\end{aligned}
\end{equation}
Again, assuming $\im z\geq 0$, we choose $k\in\N$ such that $2\pi(k-1)\leq \im z< 2\pi k$.
Then 
\[
|U(z)|=|u_{m_k}(q(z))|
=|g_{m_k}\!\left(q(z)+s_{m_k}\right)|
\leq g_{m_k}\!\left(\re q(z)+s_{m_k}\right),
\]
with $q(z)$ defined by~\eqref{defq}.
Noting that $m_k=O\!\left(k^{\gamma -1}\right)=O\!\left(|z|^{\gamma -1}\right)$ by~\eqref{0d} and thus
\begin{equation}\label{smk}
s_{m_k}=\log n_k+O(1)=\log m_k+O(1)=O(\log|z|)
\end{equation}
we deduce from~\eqref{1g} and~\eqref{gmH+} that 
\begin{align*}
\log |U(z)| 
&\leq 1+ 2m_k\!\left(\re q(z) + s_{m_k}\right) +\exp\!\left(\re q(z)+s_{m_k}\right)
\\ &
\leq \exp((1+o(1))|z|)+O(|z|^\gamma)
%\\ &
\leq \exp((1+o(1))|z|)
\end{align*}
and hence 
\[
\log\log|U(z)|
\leq (1+o(1))|z|
\quad\text{as } z\to\infty\text{ in }H^+.
\]
Together with~\eqref{maxV} we conclude that 
\begin{equation} \label{GM}
\log\log|G(z)|
\leq (1+o(1))|z|^\rho
\end{equation} 
as $|z|\to\infty$.
Hence~\eqref{FG} yields that
\[
\log\log|F(z)|
\leq (1+o(1))|z|^{\rho}
\]
as $|z|\to\infty$.
The lemma on the logarithmic derivative~\cite[Section~3.1]{GO} now implies 
that $E=F/F'$ satisfies 
\begin{equation}\label{mr1E}
m\!\left(r,\frac{1}{E}\right)=O(r^{\rho}).
\end{equation}
By~\eqref{Fn}, and since the zeros of $E$ are simple, we have
\[
N(r,0,E)=O(r^{\sigma})
\]
and we conclude that 
\[
T(r,E)=O(r^{\sigma}).
\]
In particular, $E$ has finite order.
Since $F$ is entire, $E$ is clearly a special Bank--Laine function.
Thus $E$ is the product of two solutions of~\eqref{w''+Aw}, one of which has no zeros.

The lemma on the logarithmic derivative, together with \eqref{BL} and \eqref{mr1E},
also implies that
\[
m(r,A)=2 m\!\left(r,\frac{1}{E}\right) +O(\log r)= O(r^{\rho}).
\]
We thus have $\lambda(E)\leq \rho(E)\leq\sigma$ and $\rho(A)\leq \rho$.
Since 
\[
\frac{1}{\rho}+\frac{1}{\sigma}=2
\]
by the definition of~$\sigma$, we deduce from~\eqref{ineq1} that actually
$\lambda(E)= \rho(E)=\sigma$ and $\rho(A)= \rho$.
This completes the proof of Theorem~\ref{thm1}.
\end{proof}

\begin{proof}[Proof of Theorem~\ref{thm1a}]
We estimate the asymptotics of $F$, $E$ and $A$ more accurately than in the 
previous proof.
First we note that for $|z|> 4m$ we have
\begin{equation} \label{Pm}
\begin{aligned}
\left|P_m(z)-\frac{z^{2m}}{(2m)!}\right|
& =
\left|\sum_{k=0}^{2m-1}(-1)^k\frac{z^k}{k!}\right|
%\\ &
\leq \sum_{k=0}^{2m-1}\frac{|z|^k}{k!}
= \sum_{k=0}^{2m-1}\frac{|z|^{2m-1}}{|z|^{2m-1-k}k!}
\\ &
\leq  \frac{|z|^{2m-1}}{(2m-1)!}
 \sum_{k=0}^{2m-1}\frac{1}{2^{2m-1-k}}
\leq \frac{4m}{|z|} \frac{|z|^{2m}}{(2m)!}
\end{aligned}
\end{equation}
and thus in particular $P_m(z)\neq 0$.

Let $0<\varepsilon_1<\varepsilon_2<\varepsilon_3<\varepsilon_4<\varepsilon_5<\varepsilon<1$ and, for $j\in\{1,2\}$, put
\[
H^+_{\varepsilon_j}=\left\{z\in\C\colon |\arg z|<(1-\varepsilon_j)\frac{\pi}{2}\right\}.
\]
Given $z\in H^+_{\varepsilon_1}$ with $\im z\geq 0$, we choose  $k\in\N$ with $2\pi(k-1)\leq \im z< 2\pi k$.
We then have 
$k=O(\re z)$ and hence, by~\eqref{0d},
$\log m_k=o(\re z)$ and
$m_k=o(e^z)$ as $z\to\infty$ in $H^+_{\varepsilon_1}$.
We deduce from~\eqref{Pm} that 
\[
P_{m_k}(e^z)\sim 
\frac{e^{2m_k z}}{(2m_k)!}
\]
and thus, by Stirling's formula,
\[
\log P_{m_k}(e^z)
= 2m_k z- \log ((2m_k)!) +o(1)
\sim  2m_k z
\quad\text{as } z\to\infty\text{ in }H^+_{\varepsilon_1}.
\]
It follows that 
\[
\log g_{m_k}(z)
= (1+o(1))2m_k z+ e^z
\sim  e^z 
\quad\text{as } z\to\infty\text{ in }H^+_{\varepsilon_1}.
\]
With $q(z)$ as in~\eqref{defq} we have
\begin{equation} \label{asyU}
\log U(z)=\log u_{m_k}(q(z))
=\log g_{m_k}(q(z)+s_{m_k})
\end{equation} 
and $q(z)+s_{m_k}\sim z$ as $z\to\infty$ in $H^+_{\varepsilon_2}$ by~\eqref{smk}.
Thus  $q(z)+s_{m_k}\in H^+_{\varepsilon_1}$ if $z\in H^+_{\varepsilon_2}$ and $|z|$ is sufficiently large
and we deduce from the last two equations that
\begin{equation} \label{asyU1}
\log U(z)
\sim \exp\!\left(q(z)+s_{m_k}\right)
=\exp((1+o(1))z)
\quad\text{as } z\to\infty\text{ in }H^+_{\varepsilon_2}
\end{equation} 
and thus
\begin{equation} \label{asyG}
\log G(z)
=\log U(z^\rho)
=\exp\!\left((1+o(1))z^{\rho}\right)
\quad \text{for } |\arg z|<(1-\varepsilon_2)\frac{\pi}{2\rho} .
\end{equation} 
This implies that 
\[
\begin{aligned}
\log F(z)
&=
\log G(\tau^{-1}(z))
=\exp\!\left((1+o(1))\tau^{-1}(z)^{\rho}\right)
\\ &
=\exp\!\left((1+o(1))z^{\rho}\right)
\quad \text{for } |\arg z|<(1-\varepsilon_3)\frac{\pi}{2\rho}
\end{aligned}
\]
and hence
\begin{equation} \label{asyllF}
\log\log F(z)
\sim 
z^{\rho}
\quad \text{for } |\arg z|<(1-\varepsilon_3)\frac{\pi}{2\rho}
\end{equation} 
as $|z|\to\infty$.

An asymptotic equality like~\eqref{asyllF} can be differentiated by passing to a smaller sector.
In fact, if $|\arg z|<(1-\varepsilon_4)\pi/(2\rho)$, then
$\{\zeta\colon |\zeta-z|=c|z|\}$ is contained in 
$\{\zeta\colon |\arg \zeta|<(1-\varepsilon_3)\pi/(2\rho)$ for sufficiently small~$c$,
and thus
\[
\begin{aligned}
\left|\frac{d}{dz} \!\left(\log\log F(z)-z^\rho\right)\right|
&=
\left|\frac{1}{2\pi i} \int_{|\zeta-z|=c|z|}\frac{\log\log F(\zeta)-\zeta^\rho}{(\zeta-z)^2}d\zeta\right|
\\ &
\leq
\frac{1}{c|z|} \max_{|\zeta-z|=c|z|} \left| \log\log F(\zeta)-\zeta^\rho\right| =o\!\left(|z|^{\rho-1}\right).
\end{aligned}
\]
This implies that
\[
\frac{F'(z)}{F(z)\log F(z)}
\sim 
\rho z^{\rho-1}
\quad \text{for } |\arg z|<(1-\varepsilon_4)\frac{\pi}{2\rho}
\]
and hence
\begin{align*}
E(z)
&=
\frac{F(z)}{F'(z)}
\sim\frac{z^{1-\rho}}{\rho \log F(z)}
\\ &
=\frac{z^{1-\rho}}{\rho} 
\exp\!\left(-(1+o(1))z^{\rho}\right)
\quad \text{for } |\arg z|<(1-\varepsilon_4)\frac{\pi}{2\rho}
\end{align*}
as $|z|\to\infty$.
Actually, this yields
\begin{equation} \label{asyE}
E(z)
=
\exp\!\left(-(1+o(1))z^{\rho}\right)
\quad \text{for } |\arg z|<(1-\varepsilon_4)\frac{\pi}{2\rho}
\end{equation} 
as $|z|\to\infty$.
This implies that $\log E(z)\sim -z^{\rho}$ for $|\arg z|<(1-\varepsilon_4)\pi/(2\rho)$.
Hence, differentiating this asymptotic equality as explained after~\eqref{asyllF}, we obtain
\[
\frac{E'(z)}{E(z)}=
\frac{d\log E(z)}{dz}
\sim -\rho z^{\rho-1}
\quad \text{for } |\arg z|<(1-\varepsilon_5)\frac{\pi}{2\rho}
\]
and
\[
\frac{E''(z)}{E(z)}-\left(\frac{E'(z)}{E(z)}\right)^2=
\frac{d^2\log E(z)}{dz^2}
\sim -\rho(\rho-1) z^{\rho-2}
\quad \text{for } |\arg z|<(1-\varepsilon)\frac{\pi}{2\rho}.
\]
Together with~\eqref{BL} and~\eqref{asyE} the last two formulas yield
\begin{equation} \label{asyA}
\begin{aligned}
A(z)
&\sim -\frac14 \frac{1}{E(z)^2}
\sim
-\frac14 \exp\!\left((2+o(1))z^{\rho}\right)
\\ &
\sim
\exp\!\left((2+o(1))z^{\rho}\right)
\quad \text{for } |\arg z|<(1-\varepsilon)\frac{\pi}{2\rho}.
\end{aligned}
\end{equation} 
Now~\eqref{asyAE1} follows from~\eqref{asyE} and~\eqref{asyA}.

The proof of~\eqref{asyAE2} and~\eqref{asyAE3} is similar. 
Here we use that
\begin{align*}
h_m(z)-1
&= \frac{1}{(2m)!}\int_0^z \zeta^{2m}e^{\zeta}d\zeta 
\\ &
=
\frac{1}{(2m+1)!}z^{2m+1}+ \frac{1}{(2m)!}\int_0^z \zeta^{2m}(e^{\zeta}-1)d\zeta
\\ &
= (1+\eta_m(z)) \frac{1}{(2m+1)!}z^{2m+1}
\end{align*}
where
\[
|\eta_m(z)|\leq \frac{2m+1}{|z|^{2m+1}}\left| \int_0^z \zeta^{2m}(e^{\zeta}-1)d\zeta\right|
\leq 2\frac{2m+1}{|z|^{2m+1}} \int_0^{|z|} u^{2m+1}du
\leq 2|z|
\]
for $|z|\leq 1$.
Thus, with $n=2m+1$ as before, we have
\begin{align*}
\log( v_m(z)-1) 
& 
=\log\!\left( g_m\!\left(\frac{z}{n}+s_m\right)-1\right) 
%\\ & 
=\log\!\left( h_m\!\left(e^{z/n+s_m}\right)-1\right) 
\\ &
=-\log( n!) +z+ ns_m +\log\left(1+\eta_m\left(e^{z/n+s_m}\right)\right)
\\ &
=z+ \log\left(1+\eta_m\left(e^{z/n+s_m}\right)\right)+O(n)
\end{align*}
by Lemma~\ref{lemma2}.

Similarly as before, we consider, for $j\in\{1,2\}$, the sectors
\[
H^-_{\varepsilon_j}=\left\{z\in\C\colon |\arg (-z)|<(1-\varepsilon_j)\frac{\pi}{2}\right\},
\]
For $z\in H^-_{\varepsilon_1}$ with $\im z\geq 0$ we choose $k\in\N$ with $2\pi N_{k-1}\leq \im z < 2\pi N_k$.
We can deduce from~\eqref{0d}, \eqref{0e} and Lemma~\ref{lemma2} that 
$\re z/n_{k-1}+s_{m_k}\to-\infty$ as $z\to\infty$ in $H^-_{\varepsilon_1}$.
This implies that
$e^{z/n_{k-1}+s_{m_k}}\to 0$ 
and hence $\eta_{m_k}(e^{z/n_{k-1}+s_{m_k}})\to 0$ 
as $z\to\infty$ in $H^-_{\varepsilon_1}$.
Moreover, $n_k=o(|z|)$ as $z\to\infty$ in $H^-_{\varepsilon_1}$, again by~\eqref{0e}.

It follows that $\log ( v_{m_k}(z)-1)\sim z$ 
as  $z\to\infty$ in $H^-_{\varepsilon_1}$.
Hence, with $q(z)$ defined by~\eqref{defqx}, we have
\begin{equation} \label{asyV1}
\log ( V(z) -1)=\log ( v_{m_k}(q(z))-1)\sim q(z)\sim z
\quad\text{as } z\to\infty\text{ in }H^-_{\varepsilon_2}
\end{equation}
and thus
\begin{equation} \label{asyVG}
\begin{aligned}
\log ( G(z) -1)
&=
\log ( V(-(-z)^\sigma) -1)
\\ &
\sim -(-z)^\sigma
\quad\text{for } |\arg (-z)|<(1-\varepsilon_2)\frac{\pi}{2\sigma}.
\end{aligned}
\end{equation}
This implies that
\begin{equation}\label{F1}
\begin{aligned}
\log ( F(z) -1)
&=\log\!\left(G\!\left(\tau^{-1}(z)\right)-1\right)
\sim -\!\left(-\tau^{-1}(z)\right)^\sigma
\\ &
\sim -(-z)^\sigma
\quad\text{for } |\arg (-z)|<(1-\varepsilon_3)\frac{\pi}{2\sigma}.
\end{aligned}
\end{equation}
In particular, $F(z)\to 1$ as $z\to\infty$, $|\arg (-z)|<(1-\varepsilon_3)\pi/(2\sigma)$.
Differentiating the last asymptotic equation we conclude that
\begin{equation}\label{F2}
\frac{F'(z)}{F(z) -1}\sim \sigma (-z)^{\sigma-1}
\quad\text{for } |\arg (-z)|<(1-\varepsilon_4)\frac{\pi}{2\sigma}.
\end{equation}
and thus, since 
\[
E(z)=
\frac{F(z)}{F'(z)}\sim 
\frac{1}{F'(z)}
=
\frac{F(z) -1}{F'(z)}
\frac{1}{F(z)-1},
\]
we deduce from~\eqref{F1} and~\eqref{F2} that
\begin{align*}
E(z)
&=
\frac{1}{\sigma(-z)^{\sigma-1}}\exp\!\left( (1+o(1))(-z)^\sigma  \right)
\\ &
=
\exp\!\left( (1+o(1))(-z)^\sigma  \right)
\quad\text{for } |\arg (-z)|<(1-\varepsilon_4)\frac{\pi}{2\sigma}.
\end{align*}
Thus
\[
\log E(z) \sim (-z)^\sigma 
\quad\text{for } |\arg (-z)|<(1-\varepsilon_4)\frac{\pi}{2\sigma},
\]
which implies~\eqref{asyAE2}.
As before it follows that
\[
\frac{E'(z)}{E(z)}
\sim -\sigma (-z)^{\sigma-1}
\quad\text{for } |\arg (-z)|<(1-\varepsilon_5)\frac{\pi}{2\sigma}
\]
and
\[
\frac{E''(z)}{E(z)}-\left(\frac{E'(z)}{E(z)}\right)^2
\sim \sigma(\sigma-1) (-z)^{\sigma-2}
\quad\text{for } |\arg (-z)|<(1-\varepsilon)\frac{\pi}{2\sigma}.
\]
Thus 
\[
\frac{E''(z)}{E(z)}\sim \left(\frac{E'(z)}{E(z)}\right)^2
\sim \sigma^2  (-z)^{2\sigma-2}
\quad\text{for } |\arg (-z)|<(1-\varepsilon)\frac{\pi}{2\sigma}.
\]
Now~\eqref{BL} yields
\[
A(z)\sim -\frac14\left(\frac{E'(z)}{E(z)}\right)^2
\sim
-\frac{\sigma^2}{4} (-z)^{2\sigma-2}
\quad \text{for } |\arg z|<(1-\varepsilon)\frac{\pi}{2\rho},
\]
from which~\eqref{asyAE3} immediately follows.
\end{proof}

\section{Proof of Theorem~\ref{thm2}}\label{proof2}
The main idea used in~\cite{R,Shen1985,Toda1993} is that~\eqref{BL}
implies that when $A$ is large, then $E$ is small, except possibly in the set
where $E''/E$ or $E'/E$ is large, but the latter set is small by the 
lemma on the logarithmic derivative.
We shall also use this idea, but we will need that every
unbounded component of the set
$\{z\in\C\colon |A(z)|> K|z|^p\}$ actually contains a path where $E$ tends to zero.
In order to prove this we need to show that $E$ is small on certain paths where $A$ is large
also on the exceptional set where the logarithmic derivatives are large. 
The key tool used here is an estimate of harmonic measure. 

For $a\in\C$ and $r>0$ let $D(a,r)=\{z\in\C\colon |z-a|<r\}$ be the disk
around $a$ of radius~$r$.

\begin{lemma} \label{la-th2-1}
Let $f$ be a meromorphic function of finite order, $\alpha>0,$ $\eta>0$
and $c>\rho(f)$.
Let $\{z_k\}$ be the set of zeros and poles of~$f$ in $\C\backslash\{0\}$.
Then there exists $r_0>0$ such that, for $|z|>r_0$,
\begin{equation}\label{4a1}
\left|\frac{f'(z)}{f(z)}\right|\leq  \exp\!\left((3|z|)^{\alpha}\right)
\quad\text{if }
z\notin \bigcup_k D\!\left(z_k, \exp\!\left(-|z_k|^{\alpha}\right)\right)
\end{equation}
and
\begin{equation}\label{4a2}
\left|\frac{f'(z)}{f(z)}\right|\leq  |z|^{c+\eta}
\quad\text{if }
z\notin \bigcup_k D\!\left(z_k, |z_k|^{-\eta}\right).
\end{equation}
\end{lemma}
\begin{remark}
The estimate~\eqref{4a2} is standard~\cite[p.~74]{Val}
and~\eqref{4a1} is proved 
by the same method.
\end{remark}
\begin{proof}[Proof of Lemma~\ref{la-th2-1}]
We use the estimate \cite[Chapter~3, (1.3$'$)]{GO}
\begin{equation}\label{4a}
\left|\frac{f'(z)}{f(z)}\right|\leq\frac{4RT(R,f)}{(R-|z|)^2}+2\sum_{|z_k|<R}
\frac{1}{|z-z_k|}
 \quad \text{for } |z|<R,
\end{equation}
which is obtained from the Poisson--Jensen formula
and forms the basis for the proof of the lemma on the logarithmic derivative.
We choose $R=2|z|$.  
Then the first summand on the right side of~\eqref{4a} is less than $R^{c-1}$ for 
large~$R$.
The second summand is less than $2n(R)\exp\!\left(R^{\alpha}\right)$
if $|z-z_k|\geq \exp\!\left(-|z_k|^\alpha\right)$ for all~$k$,
where $n(R)$ denotes the number of $z_k$ of modulus at most~$R$.
Since $n(R)\leq R^c$ for large $R$ we obtain
\[
\left|\frac{f'(z)}{f(z)}\right|\leq R^{c-1} 
+2R^c \exp\!\left(R^{\alpha}\right)
\]
for large~$R$, from which~\eqref{4a1} easily follows.
The estimate~\eqref{4a2} is proved analogously.
\end{proof}

We shall use the following version of the ``two constants theorem'' which is obtained from
a suitable harmonic measure estimate;
see~\cite[p.~113, Satz~4]{NevanlinnaEAF}. 
There only the case $z=0$ is stated, but the version below follows immediately.
\begin{lemma} \label{la-th2-2}
Let $G$ be a domain,
$0<\delta<1$, $0<\Theta<1$ and $R>0$.
Let $z\in G$ and
suppose that the set of all $r\in(0,R]$ for which the circle 
$\partial D(z,r)$ 
intersects the complement of $G$ has measure at least $(1-\Theta)R$.

Let $w\colon G\to\C$ be an analytic function. Suppose that $|w(\zeta)|<1$ for all $\zeta\in G$ and
that $\limsup_{\zeta\to \xi}|w(\zeta)|\leq \delta$ for all $\xi\in\partial G$ satisfying $|\xi-z|<R$.
Then
\begin{equation}\label{4b}
|w(z)|\leq \delta^M
\quad\text{with }
M= \frac{2}{\pi}\arcsin\frac{1-\Theta}{1+\Theta}.
\end{equation}
\end{lemma}
Noting that $\arcsin(1-x)\geq \pi/2-\sqrt{2x}$ we obtain
\begin{equation}\label{4c}
M\geq \frac{2}{\pi}\arcsin(1-2\Theta)\geq 1-\frac{4}{\pi}\sqrt{\Theta}
\end{equation}
in~\eqref{4b}.

The following result was proved in~\cite{Eremenko1981}.
\begin{lemma} \label{la-th2-3}
Let $f$ be a transcendental entire function, $\varepsilon>0$, $K>0$ and $p\geq 0$.
Suppose that $|f(z)|\leq K|z|^p$ for $z$ on some curve tending to $\infty$ and let
$U$ be an unbounded component of $\{z\in\C\colon |f(z)|> K|z|^p\}$.
Then there exists a curve $\gamma$ tending to $\infty$ in $U$ such that
$|f(z)|>\exp\!\left(|z|^{1/2-\varepsilon}\right)$ for $z$ in~$\gamma$.
\end{lemma}
Finally we shall use  the following result of Toda~\cite[Lemma~6]{Toda1993}.
\begin{lemma} \label{la-th2-4}
Let $A$ and $E$ be 
entire functions satisfying~\eqref{BL}.
Suppose that  $\lambda(E)<\rho(E)$. Then
$\mu(E)=\rho(E)=\mu(A)=\rho(A)$, and these numbers are equal to an integer
or~$\infty$.
\end{lemma}

\begin{proof}[Proof of Theorem~\ref{thm2}]
Let $A$ and $E$ be as in the statement of the theorem.
As noted in the introduction, we have $\mu(A)\geq N/2$
by the Denjoy--Carleman--Ahlfors Theorem.
We may also assume that $\lambda(E)<\infty$. By Lemma~\ref{la-th2-4} we then have $\rho(E)<\infty$.

Let $U_1,\dots,U_N$ be unbounded components of $\{z\in\C\colon |A(z)|> K|z|^p\}$.
Then there are curves $\sigma_1,\dots,\sigma_N$ tending to $\infty$ ``between'' these components
such that $|A(z)|=K|z|^p$ for $z\in\sigma_j$, for $j=1,\dots,N$.
We may assume that the $U_j$ and the $\sigma_j$ are numbered such that for
large $r$ there exist $\varphi_1,\dots,\varphi_N$ and $\theta_1,\dots,\theta_N$
satisfying $\varphi_1<\theta_1<\varphi_2<\dots<\varphi_N<\theta_N<\varphi_1+2\pi$
such that $re^{i\varphi_j}\in U_j$ and $re^{i\theta_j}\in \sigma_j$ for $j=1,\dots,N$. 

Let $\beta\in(0,1/2)$ and $j\in\{1,\dots,N\}$.
By Lemma~\ref{la-th2-3} 
there exists a curve $\gamma_j$ tending to $\infty$ in $U_j$ such that
\begin{equation} \label{4b1}
|A(z)|\geq \exp\!\left(|z|^{\beta}\right)
\quad\text{for } z\in\gamma_j.
\end{equation}

We denote by $\{z_k\}$ the set of zeros of $E$ and $E'$ in $\C\backslash\{0\}$,
choose $\alpha\in(0,\beta)$, and put
\[
X= \bigcup_k D\!\left(z_k, \exp\!\left(-|z_k|^{\alpha}\right)\right).
\]
Denote by $n(r)$ be the number of $z_k$ of modulus at most~$r$ and let $c>\rho(E)$.
Since $\rho(E)=\rho(E')$ we have $n(r)\leq r^c$ for large~$r$.

If one of the disks $D\!\left(z_k, \exp\!\left(-|z_k|^{\alpha}\right)\right)$ forming $X$
intersects the annulus 
\[
Q=\left\{z\in\C\colon \frac{r}{2}\leq |z|\leq 2r\right\},
\]
then $r/3<|z_k|<3r$, provided 
$r$ is sufficiently large. So there are at most $n(3r)$ such disks, and each of 
them has diameter at most $2\exp\!\left(-(r/3)^{\alpha}\right)$.
Noting that $n(3r)\leq (3r)^c$ for large $r$ we find that
the sum of the diameters of the components of $X$ that intersect $Q$
is at most $2(3r)^c\exp\!\left(-(r/3)^{\alpha}\right)$.
Since $\alpha< 1/2$ we have
\[
2(3r)^c\exp\!\left(-\left(\frac{r}{3}\right)^{\alpha}\right)
\leq \exp\!\left(-\frac{1}{2}r^\alpha\right)
\leq \frac{r}{2}
\]
for large~$r$.
This implies that for all large $r$ there exists $R\in (3r/2,2r)$ such that 
$\partial D(0,R)\subset \C\backslash X$. Moreover, $\C\backslash X$ contains
a line segment connecting $\partial D(0,r)$ and $\partial D(0,R)$. Inductively
we thus find a sequence $(R_k)$ satisfying $3R_k/2<R_{k+1}<2R_k$ and thus
tending to infinity such that  $\C\backslash X$ 
contains $\partial D(0,R_k)$ and a line segment connecting $\partial D(0,R_k)$
with $\partial D(0,R_{k+1})$, for all~$k\in\N$.

We conclude that
$\C\backslash X$ has exactly one unbounded component. We denote this component
by~$P$ and put $Y=\C\backslash P$. Then $X\subset Y$ and 
$\partial Y\subset \partial X$.
The above estimate of the sum of diameters of components of $X$ 
also shows that if $|z|$ is large enough
\begin{equation} \label{4d}
\meas\left\{t\in \left(0,\frac{|z|}{2}\right]\colon \partial D(z,t)\cap Y\neq\emptyset\right\}
\leq \exp\!\left(-\frac{1}{2}|z|^\alpha\right).
\end{equation}

Next, we deduce from~\eqref{4a1}, applied to $f=E$ and $f=E'$, that 
\begin{equation} \label{CX}
\begin{aligned}
\left|-2\frac{E''(z)}{E(z)}+\left(\frac{E'(z)}{E(z)}\right)^2\right|
&\leq  
2\left|\frac{E''(z)}{E'(z)}\right|\cdot\left|\frac{E'(z)}{E(z)}\right|+\left|\frac{E'(z)}{E(z)}\right|^2 
\\ &
\leq  
3\exp\!\left(2(3|z|)^{\alpha}\right)
\quad\text{if }z\in \C\backslash X
\end{aligned}
\end{equation}
and hence, in particular, if $z\in P$, provided $|z|$ is sufficiently large.

It follows from~\eqref{BL}, \eqref{4b1} and the last estimate 
that if $z\in\gamma_j\cap P$ and $|z|$ is large, then
\[
\frac{1}{|E(z)|^2}
\geq 4 \exp\!\left(|z|^{\beta}\right)
-3\exp\!\left(2(3|z|)^{\alpha}\right)
\geq \exp\!\left(|z|^{\beta}\right)
\]
and hence 
\begin{equation}\label{4e}
|E(z)|\leq \exp\!\left(-\frac12 |z|^{\beta}\right)
\quad\text{for } z\in\gamma_j\cap P,
\end{equation}
provided $|z|$ is sufficiently large.
We use Lemma~\ref{la-th2-2}
to estimate $E(z)$ for the points $z$ on $\gamma_j$ which are not in $\gamma_j\cap P$.
For such a point~$z$, put $R=|z|/2$
and let $G$ be the  the component of $D(z,R)\backslash (\gamma_j\cap P)$
which contains~$z$.
Since $\gamma_j$ connects $z$ with $\partial D(z,R)$,
we see that $\partial D(z,t)$ intersects $\gamma_j$ for all $t\in (0,R)$.
On the other hand, \eqref{4d} says that the set of all $t\in (0,R)$
for which $\partial D(z,t)$ intersects $Y$ has measure at most 
$\exp\!\left(-|z|^\alpha/2\right)$.
Hence the set of all $t\in (0,R)$ for which $\partial D(z,t)$ intersects $\gamma_j\backslash Y=\gamma_j\cap P$
has measure at least $R-\exp\!\left(-|z|^\alpha/2\right)$.
Since $D(z,R)\backslash G\supset D(z,R)\cap \gamma_j\cap P$
we conclude 
that $G$ satisfies the hypothesis of Lemma~\ref{la-th2-2} with 
$\Theta=\exp\!\left(-|z|^\alpha/2\right)/R$.

We choose $w(\zeta)=E(\zeta)\exp(-|z|^c)$.
Since $c>\rho(E)$ we have $|w(\zeta)|<1$ for $\zeta\in G$, provided $|z|$ is large enough.
If $\xi\in\partial G$ and $|\xi-z|<R$, then 
$\xi\in \gamma_j\cap P$.
Thus, by~\eqref{4e},
\begin{align*}
\limsup_{\zeta\to \xi}|w(\zeta)|
&=|E(\xi)|\exp\!\left(-|z|^c\right)
\leq \exp\!\left(-\frac12 |\xi|^{\beta}\right) \exp\!\left(-|z|^c\right)
\\ &
\leq \exp\!\left(-\frac12 \left(\frac{|z|}{2}\right)^\beta-|z|^c\right) 
\leq \exp\!\left(-\frac14 |z|^\beta-|z|^c\right) 
\end{align*}
for all $\xi\in\partial G$ satisfying $|\xi-z|<R$.
We may thus apply Lemma~\ref{la-th2-2} with $\delta= \exp\!\left(-|z|^\beta/4-|z|^c\right)$.
We deduce from~\eqref{4b} and~\eqref{4c} that
\[
\log|E(z)|-|z|^c
= \log|w(z)|
\leq M \log\delta
\leq -\left(1-\frac{4}{\pi}\sqrt{\Theta}\right)\left(\frac14 |z|^\beta+|z|^c\right)
\]
and hence
\[
\log|E(z)|
\leq
- \frac14 \left(1-\frac{4}{\pi}\sqrt{\Theta}\right)|z|^\beta
+ \frac{4}{\pi}\sqrt{\Theta}|z|^c.
\]
Now
\[
\sqrt{\Theta}
=\frac{\exp\!\left(-|z|^\alpha/4\right)}{\sqrt{R}}
=\frac{\sqrt{2}\exp\!\left(-|z|^\alpha/4\right)}{\sqrt{|z|}}
\leq \frac{\pi}{64} |z|^{\beta-c}
\]
and
also $\sqrt{\Theta}\leq \pi/8$ for large~$|z|$.
We conclude that $|E(z)|\leq \exp\!\left(- |z|^{\beta}/16\right)$ if $z\in\gamma_j$ but
$z\notin \gamma_j\cap P$.
By~\eqref{4e} this estimate also holds for all other $z\in\gamma_j$ of sufficiently large modulus;
that is, we have
\begin{equation}\label{4f}
|E(z)|\leq \exp\!\left(-\frac{1}{16} |z|^{\beta}\right)
\quad\text{for } z\in\gamma_j \text{ with } |z| \text{ large}.
\end{equation}

We put
\[
Z=\bigcup_k D\!\left(z_k, |z_k|^{-c}\right).
\]
We note that since $c>\rho(E)\geq \lambda(E)$ we have~\cite[Chapter~2, Theorem 1.8]{GO}
\begin{equation}\label{4h1}
\sum_k \frac{1}{|z_k|^{c}}<\infty.
\end{equation}
By~\eqref{4a2}, applied to $E$ and~$E'$, we find as in~\eqref{CX} that
\begin{equation}\label{4g}
\left|-2\frac{E''(z)}{E(z)}+\left(\frac{E'(z)}{E(z)}\right)^2\right|
\leq  3|z|^{4c}
\quad\text{for } z\notin Z.
\end{equation}
Let $M\in\N$ with $M> \max\{2c,p/2\}$.
Together with~\eqref{BL} the last equation implies that if $z\in\sigma_j\backslash Z$, then 
\[
\frac{1}{|z|^{2M}|E(z)|^2}\leq K|z|^{p-2M}+ 3|z|^{4c-2M}=o(1)
\]
as $z\to\infty$. Thus $z^ME(z)\to \infty$ as $z\to\infty$ in $\sigma_j\backslash Z$,
for $j=1,\dots, N$.
Since $z^ME(z)\to 0$ as $z\to\infty$ in $\gamma_j$ for $j=1,\dots, N$ by~\eqref{4f},
we conclude that if $K>1$ is large, 
then $\{z\in\C\colon |z^ME(z)|>K\}$ has at least $N$ unbounded components.
Let $W_1,\dots,W_N$ be such components.

It follows from~\eqref{BL} and~\eqref{4g} that if $|z^ME(z)|>K$ and $z\notin Z$, then 
\[
|A(z)|\leq  \frac34 |z|^{4c} +\frac{1}{4K^2}|z|^{2M} \leq K|z|^{2M},
\]
provided $K$ is chosen large enough. Also,
for sufficiently large $K$ 
the component $U_j$ of $\{z\in\C\colon |A(z)|> K|z|^p\}$
contains a component $V_j$ of $\{z\in\C\colon |A(z)|> K|z|^{2M}\}$.
With $V=\bigcup_{j=1}^N V_j$ and $W=\bigcup_{j=1}^N W_j$
the above argument shows that
\begin{equation}\label{4h}
V\cap W
\subset Z.
\end{equation}
For an unbounded open set $D$ and $r>0$ 
we put
\[
\theta(r,D)=\meas\{t\in [0,2\pi]\colon re^{it}\in D\}.
\]
Then
\begin{equation}\label{4i1}
\theta(r,V)+\theta(r,W)
\leq 2\pi+\theta(r,Z)
\end{equation}
by~\eqref{4h}.
By~\eqref{4h1} we have
$\theta(r,Z)\to 0$ as $r\to\infty$.
It thus follows
that
\begin{equation}\label{4i}
\theta(r,V)+\theta(r,W)
\leq 2\pi+o(1)
\end{equation}
as $r\to\infty$.

The proof is now completed by a standard application of the Ahlfors distortion theorem;
cf.~\cite[Chapter XI, \S 4, no.~267]{NevanlinnaEAF}.
Choose $r_0>1$ so large that $\partial D(0,r_0)$ intersects all $V_j$ and $W_j$.
Then, for $r>r_0$,
\begin{align*}
\log\log M(r,A)
&\geq 
\log\log \max_{\substack{|z|=r\\z\in V_j}} \left|\frac{A(z)}{Kz^{2M}}\right| 
\geq \pi\int_{r_0}^r \frac{dt}{t\theta(t,V_j)} -O(1).
\end{align*}
By the Cauchy-Schwarz inequality we have
\[
N^2=\left(\sum_{j=1}^N \frac{\sqrt{t\theta(t,V_j)}}{\sqrt{t\theta(t,V_j)}}\right)^2
\leq \sum_{j=1}^N t\theta(t,V_j) \sum_{j=1}^N \frac{1}{t\theta(t,V_j)}
= t\theta(t,V) \sum_{j=1}^N \frac{1}{t\theta(t,V_j)}
\]
and
\[
\left(\log\frac{r}{r_0}\right)^2
= \left(\int_{r_0}^r \frac{1}{\sqrt{t\theta(t,V)}} \frac{\sqrt{\theta(t,V)}}{\sqrt{t}}dt\right)^2
\leq \int_{r_0}^r \frac{dt}{{t\theta(t,V)}} \cdot  \int_{r_0}^r \frac{{\theta(t,V)}}{t} dt
\]
so that
\begin{align*}
 \log\log M(r,A)
&\geq \frac{\pi}{N} \int_{r_0}^r \sum_{j=1}^N \frac{dt}{t\theta(t,V_j)} -O(1)
\\ &
\geq N\pi\int_{r_0}^r \frac{dt}{t\theta(t,V)}-O(1)
\geq N\pi\frac{\displaystyle \left(\log \frac{r}{r_0}\right)^2}{\displaystyle \int_{r_0}^r \frac{\theta(t,V)}{t}dt}-O(1).
\end{align*}
It follows that
\[
\frac{N \log r}{\log\log M(r,A)}\leq (1+o(1))\frac{1}{\pi\log r}\int_{r_0}^r \frac{\theta(t,V)}{t}dt.
\]
Noting that $\log\log M(r,E)=\log\log M(r,z^ME(z))+o(1)$ we obtain
\[
\frac{N \log r}{\log\log M(r,E)}\leq (1+o(1))\frac{1}{\pi\log r}\int_{r_0}^r \frac{\theta(t,W)}{t}dt
\]
by the same argument. Adding the last two inequalities and using~\eqref{4i} yields 
\[
\frac{N \log r}{\log\log M(r,A)}
+
\frac{N \log r}{\log\log M(r,E)}
\leq
2 +o(1),
\]
from which we deduce that
\begin{equation}\label{4j}
\frac{N}{\mu(A)}+\frac{N}{\rho(E)}\leq 2.
\end{equation}
Since $\mu(A)<N$,
this implies that $N/\rho(E)<1$ and thus 
$\rho(E)>N>\mu(A)$. Hence $\rho(E)=\lambda(E)$ by Lemma~\ref{la-th2-4},
and substituting this in~\eqref{4j} yields the conclusion.
\end{proof}

\section{Proof of Theorem~\ref{thm3}}\label{proof3}
\subsection{Preliminaries and outline of the construction} \label{introproof3}
If $N=1$, we can take the function $A$ constructed in Theorem~\ref{thm1}, since
-- as noted already in the introduction -- it satisfies $\rho(A)=\mu(A)$. Moreover,
the set $\{z \in \C\colon |A(z)| > K|z|^p\}$ clearly has one component.
In fact, the Denjoy-Carleman-Ahlfors Theorem
implies that it has exactly one component. Thus we may restrict to the case that $N\geq 2$.

We will use the map constructed in section~\ref{defG},
but with $\rho_0=\rho/N$ and $\sigma_0=\rho_0/(2\rho_0-1)$ instead of~$\rho$ and~$\sigma$.
Note that $\rho_0\in (1/2,1)$ since $\rho\in (N/2,N)$.
We also assume that  $m_1=m_2=0$.
We denote the resulting map by~$G_0$.
We summarize the properties of $G_0$ that we need.

Let $Q$ be the homeomorphism of the right half-plane onto itself 
as defined in section~\ref{defG}.
The exact definition of $Q$ is irrelevant here but we note that $Q(z)=z$ if $\re z>1$.
We denote by $J^+$ the preimage of the half-strip 
$\{z\in\C\colon \re z\geq 0,\; 0\leq \im z\leq 2\pi\}$ under $z\mapsto Q(z^{\rho_0})$ and by
$J^-$ the preimage of $\{z\in\C\colon \re z\leq 0,\; 0\leq \im z\leq 2\pi\}$ under 
$z\mapsto -(-z)^{\sigma_0}$.  With  $s_0=\log\log 2$ we then have
\begin{equation}\label{G0J}
G_0(z)=
\begin{cases}
\exp\exp( Q(z^{\rho_0})+s_0) & \text{if } z\in J^+ ,\\
\exp\exp( (-(-z)^{\sigma_0})+s_0) & \text{if } z\in J^- .
\end{cases}
\end{equation}
We note that $J:=J^+\cup J^-$ is bounded by the real axis  and a curve in the upper half-plane
which, as the real axis, is mapped to $(1,\infty)$ by~$G_0$.

In addition to~$G_0$, we will also consider a modification $G_1$ of~$G_0$ defined as follows.
Let $U$ and $V$ be as in  section~\ref{defG}.
For $\re z>0$ we define 
\[
U_1(z)=
\begin{cases}
U(z-\pi i) & \text{if } \im z\geq\pi ,\\
\exp(-\exp(z+s_0)) & \text{if } -\pi < \im z < \pi ,\\
U(z+\pi i) & \text{if } \im z\leq -\pi ,
\end{cases}
\]
and for $\re z<0$ we define 
\[
V_1(z)=
\begin{cases}
V(z-\pi i) & \text{if } \im z\geq\pi ,\\
\exp(-\exp(z+s_0)) & \text{if } -\pi < \im z < \pi ,\\
V(z+\pi i) & \text{if } \im z\leq -\pi.
\end{cases}
\]
The map $G_1$ is then obtained by gluing $U_1$ and $V_1$ in the same way in which
$U$ and $V$ were glued to obtain~$G_0$. 
Thus we obtain
\begin{equation}\label{G1J}
G_1(z)=
\begin{cases}
\exp\!\left(-\exp(Q_1(z^{\rho_0})+s_0)\right) & \text{if } z\in J^+ ,\\
\exp\!\left(-\exp(-(-z)^{\sigma_0})+s_0)\right) & \text{if } z\in J^- ,
\end{cases}
\end{equation}
for some quasiconformal homeomorphism $Q_1$ of the right half-plane $H^+$ 
which satisfies $Q_1(\overline{z})=\overline{Q_1(z)}$ for $z\in H^+$
and $Q_1(z)=z$ for $\re z\geq 1$.
Recall that $Q$ had to be chosen to satisfy $Q(\pm iy)=\pm i g(y)=\pm iy^\gamma$
for $0\leq y\leq 2\pi$. Similarly, it is required that 
$Q_1(\pm iy)=\pm iy^\gamma$ for $0\leq y\leq 2\pi$.
This implies that $Q_1$ can be chosen such that $Q_1(z)=Q(z)$ for $|\im z|\leq 2\pi$.
It follows that  $G_1(z)=1/G_0(z)$  for $z\in J$.

For $j=1,\dots,2N$ we put
\begin{equation}\label{Sigma_j}
\Sigma_j=\left\{ z\in\C\colon (j-1)\frac{\pi}{N}<\arg z< j\frac{\pi}{N}\right\}.
\end{equation}
Let $H=\{z\in\C\colon\im z>0\}$ be the upper half-plane
and $L=H\backslash J$.
Denote by $\overline{L}$ the reflection of $L$ on the real axis; that is,
$\overline{L}=\{z\in\C\colon\overline{z}\in L\}$. Denote by $D_j$ the preimage
of $L$ or $\overline{L}$ under the map $z\mapsto z^N$ in $\Sigma_j$, for $j=1,\dots,2N$.
We will define a quasiregular map $G\colon\C\to\C$ which 
satisfies $G(z)=G_0(z^N)$ for $z\in D_j$ 
if $2\leq j\leq 2N-1$ and $G(z)=G_1(z^N)$ for $z\in D_1$ and $z\in D_{2N}$.
In the remaining part of the plane we will define $G$ by a suitable interpolation 
which will require only~\eqref{G0J} and~\eqref{G1J}.

In order to do so, we will first define $G$ in section~\ref{interpolation-infty}
in a neighborhood of~$\infty$. 
Thus, for a suitable $r_0>0$, we have to define
$G$ in certain neighborhoods of the rays $\{z\in\C\colon \arg z=j\pi /N,\, |z|>r_0\}$.
This interpolation is comparatively easy if $j\neq 1$ and
$j\neq 2N-1$ since in this case $G$ is defined by the same expression in 
the domains $D_j$ and $D_{j+1}$ adjacent to the ray. (Here we have put $D_{2N+1}=D_1$.
In similar expressions the index $j$ will also be taken modulo~$2N$.)
For $j=1$ and $j=2N-1$ the interpolation argument is more elaborate.
Next, in section~\ref{extension-bounded},
we will extend $G$ to the bounded region that remains.

Using that  $G_0$ and $G_1$ satisfy the hypothesis of the 
Teich\-m\"ul\-ler--Wit\-tich--Be\-linskii Theorem we will then show in section~\ref{estimate-dilatation}
that this is also the case for 
$G$ so that we again have~\eqref{FG} with an entire function $F$ and
a quasiconformal map $\tau\colon\C\to\C$.
As in the proof of Theorem~\ref{thm1} we will then define $E$ by $E=F/F'$ and 
$A$ by~\eqref{BL} and show in section~\ref{completion} that these functions have the required properties.

\begin{remark}
The function $G_1$ is introduced only to obtain a special Bank--Laine 
function~$E$; that is, to obtain that one of the two solutions of~\eqref{w''+Aw} whose
product is $E$ has no zeros. If we use $1/G_0$ instead of $G_1$, then both solutions
have zeros, but~\eqref{eqN} is still satisfied.
\end{remark}

\subsection{Interpolation near $\infty$} \label{interpolation-infty}
Let $\gamma_2$ be the curve forming the boundary of~$L$. We may parametrize it as
$\gamma_2\colon\R\to\C$,
\[
\gamma_2(t)=
\begin{cases}
-(-(t+2\pi i))^{1/\sigma_0}) & \text{if } t\leq 0 ,\\
Q^{-1}(t+2\pi i)^{1/\rho_0}  & \text{if } t>0 .\\
\end{cases}
\]
We shall also need the curve $\gamma_1$ in the ``middle'' between $\gamma_2$ and the real axis; 
that is, $\gamma_1\colon\R\to\C$,
\[
\gamma_1(t)=
\begin{cases}
-(-(t+i\pi))^{1/\sigma_0}) & \text{if } t\leq 0 ,\\
Q^{-1}(t+i\pi)^{1/\rho_0}  & \text{if } t>0 ,\\
\end{cases}
\]
see Figure~\ref{gamma01}.
Then $G_0$ maps $\gamma_1$ to the interval $(0,1)$ and 
$G_0(\gamma_1(t))\to 0$ as $t\to+\infty$ while $G_0(\gamma_1(t))\to 1$ as $t\to -\infty$.
\begin{figure}[!htb]
\captionsetup{width=.85\textwidth}
\centering
\begin{tikzpicture}[scale=0.8,>=latex](-10,-10)(-10,10)
\clip (-7,-0.2) rectangle (7,16);
  \def\xmin{-7}
  \def\xmax{7}
  \def\ymin{-0.2}
  \def\ymax{15.1}
\def\pi{3.14159}
\def\rho{0.75}
\def\sigma{1.5}
\draw[-] (\xmin,0) -- (\xmax,0);
 \draw[-] (0,\ymin) -- (0,\ymax);
% piece connecting left and right of gamma_1
\draw[black, samples=40, domain=0.001:1] plot ({(\x^2+(\x*\pi+(1-\x)*\pi^(2*\rho-1))^2)^(1/(2*\rho))*cos(atan((\x*\pi+(1-\x)*\pi^(2*\rho-1))/\x)/\rho)},{(\x^2+(\x*\pi+(1-\x)*\pi^(2*\rho-1))^2)^(1/(2*\rho))*sin(atan((\x*\pi+(1-\x)*\pi^(2*\rho-1))/\x)/\rho)});
% piece connecting left and right of gamma_2
\draw[black, samples=40, domain=0.001:1] plot ({(\x^2+(\x*2*\pi+(1-\x)*(2*\pi)^(2*\rho-1))^2)^(1/(2*\rho))*cos(atan((\x*2*\pi+(1-\x)*(2*\pi)^(2*\rho-1))/\x)/\rho)},{(\x^2+(\x*2*\pi+(1-\x)*(2*\pi)^(2*\rho-1))^2)^(1/(2*\rho))*sin(atan((\x*2*\pi+(1-\x)*(2*\pi)^(2*\rho-1))/\x)/\rho)});
% right part of gamma_1
\draw[black, samples=40, domain=1:0.8*\xmax] plot ({(\x^2+\pi^2)^(1/(2*\rho))*cos(atan(\pi/\x)/\rho)},{(\x^2+\pi^2)^(1/(2*\rho))*sin(atan(\pi/\x)/\rho)});
% right part of gamma_2
\draw[black, samples=40, domain=1:0.8*\xmax] plot ({(\x^2+4*\pi^2)^(1/(2*\rho))*cos(atan(2*\pi/\x)/\rho)},{(\x^2+4*\pi^2)^(1/(2*\rho))*sin(atan(2*\pi/\x)/\rho)});
% right part of gamma_0
\draw[black, samples=40, domain=1.1447:0.8*\xmax] plot ({(\x^2+(\pi-\pi*asin(\pi/exp(\x))/180)^2)^(1/(2*\rho))*cos(atan((\pi-\pi*asin(\pi/exp(\x))/180)/\x)/\rho)},{(\x^2+(\pi-\pi*asin(\pi/exp(\x))/180)^2)^(1/(2*\rho))*sin(atan((\pi-\pi*asin(\pi/exp(\x))/180)/\x)/\rho)});
% left part of gamma_1
\draw[black, samples=40, rotate=180, domain=0.001:2.9*\xmax] plot ({(\x^2+\pi^2)^(1/(2*\sigma))*cos(atan(\pi/\x)/\sigma)},{-(\x^2+\pi^2)^(1/(2*\sigma))*sin(atan(\pi/\x)/\sigma)});
% left part of gamma_2
\draw[black, samples=40, rotate=180, domain=0.001:2.9*\xmax] plot ({(\x^2+4*\pi^2)^(1/(2*\sigma))*cos(atan(2*\pi/\x)/\sigma)},{-(\x^2+4*\pi^2)^(1/(2*\sigma))*sin(atan(2*\pi/\x)/\sigma)});
% middle piece of gamma_0
\draw[black,-] (0,0) -- ({1.1447^2+(\pi-\pi*asin(\pi/exp(1.1447))/180)^2)^(1/(2*\rho))*cos(atan((\pi-\pi*asin(\pi/exp(1.1447))/180)/1.1447)/\rho)},{(1.1447^2+(\pi-\pi*asin(\pi/exp(1.1447))/180)^2)^(1/(2*\rho))*sin(atan((\pi-\pi*asin(\pi/exp(1.1447))/180)/1.1447)/\rho)});
\node at (0.2,5.7) [right] {$\gamma_1$};
\node at (0.2,13.6) [right] {$\gamma_2$};
\node at (1.,4.) [right] {$\gamma_0$};
\draw [black] plot  [smooth,tension=0.7] coordinates {(1.65,4.8) (1.4,5.2) (1.3,5.6)};
\draw [black] plot  [smooth,tension=0.7] coordinates {(-4,0) (-4.0,0.5) (-4.1,1.04) };
\node at (-6.,0.5) [right] {$\Omega_l$};
\node at (-4.1,0.5) [right] {$\gamma_l$};
\node at (-2.,0.8) [right] {$\Omega_m$};
\node at (1.4,5.4) [right] {$\Omega_r$};
\node at (0.85,5.0) [right] {$\gamma_r$};
\end{tikzpicture}
\caption{The curves $\gamma_0$, $\gamma_1$ and $\gamma_2$ for $\rho_0=3/4$ and $\sigma_0=3/2$.}
\label{gamma01}
\end{figure}

We note that if $x_l<0$ and $\gamma_l\colon [0,\pi]\to \C$, $\gamma_l(t)=-(-(x_l+it))^{1/\sigma_0}$,
then the domain $\Omega_l$ which is to the left of $\gamma_l$ and between $\gamma_1$ and the 
negative real axis is mapped 
by $z\mapsto -(-z)^{\sigma_0}$
to the half-strip
$$P=\{z\in\C\colon \re z < x_l\text{ and }0<\im z<\pi\}.$$
Choosing $x_l=\log\varepsilon-s_0$, for sufficiently small
$\varepsilon>0$, we see that $\Omega_l$ is mapped
univalently onto the half-disk
$\{z\in\C\colon |z|<\varepsilon,\; \im z>0\}$ 
by the function $z\mapsto \exp \!\left(-(-z)^{\sigma_0}+s_0\right)$
and hence  univalently onto some half-neighborhood of $1$ by~$G_0$.
Here $G_0(\gamma_l(0))=e^\varepsilon$ and
$G_0(\gamma_l(\pi))=e^{-\varepsilon}$.

Similarly, we now define a curve $\gamma_0$  ``below'' $\gamma_1$ and an arc
$\gamma_r$ connecting $\gamma_0$ and $\gamma_1$ in the domain between them
such that the domain $\Omega_r$ between  $\gamma_0$ and $\gamma_1$ and to the 
right of $\gamma_r$ is mapped  univalently onto some half-neighborhood of $0$ by~$G_0$.
In order to do so, we define $\gamma_0\colon\R\to\C$ by $\gamma_0(t)=t$ for $t\leq 0$ and 
\[
\gamma_0(t)=\left( t+i\pi -i\arcsin\!\left(\frac{\pi}{e^{t+s_0}}\right)\right)^{1/\rho_0}
\quad\text{for } t\geq 1.
\]
For $0<t<1$ we define $\gamma_0$ in such a way that the curve $\gamma_0$ is below
the curve~$\gamma_1$; see Figure~\ref{gamma01}.
We denote by $\Omega$ the domain between the curves $\gamma_0$ and $\gamma_1$.

We find that $G_0(\gamma_0(t))$ is real and negative for $t\geq 1$ and $G_0(\gamma_0(t))\to 0$ as $t\to \infty$.
Moreover, if $x_r>0$ is large and $\kappa$ is an arc connecting $\gamma_0$ and $\gamma_1$
in $\Omega\cap\{z\colon\re x>x_r\}$,
then the image of $\kappa$ under the function $z\mapsto \exp\!\left(z^{\rho_0}+s_0\right)$
 is an arc connecting the real axis with the line $\{z\in\C\colon \im z=\pi\}$.
Hence $G_0\circ \kappa$ is an arc which is contained in the intersection of the upper half-plane with a small
neighborhood of $0$ and which connects a point on the negative real axis with a point
on the positive real axis. 
In fact, for $\varepsilon>0$ sufficiently small
there exists an arc $\gamma_r$ connecting $\gamma_0$ and $\gamma_1$ such that 
$\Omega_r$ 
is mapped univalently onto the half-disk
$\{z\in\C\colon |z|<\varepsilon ,\; \im z>0\}$ by~$G_0$.
Here the common point of $\gamma_0$ and $\gamma_r$ is mapped to $-\varepsilon$
by $G_0$ while the common point of $\gamma_1$ and $\gamma_r$ is mapped to~$\varepsilon$;
cf.\ Figures~\ref{gamma01} and~\ref{boundaryA}.

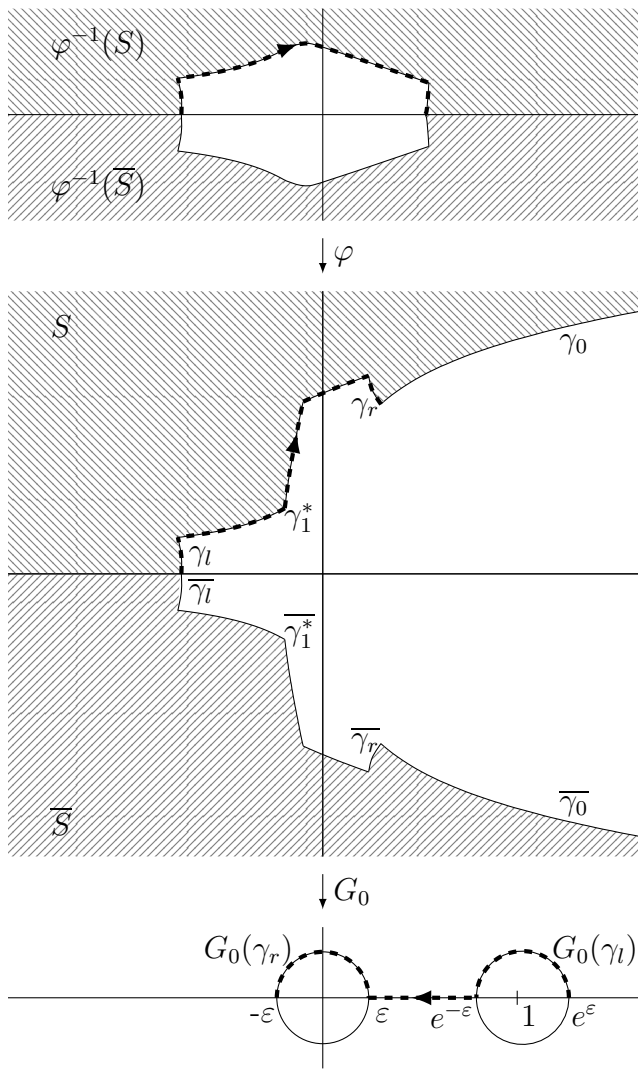
\begin{figure}[!htb]
\captionsetup{width=.85\textwidth}
\centering
\begin{tikzpicture}[scale=0.47,>=latex](-10,-15)(-10,16)
\clip (-8.9,-15) rectangle (8.9,16);
  \def\xmin{-7}
  \def\xmax{7}
  \def\ymin{-0.2}
  \def\ymax{15.1}
\def\pi{3.14159}
\def\rho{0.75}
\def\sigma{1.5}
\draw[-] (-9,0) -- (9,0);
\draw[-] (0,-8) -- (0,8);
% piece connecting left and right of gamma_1
\fill[pattern=north west lines,pattern color=black!50, samples=40, domain=0.001:1,variable=\x] (-7,8) -- plot ({(\x^2+(\x*\pi+(1-\x)*\pi^(2*\rho-1))^2)^(1/(2*\rho))*cos(atan((\x*\pi+(1-\x)*\pi^(2*\rho-1))/\x)/\rho)},{(\x^2+(\x*\pi+(1-\x)*\pi^(2*\rho-1))^2)^(1/(2*\rho))*sin(atan((\x*\pi+(1-\x)*\pi^(2*\rho-1))/\x)/\rho)}) -- (1,8) -- cycle;
\draw[black,ultra thick, dashed,samples=40, domain=0.001:1] plot ({(\x^2+(\x*\pi+(1-\x)*\pi^(2*\rho-1))^2)^(1/(2*\rho))*cos(atan((\x*\pi+(1-\x)*\pi^(2*\rho-1))/\x)/\rho)},{(\x^2+(\x*\pi+(1-\x)*\pi^(2*\rho-1))^2)^(1/(2*\rho))*sin(atan((\x*\pi+(1-\x)*\pi^(2*\rho-1))/\x)/\rho)});
\draw[black, samples=40, domain=0.001:1] plot ({(\x^2+(\x*\pi+(1-\x)*\pi^(2*\rho-1))^2)^(1/(2*\rho))*cos(atan((\x*\pi+(1-\x)*\pi^(2*\rho-1))/\x)/\rho)},{(\x^2+(\x*\pi+(1-\x)*\pi^(2*\rho-1))^2)^(1/(2*\rho))*sin(atan((\x*\pi+(1-\x)*\pi^(2*\rho-1))/\x)/\rho)});
% conjugate piece connecting left and right of gamma_1
\fill[pattern=north east lines,pattern color=black!50, samples=40, domain=0.001:1]  (-7,-8) -- plot ({(\x^2+(\x*\pi+(1-\x)*\pi^(2*\rho-1))^2)^(1/(2*\rho))*cos(atan((\x*\pi+(1-\x)*\pi^(2*\rho-1))/\x)/\rho)},{-((\x^2+(\x*\pi+(1-\x)*\pi^(2*\rho-1))^2)^(1/(2*\rho))*sin(atan((\x*\pi+(1-\x)*\pi^(2*\rho-1))/\x)/\rho))})  -- (1,-8) -- cycle;
\draw[black, samples=40, domain=0.001:1] plot ({(\x^2+(\x*\pi+(1-\x)*\pi^(2*\rho-1))^2)^(1/(2*\rho))*cos(atan((\x*\pi+(1-\x)*\pi^(2*\rho-1))/\x)/\rho)},{-((\x^2+(\x*\pi+(1-\x)*\pi^(2*\rho-1))^2)^(1/(2*\rho))*sin(atan((\x*\pi+(1-\x)*\pi^(2*\rho-1))/\x)/\rho))});
% right part of gamma_1
\fill[pattern=north west lines,pattern color=black!50, samples=40, domain=1:0.284*\xmax,variable=\x]  (1.0,8) -- plot ({(\x^2+\pi^2)^(1/(2*\rho))*cos(atan(\pi/\x)/\rho)},{(\x^2+\pi^2)^(1/(2*\rho))*sin(atan(\pi/\x)/\rho)}) -- (3,8) -- cycle;
\draw[black,ultra thick, dashed, samples=40, domain=1:0.284*\xmax] plot ({(\x^2+\pi^2)^(1/(2*\rho))*cos(atan(\pi/\x)/\rho)},{(\x^2+\pi^2)^(1/(2*\rho))*sin(atan(\pi/\x)/\rho)});
\draw[black, samples=40, domain=1:0.284*\xmax] plot ({(\x^2+\pi^2)^(1/(2*\rho))*cos(atan(\pi/\x)/\rho)},{(\x^2+\pi^2)^(1/(2*\rho))*sin(atan(\pi/\x)/\rho)});
% conjugate right part of gamma_1
\fill[pattern=north east lines,pattern color=black!50, samples=40, domain=1:0.284*\xmax]  (1.0,-8) -- plot ({(\x^2+\pi^2)^(1/(2*\rho))*cos(atan(\pi/\x)/\rho)},{-(\x^2+\pi^2)^(1/(2*\rho))*sin(atan(\pi/\x)/\rho)})  -- (3,-8) -- cycle;
\draw[black, samples=40, domain=1:0.284*\xmax] plot ({(\x^2+\pi^2)^(1/(2*\rho))*cos(atan(\pi/\x)/\rho)},{-(\x^2+\pi^2)^(1/(2*\rho))*sin(atan(\pi/\x)/\rho)});
% right part of gamma_0
\fill[pattern=north west lines,pattern color=black!50, samples=40, domain=2.0247:0.8*\xmax]  (5.0,8) -- plot ({(\x^2+(\pi-\pi*asin(\pi/exp(\x))/180)^2)^(1/(2*\rho))*cos(atan((\pi-\pi*asin(\pi/exp(\x))/180)/\x)/\rho)},{(\x^2+(\pi-\pi*asin(\pi/exp(\x))/180)^2)^(1/(2*\rho))*sin(atan((\pi-\pi*asin(\pi/exp(\x))/180)/\x)/\rho)})  -- (10,8) -- cycle;
\draw[black, samples=40, domain=2.0247:0.8*\xmax] plot ({(\x^2+(\pi-\pi*asin(\pi/exp(\x))/180)^2)^(1/(2*\rho))*cos(atan((\pi-\pi*asin(\pi/exp(\x))/180)/\x)/\rho)},{(\x^2+(\pi-\pi*asin(\pi/exp(\x))/180)^2)^(1/(2*\rho))*sin(atan((\pi-\pi*asin(\pi/exp(\x))/180)/\x)/\rho)});
% conjugate right part of gamma_0
\fill[pattern=north east lines,pattern color=black!50, samples=40, domain=2.0247:0.8*\xmax]  (5.0,-8) -- plot ({(\x^2+(\pi-\pi*asin(\pi/exp(\x))/180)^2)^(1/(2*\rho))*cos(atan((\pi-\pi*asin(\pi/exp(\x))/180)/\x)/\rho)},{-(\x^2+(\pi-\pi*asin(\pi/exp(\x))/180)^2)^(1/(2*\rho))*sin(atan((\pi-\pi*asin(\pi/exp(\x))/180)/\x)/\rho)})  -- (10,-8) -- cycle;
\draw[black, samples=40, domain=2.0247:0.8*\xmax] plot ({(\x^2+(\pi-\pi*asin(\pi/exp(\x))/180)^2)^(1/(2*\rho))*cos(atan((\pi-\pi*asin(\pi/exp(\x))/180)/\x)/\rho)},{-(\x^2+(\pi-\pi*asin(\pi/exp(\x))/180)^2)^(1/(2*\rho))*sin(atan((\pi-\pi*asin(\pi/exp(\x))/180)/\x)/\rho)});
% left part of gamma_1
\fill[pattern=north west lines,pattern color=black!50, samples=40, rotate=180, domain=0.001:1.16*\xmax,variable=\x] (7.0,-8) -- plot ({(\x^2+\pi^2)^(1/(2*\sigma))*cos(atan(\pi/\x)/\sigma)},{-(\x^2+\pi^2)^(1/(2*\sigma))*sin(atan(\pi/\x)/\sigma)})  -- (10,-8) -- cycle;
\draw[black, samples=40, rotate=180, domain=0.001:1.16*\xmax] plot ({(\x^2+\pi^2)^(1/(2*\sigma))*cos(atan(\pi/\x)/\sigma)},{-(\x^2+\pi^2)^(1/(2*\sigma))*sin(atan(\pi/\x)/\sigma)});
\draw[black,ultra thick,dashed, samples=40, rotate=180, domain=0.001:1.16*\xmax] plot ({(\x^2+\pi^2)^(1/(2*\sigma))*cos(atan(\pi/\x)/\sigma)},{-(\x^2+\pi^2)^(1/(2*\sigma))*sin(atan(\pi/\x)/\sigma)});
% conjugate left part of gamma_1
\fill[pattern=north east lines,pattern color=black!50, samples=40, rotate=180, domain=0.001:1.16*\xmax]  (7.0,8) -- plot ({(\x^2+\pi^2)^(1/(2*\sigma))*cos(atan(\pi/\x)/\sigma)},{(\x^2+\pi^2)^(1/(2*\sigma))*sin(atan(\pi/\x)/\sigma)})  -- (10,8) -- cycle;
\draw[black, samples=40, rotate=180, domain=0.001:1.16*\xmax] plot ({(\x^2+\pi^2)^(1/(2*\sigma))*cos(atan(\pi/\x)/\sigma)},{(\x^2+\pi^2)^(1/(2*\sigma))*sin(atan(\pi/\x)/\sigma)});
\node at (-1.4,1.6) [right] {$\gamma_1^*$};
\node at (-1.4,-1.6) [right] {$\overline{\gamma_1^*}$};
% gamma_l and gamma_r
\fill [pattern=north west lines,pattern color=black!50]  (-10,0) -- plot  [smooth,tension=0.7] coordinates {(-4,0) (-4.0,0.5) (-4.1,1.04) } -- (-10,8) -- cycle;
\fill [pattern=north west lines,pattern color=black!50] (5,8) -- plot [smooth,tension=0.7] coordinates {(1.65,4.8) (1.4,5.2) (1.3,5.6)} -- (3,8) -- cycle;
\fill [pattern=north east lines,pattern color=black!50]  (-10,0) -- plot  [smooth,tension=0.7] coordinates {(-4,0) (-4.0,0.5) (-4.1,-1.04) } -- (-10,-8) -- cycle;
\fill [pattern=north east lines,pattern color=black!50] (5,-8) -- plot [smooth,tension=0.7] coordinates {(1.65,-4.8) (1.4,-5.2) (1.3,-5.6)} -- (3,-8) -- cycle;
\draw [black] plot  [smooth,tension=0.7] coordinates {(1.65,4.8) (1.4,5.2) (1.3,5.6)};
\draw [black] plot  [smooth,tension=0.7] coordinates {(-4,0) (-4.0,0.5) (-4.1,1.04) };
\draw [black,ultra thick,dashed] plot  [smooth,tension=0.7] coordinates {(1.65,4.8) (1.4,5.2) (1.3,5.6)};
\draw [black,ultra thick,dashed] plot  [smooth,tension=0.7] coordinates {(-4,0) (-4.0,0.5) (-4.1,1.04) };
\draw [black] plot  [smooth,tension=0.7] coordinates {(1.65,-4.8) (1.4,-5.2) (1.3,-5.6)};
\draw [black] plot  [smooth,tension=0.7] coordinates {(-4,0) (-4.0,-0.5) (-4.1,-1.04) };
\node at (-4.1,0.5) [right] {$\gamma_l$};
\node at (-4.1,-0.5) [right] {$\overline{\gamma_l}$};
\node at (0.5,4.7) [right] {$\gamma_r$};
\node at (0.5,-4.7) [right] {$\overline{\gamma_r}$};
\draw[-] (-9,0) -- (9,0);
\draw[-] (0,-8) -- (0,8);
% arrow for phi
 \draw[->] (0,9.5) -- (0,8.5);
\node at (0.,9) [right] {$\varphi$};
% boundary of A
% gamma_l and its conjugate
\fill [pattern=north west lines,pattern color=black!50]  (-10,13) -- plot  [smooth,tension=0.7] coordinates {(-4,13) (-4.0,13.5) (-4.1,14.04) } -- (-10,16) -- cycle;
\draw [black] plot  [smooth,tension=0.7] coordinates {(-4,13) (-4.0,13.5) (-4.1,14.04) };
\draw [black,ultra thick,dashed] plot  [smooth,tension=0.7] coordinates {(-4,13) (-4.0,13.5) (-4.1,14.04) };
\fill [pattern=north east lines,pattern color=black!50]  (-10,13) -- plot  [smooth,tension=0.7] coordinates {(-4,13) (-4.0,12.5) (-4.1,11.96) }  -- (-10,10) -- cycle;
\draw [black] plot  [smooth,tension=0.7] coordinates {(-4,13) (-4.0,12.5) (-4.1,11.96) };
% left part of gamma_1
\fill[pattern=north west lines,pattern color=black!50, samples=40, rotate=180, domain=0.001:1.16*\xmax,variable=\x] (2.0,-16) -- plot ({(\x^2+\pi^2)^(1/(2*\sigma))*cos(atan(\pi/\x)/\sigma)},{-(\x^2+\pi^2)^(1/(2*\sigma))*sin(atan(\pi/\x)/\sigma)-13})  -- (10,-16) -- cycle;
\draw[black, samples=40, rotate=180, domain=0.001:1.16*\xmax] plot ({(\x^2+\pi^2)^(1/(2*\sigma))*cos(atan(\pi/\x)/\sigma)},{-(\x^2+\pi^2)^(1/(2*\sigma))*sin(atan(\pi/\x)/\sigma)-13});
\draw[black,ultra thick,dashed, samples=40, rotate=180, domain=0.001:1.16*\xmax] plot ({(\x^2+\pi^2)^(1/(2*\sigma))*cos(atan(\pi/\x)/\sigma)},{-(\x^2+\pi^2)^(1/(2*\sigma))*sin(atan(\pi/\x)/\sigma)-13});
% conjugate left part of gamma_1
\fill[pattern=north east lines,pattern color=black!50, samples=40, rotate=180, domain=0.001:1.16*\xmax] (2.0,-10) -- plot ({(\x^2+\pi^2)^(1/(2*\sigma))*cos(atan(\pi/\x)/\sigma)},{(\x^2+\pi^2)^(1/(2*\sigma))*sin(atan(\pi/\x)/\sigma)-13})  -- (10,-10) -- cycle;
\draw[black, samples=40, rotate=180, domain=0.001:1.16*\xmax] plot ({(\x^2+\pi^2)^(1/(2*\sigma))*cos(atan(\pi/\x)/\sigma)},{(\x^2+\pi^2)^(1/(2*\sigma))*sin(atan(\pi/\x)/\sigma)-13});
% right part of boundary
\fill [pattern=north west lines,pattern color=black!50] (-2,16) -- plot  [smooth,tension=0.3] coordinates {(-1.08,14.85) (-0.3,15) (3.0,13.9) } -- (9,16) -- cycle;
\fill [pattern=north west lines,pattern color=black!50] (9,13) -- plot  [smooth,tension=0.7] coordinates {(2.9,13) (2.95,13.2) (3.0,13.9) } -- (9,16) -- cycle;
\fill [pattern=north east lines,pattern color=black!50] (-2,10) -- plot  [smooth,tension=0.3] coordinates {(-1.08,11.15) (-0.3,11) (3.0,12.1) } -- (9,10) -- cycle;
\fill [pattern=north east lines,pattern color=black!50] (9,13) -- plot  [smooth,tension=0.7] coordinates {(2.9,13) (2.95,12.8) (3.0,12.1) } -- (9,10) -- cycle;
\draw [black] plot  [smooth,tension=0.3] coordinates {(-1.08,14.85) (-0.3,15) (3.0,13.9) };
\draw [black] plot  [smooth,tension=0.7] coordinates {(2.9,13) (2.95,13.2) (3.0,13.9) };
\draw [black,ultra thick, dashed] plot  [smooth,tension=0.3] coordinates {(-1.08,14.85) (-0.3,15) (3.0,13.9) };
\draw [black] plot  [smooth,tension=0.3] coordinates {(-1.08,11.15) (-0.3,11) (3.0,12.1) };
\draw [black] plot  [smooth,tension=0.7] coordinates {(2.9,13) (2.95,12.8) (3.0,12.1) };
\draw [black,ultra thick, dashed] plot  [smooth,tension=0.7] coordinates {(2.9,13) (2.95,13.2) (3.0,13.9) };
\node at (-8,7) [right] {$S$};
\node at (-8,-7) [right] {$\overline{S}$};
\node at (-8,15) [right] {$\varphi^{-1}(S)$};
\node at (-8,11) [right] {$\varphi^{-1}(\overline{S})$};
\draw[-] (-9,13) -- (9,13);
\draw[-] (0,10) -- (0,16);
% arrow for G_0
\draw[->] (0,-8.5) -- (0,-9.5);
\node at (0.,-9) [right] {$G_0$};
% images under G_0
\draw[-] (5.5,-12.2) -- (5.5,-11.8);
\draw[-,ultra thick,dashed] (1.3,-12) -- (4.33,-12);
\draw[-] (-9,-12) -- (9,-12);
\draw[-] (0,-10) -- (0,-14);
\draw [black,ultra thick, dashed] (-1.3,-12) arc (180:0:1.3) ;
\draw[black,ultra thick,dashed, samples=40, domain=0.0:180] plot ({5.5*exp(0.236*cos(\x))*cos(0.236*sin(\x)*180/3.1415)},{-12+5.5*exp(0.236*cos(\x))*sin(0.236*sin(\x)*180/3.1415)});
\draw (1.3,-12) arc (0:360:1.3) ;
\draw[black, samples=40, domain=0.0:360] plot ({5.5*exp(0.236*cos(\x))*cos(0.236*sin(\x)*180/3.1415)},{-12+5.5*exp(0.236*cos(\x))*sin(0.236*sin(\x)*180/3.1415)});
\node at (5.3,-12.5) [right] {$1$};
\node at (7.4,-12.5) {$e^{\varepsilon}$};
\node at (3.65,-12.5) {$e^{-\varepsilon}$};
\node at (1.7,-12.5) {$\varepsilon$};
\node at (-1.7,-12.5) {-$\varepsilon$};
\node at (7.7,-10.65) {$G_0(\gamma_l)$};
\node at (-2.1,-10.65) {$G_0(\gamma_r)$};
\draw [->,ultra thick] (-1.1,14.85) -- (-0.7,15.0) ;
\draw [->,ultra thick] (-0.8,3.65) -- (-0.72,4.0) ;
\draw [->,ultra thick] (2.75,-12) -- (2.5,-12) ;
\node at (7.1,6.5) {$\gamma_0$};
\node at (7.1,-6.5) {$\overline{\gamma_0}$};
\end{tikzpicture}
\caption{The boundary of $A=\cl\!\left( \varphi^{-1}(S)\cup \varphi^{-1}(\overline{S})\right)$ and its images under $\varphi$ and $G_0\circ\varphi$. 
Here $\partial A\cap H$ and its images 
are drawn as dashed curves. The domains $S$ and $\overline{S}$ and their
preimages under $\varphi$ are shaded.}
\label{boundaryA}
\end{figure}

Let $\Omega_m=\Omega\backslash (\Omega_r\cup \Omega_l)$ be the ``middle piece'' of~$\Omega$.
Let $T$ be the domain above the curve $\gamma_0$ and put $S=T\backslash\cl(\Omega_m)$ and
$R=T\backslash \cl(\Omega_m\cup \Omega_l)=S\backslash \cl(\Omega_l)$.
Here and in the following $\cl(\cdot)$ denotes the closure of a set.
Thus we have $L\subset R\subset S\subset T\subset H$.

It is not difficult to see that there exists a quasiconformal map $\varphi\colon H\to T$ 
which satisfies $\varphi(z)=z$ for $z\in L$ as well as $\varphi(z)=z$ for $\re z<x_0$
with some $x_0<0$. Choosing $\varepsilon>0$ sufficiently small we may achieve that
$x_l<x_0$ and hence that $\varphi(z)=z$ in particular for $z\in\Omega_l$.
Moreover, the map $\varphi$ can be chosen such that $\varphi(z)\sim z$ as $z\to\infty$.
In fact, one can show that any quasiconformal map $\varphi$ with the properties 
listed before also has the last property.
The map $\varphi$ extends continuously $(-\infty,0]$ and we may assume that $\varphi(0)=0$ and hence
$\varphi((-\infty,0]=(-\infty,0]$.

We denote by $\overline{L}$, $\overline{R}$, $\overline{S}$ and $\overline{T}$ the reflections
of $L$, $R$, $S$ and $T$ on the real axis. 
The function $\varphi$ extends to the lower half-plane by putting
$\varphi(z)= \overline{\varphi\!\left(\overline{z}\right)}$ for $\im z<0$ and it maps
the lower half-plane to~$\overline{T}$.

Let $A$ be the closure of $\varphi^{-1}(S)\cup\varphi^{-1}\!\left(\overline{S}\right)$ 
and put $B=\{z\in\C\colon z^N\in A\}$; see Figures~\ref{boundaryA} and~\ref{B}.
The boundary of $B$ consists of 
a Jordan curve $\Gamma\colon [0,1]\to\partial B$ which we may assume to have positive orientation.

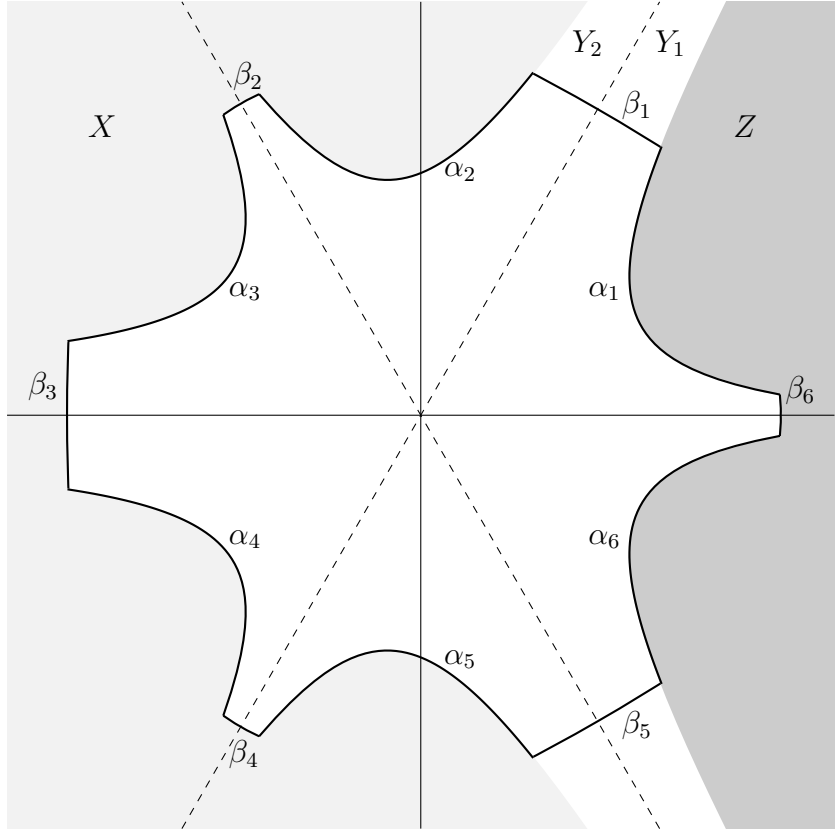
\begin{figure}[!htb]
\captionsetup{width=.85\textwidth}
\centering
\begin{tikzpicture}[scale=0.55,>=latex](-10,-10)(-10,10)
\clip (-10,-10) rectangle (10,10);
\filldraw [gray!10] plot  [smooth,tension=0.7] coordinates {(-15,0) (-13,-5) (-8.4,-4) (-8.5,-2.23) (-8.55,0) (-8.5,2.23) (-8.4,4) (-13,5) (-15,0)};
\draw [black,thick] plot  [smooth,tension=0.7] coordinates {(-15,0) (-13,-5) (-8.4,-4) (-8.5,-2.23) (-8.55,0) (-8.5,2.23) (-8.4,4) (-13,5) (-15,0)};
\draw [black,thick] plot  [rotate=120,smooth,tension=0.7] coordinates {(-15,0) (-13,-5) (-8.4,-4) (-8.5,-2.23) (-8.55,0) (-8.5,2.23) (-8.4,4) (-13,5) (-15,0)};
\draw [black,thick] plot  [rotate=240,smooth,tension=0.7] coordinates {(-15,0) (-13,-5) (-8.4,-4) (-8.5,-2.23) (-8.55,0) (-8.5,2.23) (-8.4,4) (-13,5) (-15,0)};
\filldraw [gray!40] plot  [smooth,tension=0.7] coordinates {(13,0) (8.5,-3) (8.68,-0.5) (8.7,0) (8.68,0.5) (8.5,3) };
\draw [black, thick] plot  [smooth,tension=0.7] coordinates {(13,0) (8.5,-3) (8.68,-0.5) (8.7,0) (8.68,0.5) (8.5,3) (13,0) };
\filldraw [gray!10] plot  [rotate=120, smooth,tension=0.7] coordinates {(13,0) (8.5,-3) (8.68,-0.5) (8.7,0) (8.68,0.5) (8.5,3) };
\draw [black, thick] plot  [rotate=120, smooth,tension=0.7] coordinates {(13,0) (8.5,-3) (8.68,-0.5) (8.7,0) (8.68,0.5) (8.5,3) (13,0) };
\filldraw [gray!10] plot  [rotate=240, smooth,tension=0.7] coordinates {(13,0) (8.5,-3) (8.68,-0.5) (8.7,0) (8.68,0.5) (8.5,3) };
\draw [black, thick] plot  [rotate=240, smooth,tension=0.7] coordinates {(13,0) (8.5,-3) (8.68,-0.5) (8.7,0) (8.68,0.5) (8.5,3) (13,0) };
\filldraw[gray!40, thick, samples=50, domain=0:28] plot ({0.7+0.57735*\x+10/(\x+0.8)},{\x});
\filldraw[gray!40, thick, samples=50, domain=0:28] plot ({0.7+0.57735*\x+10/(\x+0.8)},{-\x});
\filldraw[gray!10, thick, rotate=120, samples=50, domain=0:28] plot ({0.7+0.57735*\x+10/(\x+0.8)},{\x});
\filldraw[gray!10, thick, rotate=120, samples=50, domain=0:28] plot ({0.7+0.57735*\x+10/(\x+0.8)},{-\x});
\filldraw[gray!10, thick, rotate=240, samples=50, domain=0:28] plot ({0.7+0.57735*\x+10/(\x+0.8)},{\x});
\filldraw[gray!10, thick, rotate=240, samples=50, domain=0:28] plot ({0.7+0.57735*\x+10/(\x+0.8)},{-\x});
\draw[-](-10,0)->(10,0);
\draw[-](0,-10)->(0,10);
\draw[-,dashed](-5.7735,-10)->(5.7735,10);
\draw[-,dashed](5.7735,-10)->(-5.7735,10);
\draw[black, thick, samples=50, domain=0.5:6.5] plot ({0.7+0.57735*\x+10/(\x+0.8)},{\x});
\draw[black, thick, samples=50, domain=0.5:6.5] plot ({0.7+0.57735*\x+10/(\x+0.8)},{-\x});
\draw[black, thick, rotate=120, samples=50, domain=0.5:6.5] plot ({0.7+0.57735*\x+10/(\x+0.8)},{\x});
\draw[black, thick, rotate=120, samples=50, domain=0.5:6.5] plot ({0.7+0.57735*\x+10/(\x+0.8)},{-\x});
\draw[black, thick, rotate=240, samples=50, domain=0.5:6.5] plot ({0.7+0.57735*\x+10/(\x+0.8)},{\x});
\draw[black, thick, rotate=240, samples=50, domain=0.5:6.5] plot ({0.7+0.57735*\x+10/(\x+0.8)},{-\x});
\node at (5.4,9) [right] {$Y_1$};
\node at (3.4,9) [right] {$Y_2$};
\node at (7.3,7) [right] {$Z$};
\node at (-8.3,7) [right] {$X$};
\node at (5.1,3) [left] {$\alpha_1$};
\node at (5.1,-3) [left] {$\alpha_6$};
\node at (-4.9,3) [right] {$\alpha_3$};
\node at (-4.9,-3) [right] {$\alpha_4$};
\node at (0.3,5.9) [right] {$\alpha_2$};
\node at (0.3,-5.9) [right] {$\alpha_5$};
\node at (4.6,7.5) [right] {$\beta_1$};
\node at (4.6,-7.5) [right] {$\beta_5$};
\node at (-8.5,0.7) [left] {$\beta_3$};
\node at (8.55,0.6) [right] {$\beta_6$};
\node at (-4.8,8.2) [right] {$\beta_2$};
\node at (-4.9,-8.2) [right] {$\beta_4$};
\end{tikzpicture}
\caption{Sketch of $B$ for $N=3$. The domain $B$ is the exterior of the Jordan curve
$\Gamma$ formed by $\alpha_1,\dots,\alpha_6$ and $\beta_1,\dots,\beta_6$.}
\label{B}
\end{figure}

The curve $\Gamma$ splits into curves $\alpha_1,\dots,\alpha_{2N}$ and
$\beta_1,\dots,\beta_{2N}$. 
Informally, $\alpha_j$ is mapped to the upper boundary of $A$ under the map
$z\mapsto z^N$ if $j$ is odd while
it is mapped to the lower boundary of $A$ if $j$ is even.
Similarly, $\beta_j$ is mapped to the left or right boundary of~$A$ by $z\mapsto z^N$, depending on whether
$j$ is odd or even.
So
$\alpha_j$ is the part of $\Gamma$ which is contained 
in the sector $\Sigma_j$ and which by the function $z\mapsto \varphi(z^N)$ is mapped 
to a subcurve of $\gamma_1$ or its reflection on the real axis.
For a more precise description of this subcurve,
let $\gamma_1^*$ be the subcurve of $\gamma_1$ that lies between
the intersection of $\gamma_1$ with $\gamma_l$ and $\gamma_r$, let $-\gamma_1^*$ 
be the curve $\gamma_1^*$ with reversed orientation and let 
$\overline{\gamma_1^*}$ be the reflection of $\gamma_1^*$ on the real axis.
If $j$ is odd, then $\alpha_j$ is mapped bijectively to $-\gamma_1^*$ by $z\mapsto \varphi(z^N)$,
and if $j$ is even, then $\alpha_j$ is mapped bijectively to $\overline{\gamma_1^*}$ by
this function.
Thus for odd $j$ the function
$z\mapsto G_0(\varphi(z^N))$ maps $\alpha_j$ bijectively to the interval
$\left[\varepsilon,e^{-\varepsilon}\right]$, preserving the orientation.
For even $j$ it maps $\alpha_j$ to the same interval,
but reversing the orientation.
Noting that $G_1(z)=1/G_0(z)$ for $z\in J\cup \overline{J}$, and thus in particular for 
$z$ on the curves $-\gamma_1^*$ and $\overline{\gamma_1^*}$, we see that 
$z\mapsto G_1(\varphi(z^N))$ maps $\alpha_j$ to the interval $[e^{\varepsilon},1/\varepsilon]$, 
reversing the orientation for odd $j$ and preserving the orientation for even~$j$.

The curve $\beta_j$ connects $\alpha_j$ and $\alpha_{j+1}$ and is
mapped onto the concatenation of the curves $\varphi^{-1}(\gamma_l)$ and $\overline{\varphi^{-1}(\gamma_l)}$
by the map $z\mapsto z^N$ if $j$ is odd, and onto that of 
$\varphi^{-1}(\gamma_r)$ and $\overline{\varphi^{-1}(\gamma_r)}$ if $j$ is even.
We remark that $\varphi(\gamma_l)=\gamma_l$ since
$x_l<x_0$ so that $\varphi$ is the identity on the curve~$\gamma_l$.
We conclude that $z\mapsto G_0(\varphi(z^N))$ maps $\beta_j$
to a loop surrounding the point $1$ once for odd $j$. 
The loop starts at $e^{-\varepsilon}$ and passes through the point $e^\varepsilon$.
For even $j$, the image of the curve $\beta_j$ under the map $z\mapsto z^N$ intersects the 
positive real axis. The map $\varphi$ is not defined there.
In fact, the image of $\beta_j\backslash\R$ under the map $z\mapsto\varphi(z^N)$ consists
of the curves $\gamma_r$ and $\overline{\gamma_r}$, except for their endpoints on 
$\gamma_0$ and $\overline{\gamma_0}$; cf.\ Figure~\ref{boundaryA}.
However, these endpoints are both mapped to $-\varepsilon$ by $G_0$. So the image of 
$\beta_j$ under the map $z\mapsto G_0(\varphi(z^N))$ is a single curve passing through
$-\varepsilon$. Combined with previous considerations we see that 
$z\mapsto G_0(\varphi(z^N))$ maps $\beta_j$ to a circle of radius $\varepsilon$ around $0$
for even~$j$.

Using again that $G_1(z)=1/G_0(z)$ for $z\in J\cup \overline{J}$
we see that $z\mapsto G_1(\varphi(z^N))$ maps $\beta_j$
to a loop around $1$ and~$\infty$, respectively.

We will define a quasiregular map $G\colon\C\to\C$ first in $B$ and later extend it 
to the (bounded) complement of~$B$. In order to do so, put 
$B_j=B\cap \Sigma_j$ for $j=1,\dots, 2N$, let
\[
\eta\colon \bigcup_{j=1}^{2N} \Sigma_j\to T\cup \overline{T},
\quad \eta(z)=\varphi(z^N),
\]
and denote by $\eta_j$ the restriction of $\eta$ to~$\Sigma_j$.
Then $\eta_j$ maps $\Sigma_j$ univalently onto 
$T$ or $\overline{T}$, depending on whether $j$ is odd or even.
Since $\varphi(z)\sim z$ as $z\to\infty$, we have
\begin{equation}\label{asyeta}
\eta(z)\sim z^N
\end{equation}
as $z\to\infty$.

Let
\[
X= \eta_2^{-1}\!\left(\overline{R}\right) \cup \eta_{2N-1}^{-1}(R)\cup \bigcup_{j=3}^{2N-2} B_j 
\quad\text{and}\quad
Z=\eta_1^{-1}(R) \cup \eta_{2N}^{-1}\!\left(\overline{R}\right).
\]
The sets $X$ and $Z$ are shown in light and dark gray in Figure~\ref{B}.

We put $G(z)=G_0(\eta(z))$ for $z\in X$ and $G(z)=G_1(\eta(z))$ for $z\in Z$.
The map $G$ extends continuously to the parts of the rays 
$\{re^{ik\pi/N}\colon r>0\}$ that are contained in~$B$, for $k=0$ and $k=2,\dots 2N-2$.
Thus $G$ extends continuously to the closures of $X$ and~$Z$.
Note that since $\varphi(z)=z$ for $z\in L\cup\overline{L}$ we indeed have 
$G(z)=G_0(z^N)$ for $z\in D_j$ if $2\leq j\leq 2N-1$ and
$G(z)=G_1(z^N)$ for $z\in D_j$ if $j=1$ or $j= 2N$, as said in section~\ref{introproof3}.

Next we define $G$ in the remaining part of~$B$; that is,
in $B\backslash(\cl(X)\cup \cl(Z))$. As we will define a map  which 
commutes with complex conjugation, it suffices to define $G$ in the remaining part
in the upper half-plane. We put 
(cf.\ Figure~\ref{B})
\[
Y_1=\eta_1^{-1}(\Omega_l),
\quad
Y_2=\eta^{-1}\!\left(\overline{\Omega_l}\right)
\quad\text{and}\quad
Y_0=\partial Y_1\cap \partial Y_2.
\]
Hence $Y_0=\{re^{i\pi/N}\colon r>r_0\}$ for some $r_0>0$.

In order to define $G$ in $Y_1\cup Y_2\cup Y_0$,
we note that since $\varphi(z)=z$ for $z\in \Omega_l$, we have $\eta(z)=z^N$ for 
$z\in Y_1\cup Y_2\cup Y_0$.
Also, $z\mapsto -(-\eta(z))^{\sigma_0}=-(-z^N)^{\sigma_0}$ maps $Y_1$ univalently onto the half-strip $P$ and
$Y_2$ univalently onto 
\[
\overline{P}=\{z\in\C\colon \re < x_l\text{ and }-\pi<\im z<0\}.
\]
Moreover, $Y_0$ is mapped to $(-\infty,x_l]$ by this map.

We consider the quasiconformal maps 
\[
\tau_1\colon P\to \C, \quad \tau_1(x+iy)=x+i\frac{y+\pi}{2},
\]
and $\tau_2\colon \overline{P}\to \C$, $\tau_2(z)=\overline{\tau_1\!\left(\overline{z}\right)}$.
Then we define
\begin{equation}\label{GY}
G(z)=
\begin{cases}
\exp\!\left(-\exp\!\left(\tau_1\!\left(-(-z^N)^{\sigma_0}\right)+s_0\right)\right) & \text{if }z\in Y_1, \\
\exp\!\left(-i\exp\!\left(-(-z^N)^{\sigma_0}+s_0\right)\right) & \text{if }z\in Y_0, \\
\exp\!\left(\exp\!\left(\tau_2\!\left(-(-z^N)^{\sigma_0}\right)+s_0\right)\right) & \text{if }z\in Y_2. \\
\end{cases}
\end{equation}
Note that $G$ has already been defined on $\cl(X)$ and thus in particular on
$\partial X\cap\partial Y_2$, and that we have 
$G(z)=\exp\exp(-(-\eta(z))^{\sigma_0}+s_0))$ for $z\in \partial X\cap\partial Y_2$.
Since $\tau_2(x-i\pi)=x-i\pi$ for $x<x_l$
we conclude that $G$ is continuous in $\partial X\cap\partial Y_2$.
Similarly, 
$G(z)=\exp\!\left(-\exp(-(-\eta(z))^{\sigma_0}+s_0))\right)$ for $z\in \partial Z\cap\partial Y_1$
and thus $G$ is also continuous there.
Clearly $G$ also extends continuously to the remaining parts of $\partial Y_1$ and $\partial Y_2$.
Overall we have thus defined a continuous map $G$ in $B$ which is quasiregular in
the interior of~$B$.

\subsection{Extension of $G$ to the complement of $B$} \label{extension-bounded}
To extend $G$ to the complement of~$B$, recall that the boundary of $B$ is given 
by the Jordan curve $\Gamma\colon [0,1]\to\partial B$ which 
consists of the curves $\alpha_1,\dots,\alpha_{2N}$ and $\beta_1,\dots,\beta_{2N}$.

For $2\leq j\leq 2N-1$ the function $G$ maps $\alpha_j$ 
to the interval $\left[\varepsilon,e^{-\varepsilon}\right]$, preserving the orientation
for odd $j$ and reversing the orientation for even~$j$.
Moreover, $\alpha_1$ and $\alpha_{2N}$ are mapped to $\left[e^{\varepsilon},1/\varepsilon\right]$,
preserving the orientation for $\alpha_1$ and reversing the orientation for $\alpha_{2N}$.

Also, for odd $j$ satisfying $j\neq 1$ and $j\neq 2N-1$ the function $G$ maps $\beta_j$ to a loop surrounding
the point~$1$ once and for even $j\neq 2N$ it maps $\beta_j$ to a loop surrounding
the point~$0$ once. The curve $\beta_{2N}$ is mapped to a loop around $\infty$
and $\beta_1$ and $\beta_{2N-1}$ are mapped to half-loops around~$1$,
connecting the points $e^{\varepsilon}$ and~$e^{-\varepsilon}$.

To define a quasiregular map in the interior of $\Gamma$ which has this
boundary behavior we first restrict to the case that $N$ is odd.
We note that there exists a conformal map $\nu$ from
the sector $\{z\in\C\colon 0<\arg z < \pi/(2N-2),\; |z|<1\}$ onto
the lower half-plane such that the continuous extension to
the boundary maps $(0,1,e^{i\pi/(2N-2)})$ to $(\infty,1,0)$.
Reflecting $\nu$ along the sides of sectors $2N-3$ times we can extend
$\nu$ to a locally univalent map 
$\nu$ from the half-disk $\{z\in\C\colon \re z> 0,\; |z|< 1\}$ to~$\C$.
Moreover, $\nu$ extends continuously to the boundary of this half-disk.

The boundary curve of this half-disk, beginning at the origin, is mapped
-- in the following order -- to the intervals
$[\infty,1]$, $[1,0]$, $[0,1]$, $[1,0]$, $\dots$, $[1,0]$, $[0,1]$, $[1,\infty]$,
with $\nu(e^{ij\pi/(2N-2)})=1$ for even $j\in\{-(N-1),\dots,N-1\}$ while $\nu(e^{ij\pi/(2N-2)})=0$ for odd~$j\in\{-(N-1),\dots,N-1\}$.

As noted above, $G$ maps the curves $\beta_j$ to certain (small) loops around $0$, $1$ or~$\infty$,
or (small) half-loops around~$1$, depending on the value of~$j$.
We conclude that for $2\leq j\leq 2N-2$ 
there is a curve $\beta_j^*$ near $e^{i(-\pi/2 +(j-1)\pi/(2N-2)}$ such that  
$G\circ \beta_j$ and $\nu\circ \beta_j^*$ are parametrizations of the same loop around $1$ or~$0$.
Similarly, for  $j= 1$  and $j= 2N-1$, 
there exist curves $\beta_1^*$ and $\beta_{2N-1}^*$, located near the points $\pm i$, 
such that $G\circ \beta_j$ and $\nu\circ \beta_j^*$ parametrize the same half-loop around~$1$, 
and finally there exists a curve $\beta_{2N}^*$ near $0$ such that 
$G\circ \beta_{2N}$ and $\nu\circ \beta_{2N}^*$ give the same loop around~$\infty$.

Using the convention $\beta_{2N+1}^*=\beta_1^*$, 
we join $\beta_j^*$ and $\beta_{j+1}^*$ by a curve $\alpha_j^*$, which is
chosen to be part of 
the semi-circle $\{z\colon \re z\geq 0,|z|=1\}$ for $2\leq j\leq 2N-1$, and part
of the line segment $[i,-i]$ for $j=1$ and $j=2N$.
Then for $2\leq j\leq 2N-1$ the curve $\alpha_j^*$ is mapped to the interval $[\varepsilon,e^{-\varepsilon}]$ by~$\nu$,
with the orientation preserved for odd $j$ and the orientation reversed for even~$j$.
Taking into account the mapping behavior of $G$ on the curves $\alpha_j$ as noted above,
we see that for $2\leq j\leq 2N-1$ both $G\circ \alpha_j$ and $\nu\circ \alpha_j^*$ are 
parametrizations of the interval $[\varepsilon,e^{-\varepsilon}]$, with the same orientation.
Similarly we see that for 
$j=1$ and $j=2N$ both $G\circ \alpha_j$ and $\nu\circ \alpha_j^*$ are parametrizations of $[e^\varepsilon,1/\varepsilon]$,
reversing the orientation for $j=1$ and preserving it for $j=2N$.

The curves $\alpha_j^*$ and $\beta_j^*$ are shown in Figure~\ref{Delta}.
Their concatenation is a Jordan curve whose interior we denote by~$\Delta$.

\begin{figure}[!htb]
\captionsetup{width=.85\textwidth}
\centering
\begin{tikzpicture}[scale=5.,>=latex](-1.2,-1.1)(-1.2,1.1)
\draw [domain=-90:90,samples=200] plot ({cos(\x)}, {sin(\x)});
\draw[-](0,-1)->(0,1);
\filldraw[white] (0,0) circle (0.13);
\draw [domain=-90:90] plot ({0.13*cos(\x)}, {0.13*sin(\x)});
\filldraw[white] (0,1) circle (0.13);
\draw [domain=-90:-3] plot ({0.13*cos(\x)}, {1.+0.13*sin(\x)});
\filldraw[white] (0,-1) circle (0.13);
\draw [domain=3:90] plot ({0.13*cos(\x)}, {-1.+0.13*sin(\x)});
\filldraw[white] (1,0) circle (0.13);
\draw [domain=93:267] plot ({1.+0.13*cos(\x)}, {0.13*sin(\x)});
\filldraw[white] (0.707,0.707) circle (0.13);
\draw [domain=138:312] plot ({0.707+0.13*cos(\x)}, {0.707+0.13*sin(\x)});
\filldraw[white] (0.707,-0.707) circle (0.13);
\draw [domain=48:222] plot ({0.707+0.13*cos(\x)}, {-0.707+0.13*sin(\x)});
\draw[-,dashed](-0.2,0)->(1.2,0);
\draw[-,dashed](-0,-1.1)->(0,1.1);
        \node[right] at (0.11,0.08) {$\beta_6^*$};
        \node[right] at (0.05,0.85) {$\beta_5^*$};
        \node[right] at (0.05,-0.85) {$\beta_1^*$};
        \node[left] at (0.89,0.08) {$\beta_3^*$};
        \node[left] at (0.6,0.6) {$\beta_4^*$};
        \node[left] at (0.6,-0.6) {$\beta_2^*$};
        \node[right] at (0,0.5) {$\alpha_6^*$};
        \node[right] at (0,-0.5) {$\alpha_1^*$};
        \node[left] at (0.45,0.85) {$\alpha_5^*$};
        \node[left] at (0.92,0.35) {$\alpha_4^*$};
        \node[left] at (0.45,-0.85) {$\alpha_2^*$};
        \node[left] at (0.92,-0.35) {$\alpha_3^*$};
        \node at (0.45,-0.15) {$\Delta$};
\filldraw[black] (1.,0.) circle (0.01);
        \node[below right] at (1.0,0.0) {$1$};
\filldraw[black] (0.,1.) circle (0.01);
        \node[left] at (0.0,1.0) {$e^{i2\pi/4}=i$};
\filldraw[black] (0.,-1.) circle (0.01);
        \node[left] at (0.0,-1.0) {$e^{-i2\pi/4}=-i$};
\filldraw[black] (0.,0.) circle (0.01);
        \node[below left] at (0.0,0.0) {$0$};
\filldraw[black] (0.707,-0.707) circle (0.01);
        \node[right] at (0.707,-0.707) {$e^{-i\pi/4}=\frac12\sqrt{2}(1-i)$};
\filldraw[black] (0.707,0.707) circle (0.01);
        \node[right] at (0.707,0.707) {$e^{i\pi/4}=\frac12\sqrt{2}(1+i)$};
\end{tikzpicture}
\caption{Sketch of the curves $\alpha_j^*$ and $\beta_j^*$ and the domain $\Delta$ for $N=3$.}
\label{Delta}
\end{figure}
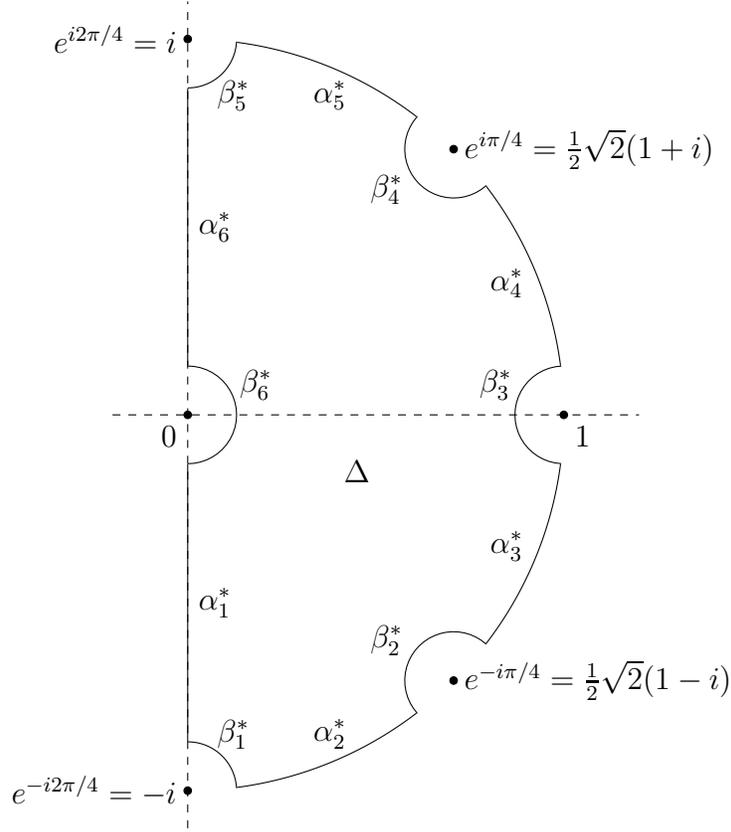

Let $\D=D(0,1)$ be the unit disk and let $\phi_1\colon \D \to\C\backslash \cl(B)$ and
$\phi_2\colon \D\to\Delta$ be conformal maps. These maps extend homeomorphically to
the closure of $\D$ and they map $\partial \D$ to $\partial B$  and $\partial \Delta$, respectively.
Since $\partial B$ is given by the curves $\alpha_j$ and $\beta_j$, 
and $\partial \Delta$ by the curves $\alpha_j^*$ and $\beta_j^*$, 
the above arguments imply that 
we may assume that 
the maps are chosen such that $G(\phi_1(1))=\nu(\phi_2(1))$,
which implies that 
$\chi_1\colon [0,2\pi] \to\C$, $\chi_1(t)=G(\phi_1(e^{it}))$
 and $\chi_2\colon [0,2\pi] \to\C$, $\chi_2(t)=\nu(\phi_2(e^{it}))$
are parametrizations of the same curve.
Thus there exists an increasing homeomorphism $h\colon [0,2\pi] \to [0,2\pi]$ 
such that $\chi_1(t)=\chi_2(h(t))$.
In fact, $h$ and $h^{-1}$ are piecewise differentiable.
Thus $h$ is quasisymmetric.
This implies (see~\cite[\S II.7]{LV}) that the 
function $\psi\colon\partial\D\to\partial \D$, $\psi(e^{it})=e^{ih(t)}$ 
extends to a quasiconformal map $\psi\colon\cl(\D)\to\cl(\D)$.

For $z\in \C\backslash\cl(B)$ we now put 
\begin{equation} \label{x21}
G(z)=\nu(\phi_2(\psi(\phi_1^{-1}(z)))).
\end{equation}
To show that the two expression agree on $\partial B$, note that a point $z\in\partial B$ 
is of the form $z=\phi_1(e^{it})$ for some $t\in [0,2\pi]$.
It follows that 
\[
\begin{aligned}
\nu(\phi_2(\psi(\phi_1^{-1}(z))))
&=\nu(\phi_2(\psi(e^{it})))=\nu(\phi_2(e^{ih(t)}))
\\ &
=\chi_2(h(t))=\chi_1(t)=G(\phi_1(e^{it}))=G(z),
\end{aligned}
\]
meaning that the two expressions in \eqref{x21} do indeed coincide for $z\in\partial B$.
So we have defined a locally quasiregular map $G\colon \C\to \C$.

The case that $N$ is even can be handled with the same method.
The only difference is that we choose the conformal map $\nu$ such that 
the sector $\{z\in\C\colon 0<\arg z < \pi/(2N-2),\; |z|<1\}$ is mapped onto
the upper half-plane, with $(0,1,e^{i\pi/(2N-2)})$ being mapped (by continuous extension) to $(\infty,0,1)$.
As before we obtain a locally univalent map 
$\nu$ from the half-disk $\{z\in\C\colon \re z> 0,\; |z|< 1\}$ to~$\C$ by repeated reflection.
The boundary curve of this half-disk, beginning at the origin, is again mapped to the intervals
$[\infty,1]$, $[1,0]$, $[0,1]$, $[1,0]$, $\dots$, $[1,0]$, $[0,1]$, $[1,\infty]$,
but this time we have
$\nu(e^{ij\pi/(2N-2)})=0$ for even $j\in\{-(N-1),\dots,N-1\}$ while $\nu(e^{ij\pi/(2N-2)})=1$ for odd~$j\in\{-(N-1),\dots,N-1\}$.
Apart from this the argument is identical.

\begin{remark}
The quasiregular map $\nu$ in the above proof was defined by an ad hoc construction.
A systematic way to construct such maps was described
by Nevanlinna \cite[no.~16]{Ne}; see also~\cite[no.~44]{Sibuya}.

A special case of the result proved there yields the
following. Let a polygon $K_p$ with $p$ sides and a
continuous function $F\colon \partial K_p\to [0,\infty]$
be given. Suppose that $F$ maps each side of $K_p$
homeomorphically onto one of the intervals $[0,1]$ and
$[1,\infty]$, with each interval occurring as
the image of at least one side. Then there exists a
locally homeomorphic extension of $F$ to $K_p$.

Nevanlinna used this to construct functions in the
class $\SS$ by gluing logarithmic ends to the sides
of this polygon. Instead, we glue restrictions of
our maps $G_0$ and $G_1$ to half-planes to the sides
of this polygon.
\end{remark}

\subsection{Estimate of the dilatation} \label{estimate-dilatation}
Let $L'$ be the preimage of $L\cup\overline{L}$ under $z\mapsto z^N$.
Since $\varphi(z)=z$ for $z\in L\cup\overline{L}$ we have
$G(z)=G_0(z^N)$ for $z\in L'\cap\Sigma_j$ if $2\leq j\leq 2N-1$ and
$G(z)=G_1(z^N)$ for $z\in L'\cap\Sigma_j$ if $j=1$ or $j=2N$.
For odd $j$ satisfying $2\leq j\leq 2N-1$ we thus have, as in sections~\ref{defG} and~\ref{est-dil},
\[
\int_{L'\cap\Sigma_j}\frac{K_G(z)-1}{x^2+y^2}dx\, dy
=\frac{1}{N^2}\int_{L}\frac{K_{G_0}(z)-1}{x^2+y^2}dx\, dy ,
\]
and for even $j$ in that range the same equation holds with $L$ replaced 
by $\overline{L}$ on the right side.
For $j=1$ and $j=2N$ these equations hold with $G_0$ replaced by~$G_1$.
Since $G_0$ and $G_1$ satisfy the hypothesis of the Teich\-m\"ul\-ler--Wit\-tich--Be\-linskii Theorem
we conclude that 
\[
\int_{L'}\frac{K_G(z)-1}{x^2+y^2}dx\, dy<\infty.
\]
Next we note that 
$\gamma_2(t)^{1/N}=(t+2\pi i)^{1/(N\rho_0)}$ for large~$t$. Since $N\rho_0>N/2\geq 1$ we thus have
$\im \gamma_2(t)\to 0$ as $t\to +\infty$.
So the distance of $\gamma_2(t)^{1/N}$ to the positive real axis tends to $0$ as $t\to +\infty$.
Similarly we see that the distance of $\gamma_2(t)^{1/N}$ to the ray $\{z\in\C\colon \arg z=\pi/N\}$
tends to $0$ as $t\to -\infty$.
We conclude that there exists $R>1$ such that $\C\backslash (D(0,R)\cup L')$ is contained 
in strips of width $1$ around the rays $\{z\in\C\colon \arg z=j\pi/N\}$, for $j=1,\dots,2N$.
Now
\[
\int_{\substack{|z|>R\\|\im z|<\frac12}} \frac{1}{x^2+y^2}dx\, dy
\leq \int_{\frac12 R}^\infty \frac{dx}{x^2}=\frac{2}{R}.
\]
Denoting by $K=\sup_{z\in \C}K_G(z)$ the dilatation of $G$ we conclude that 
\[
\int_{\C\backslash (D(0,R)\cup L')}\frac{K_G(z)-1}{x^2+y^2}dx\, dy\leq \frac{4NK}{R}<\infty.
\]
Overall we see that the quasiregular map 
$G$ satisfies the hypothesis of the Teich\-m\"ul\-ler--Wit\-tich--Be\-linskii Theorem.
Thus there exist an entire function $F$ and
a quasiconformal map $\tau\colon\C\to\C$ satisfying~\eqref{FG}.

\subsection{Completion of the proof} \label{completion}
It follows from~\eqref{GY} that $G$ has no zeros in~$Y_j$, for $0\leq j\leq 2$.
Since $G$ is symmetric with respect to the real axis, there are also no zeros 
in $\overline{Y_j}$,  for $0\leq j\leq 2$.
The construction in \S\ref{extension-bounded} shows that $G$ has no zeros in the
complement of~$B$. (In any case, since $\C\backslash B$ is bounded,
$\C\backslash B$ could contain only finitely many zeros.)
So all zeros of $G$ lie in $X\cup Z$.

Recall that $G(z)=G_0(\eta(z))$ for $z\in X$ and $G(z)=G_1(\eta(z))$ for $z\in Z$.
Next note that~\eqref{Gn} says that $n(r,0,G_0)=O(r^{\sigma_0})$ as $r\to\infty$.
The same argument shows that $n(r,0,G_1)=O(r^{\sigma_0})$  as $r\to\infty$.
Using~\eqref{asyeta} we now see that,
with $\sigma=N\sigma_0$,
\[
n(r,0,G)=O(r^{N\sigma_0}) =O(r^{\sigma}).
\]

Next, \eqref{GM} says that $\log\log |G_0(z)|\leq (1+o(1))|z|^{\rho_0}$
as $|z|\to\infty$, and the same argument yields that 
$\log\log |G_1(z)|\leq (1+o(1))|z|^{\rho_0}$.
Since $G$ is bounded in $\C\backslash \cl(X\cup Z)$, we now conclude 
from~\eqref{asyeta}
that
\[
\log\log |G(z)|\leq (1+o(1))|z|^{N\rho_0}= (1+o(1))|z|^{\rho}
\]
as $|z|\to\infty$.
As in the proof of Theorem~\ref{thm1}, we now deduce from~\eqref{FG} 
and the last two equations that
\[
N(r,0,F)=O(r^{\sigma})
\]
and
\[
\log\log |F(z)|\leq (1+o(1))|z|^{\rho},
\]
which together with the lemma on the logarithmic derivative again yields that
$E=F/F'$ satisfies
$m(r,1/E)=O(r^\rho)$ and
$N(r,0,E)=O(r^{\sigma})$
so that 
$T(r,E)=O(r^{\sigma})$.
Hence, using the well-known~\cite[Chapter~1, Theorem 7.1]{GO} inequality $\log M(r,E)\leq 3 T(2r,E)$, we see that
\begin{equation}\label{lE}
\log |E(z)|= O(|z|^\sigma)
\end{equation}
as $|z|\to\infty$.

As in the proof of Theorem~\ref{thm1}, the lemma on the logarithmic derivative 
also implies that $m(r,A)=2m(r,1/E)+O(\log r)=O(r^\rho)$
so that altogether we have 
$\lambda(E)\leq \rho(E)\leq \sigma$ and $\mu(A)\leq \rho(A)\leq \rho$.

For odd $j$ satisfying $1\leq j\leq 2N-1$ and $0<\varepsilon<1$ we put
\[
S_j(\varepsilon)=\left\{ z\in\C\colon \left|\arg z-\frac{j\pi}{N}\right|\leq (1-\varepsilon)\frac{\pi}{2\sigma_0 N}\right\}.
\]
The image of $S_j(\varepsilon)$ under the map $z\mapsto z^N$ is the 
sector $\{z\colon |\arg(-z)|<(1-\varepsilon)\pi/(2\sigma_0)\}$.
Noting that~\eqref{asyVG} holds with $G$ replaced by $G_0$ or $G_1$, and with $\sigma$ replaced by $\sigma_0$,
we deduce that $G(z)\to 1$ as $|z|\to\infty$  in $S_j(\varepsilon_1)$, for any given $\varepsilon_1\in (0,1)$.
For $0<\varepsilon_1<\varepsilon_2<\varepsilon_3<1$
we thus find that
\[
F(z)\to 1
\quad\text{as } |z|\to\infty
\text{ in } S_j(\varepsilon_1).
\]
Differentiating this as in \S \ref{AsyBeh},
we find that $|F'(z)|=o(1)$ as $z\to\infty$ in $S_j(\varepsilon_2)$ and
thus
\[
\log \frac{1}{|F'(z)|} = \log |E(z)| =  O(|z|^\sigma)
\quad\text{as } |z|\to\infty
\text{ in } S_j(\varepsilon_2)
\]
by~\eqref{lE}.
Differentiating again we conclude that 
\[
\left|\frac{F''(z)}{F'(z)}\right| = O(|z|^{\sigma-1})
\quad\text{and}\quad
\left|\frac{d}{dz}\frac{F''(z)}{F'(z)}\right| = O(|z|^{\sigma-2})
\quad\text{as } |z|\to\infty
\text{ in } S_j(\varepsilon_3).
\]
Using~\eqref{schwarz} we see that 
\begin{equation}\label{A(z)}
|A(z)| = O(|z|^{2\sigma-2})
\quad\text{as } |z|\to\infty
\text{ in } S_j(\varepsilon_3).
\end{equation}

For even $j$ we denote by $T_j$ the sector between $S_{j-1}$ and $S_{j+1}$.
Noting that~\eqref{asyG} also holds with $G$ replaced by $G_0$ or $G_1$, and with $\rho$ replaced by $\rho_0$,
we find that there exists a curve $\Gamma_j$  tending to $\infty$ in $T_j$ such that
\begin{equation}\label{M_UG}
|G(z)| \geq \exp\exp\!\left((1-o(1))|z|^\rho\right)
\quad\text{and}\quad
\arg z \to \frac{j\pi}{N}
\quad\text{as }\ z\to\infty,\ z\in\Gamma_j.
\end{equation}
Then also $|F(z)| \geq \exp\exp((1-o(1))|z|^\rho)$ as $z\to\infty$ in $\Gamma_j$.
Choosing $R$ sufficiently large we can thus deduce that $T_j$ contains  a
component $U$ of the set $\{z\in\C\colon |F(z)|>R\}$ satisfying 
\begin{equation}\label{M_U}
M_U(r):=\max_{\substack{|z|=r\\ z\in U}}|F(z)| \geq \exp\exp((1-o(1))r^\rho)
\end{equation}
as $r\to\infty$.

For large $r$ we choose $z_r\in U$ such that $|z_r|=r$ and $|F(z_r)|=M_U(r)$.
It follows from~\cite[equation (2.10)]{BRS} that with
\[
a(r)=\frac{d\log M_U(r)}{\log r}
\]
we have, for all $k\in\N$,
\[
F^{(k)}(z_r)\sim \left(\frac{a(r)}{z_r}\right)^k F(z_r)
\quad\text{as } r\to\infty,\; r\notin M,
\]
where $M\subset [1,\infty)$ is some exceptional set of finite logarithmic measure.
Now~\eqref{schwarz} yields that 
\[
A(z_r)=-\frac14 \left(\frac{a(r)}{z_r}\right)^2
\quad\text{as } r\to\infty,\; r\notin M.
\]
Noting that~\cite[equation (2.6)]{BRS} 
\[
a(r)\geq (1-o(1)) \frac{\log M_U(r)}{\log r}\geq \exp((1-o(1))r^\rho)
\]
as $r\to\infty$ we deduce
from~\eqref{A(z)} and~\eqref{M_U} that 
$\{z\in\C\colon |A(z)|>K|z|^p\}$ has at least $N$ unbounded components if 
$p=2N\rho_0/(2\rho_0-1)-2=2\sigma-2$ and $K$ is sufficiently large.

Since 
\[
\frac{N}{\rho}+\frac{N}{\sigma}=
\frac{1}{\rho_0}+\frac{1}{\sigma_0}=2,
\]
it now follows from~\eqref{ineq3} that we actually have
$\lambda(E)=\rho(E)=\sigma$ and $\mu(A)=\rho(A)=\rho$.
Hence~\eqref{eqN} holds.

\vspace{.1in}

{\em Mathematisches Seminar

Christian--Albrechts--Universit\"at zu Kiel

Ludewig--Meyn--Stra{\ss}e 4

24098 Kiel

Germany
\vspace{.1in}

Purdue University

West Lafayette, IN 47907

USA}

\end{document}